\documentclass[12pt,oneside,reqno]{amsart}
\usepackage{etex}
\usepackage{pifont}
\usepackage{soul}
\usepackage[dvips]{graphicx}
\usepackage{amsmath,latexsym,amsthm,cmap,indentfirst,setspace,ccaption,amssymb,color,amscd,mathrsfs,euscript}

\usepackage[normalem]{ulem}
\usepackage[svgnames]{xcolor}

\newcommand{\nicecolor}{Navy}
\usepackage[pdfauthor={E. Kurz \& E. Yasinsky}, backref=page, colorlinks, linktocpage, citecolor=\nicecolor, linkcolor=\nicecolor, urlcolor=\nicecolor]{hyperref}
\usepackage{array}

\usepackage{amssymb,lipsum}
\usepackage{float,xypic}
\usepackage{cmap,longtable,booktabs}
\usepackage[all,cmtip]{xy}

\usepackage{diagbox}

\usepackage{geometry}
\usepackage{caption, subcaption,xcolor}

\usepackage[shortlabels]{enumitem}
\setlist[1]{wide}
\setlist[2]{leftmargin=15mm}
\setlist[enumerate]{label=\rm{(\arabic*)}}
\setlist[enumerate,2]{label=\rm({\it\roman*}), }
\setlist[itemize]{label=\raisebox{0.25ex}{\tiny$\bullet$}}
\definecolor{grisclair}{rgb}{0.9,0.9,0.9}

\usepackage{tikz}
\usetikzlibrary{cd,arrows,positioning}
\tikzset{>=stealth}
\tikzcdset{arrow style=tikz}
\tikzset{link/.style={column sep=1.8cm,row sep=0.16cm}}
\tikzset{map/.style={row sep=0em, column sep=0em}}

\DeclareFontFamily{U}{mathb}{\hyphenchar\font45}
\DeclareFontShape{U}{mathb}{m}{n}{
	<5> <6> <7> <8> <9> <10> gen * mathb
	<10.95> mathb10 <12> <14.4> <17.28> <20.74> <24.88> mathb12
}{}
\DeclareSymbolFont{mathb}{U}{mathb}{m}{n}
\DeclareMathSymbol{\bigast}{1}{mathb}{"06}

\DeclareFontFamily{U}{mathx}{\hyphenchar\font45}
\DeclareFontShape{U}{mathx}{m}{n}{<-> mathx10}{}
\DeclareSymbolFont{mathx}{U}{mathx}{m}{n}
\DeclareMathAccent{\widebar}{0}{mathx}{"73}

\usepackage{verbatim}
\hypersetup{
    bookmarks=true,         
    unicode=true,          
    pdftoolbar=true,        
    pdfmenubar=true,        
    pdffitwindow=false,     
    pdfnewwindow=true,      
    colorlinks=true,       
    linkcolor=blue,          
    citecolor=blue,        
    filecolor=magenta,      
    urlcolor=blue           
}

\DeclareFontFamily{U}{mathx}{\hyphenchar\font45}
\DeclareFontShape{U}{mathx}{m}{n}{<-> mathx10}{}
\DeclareSymbolFont{mathx}{U}{mathx}{m}{n}
\DeclareMathAccent{\widebar}{0}{mathx}{"73}

\renewcommand{\to}{ \, \tikz[baseline=-.6ex] \draw[->,line width=.5] (0,0) -- +(.5,0); \, }
\renewcommand{\rightarrow}{ \, \tikz[baseline=-.6ex] \draw[->,line width=.5] (0,0) -- +(.5,0); \, }

\newcommand{\rat}{ \, \tikz[baseline=-.6ex] \draw[->,densely dashed,line width=.5] (0,0) -- +(.5,0); \, }

\renewcommand{\dashrightarrow}{\rat}

\renewcommand{\mapsto}{ \, \tikz[baseline=-.6ex] \draw[|->,line width=.5] (0,0) -- +(.5,0); \, }
\newcommand\iso{\stackrel{\sim}{\to}}
\newcommand{\tto}{ \, \tikz[baseline=-.6ex] \draw[->>,line width=.5] (0,0) -- +(.5,0); \, }
\renewcommand{\twoheadrightarrow}{\tto}
\newcommand{\hookto}{ \, \tikz[baseline=-.6ex] \draw[right hook->,line width=.5] (0,0) -- +(.5,0); \, }
\renewcommand{\hookrightarrow}{\hookto}

\usepackage{tikz}

\usetikzlibrary{cd,arrows,positioning,calc}

\usetikzlibrary{shapes.geometric}

\tikzset{map/.style={row sep=0em, column sep=0em}}

\usetikzlibrary{intersections, through}

\DeclareMathOperator{\Centralizer}{Cent}

\DeclareMathOperator{\Bir}{Bir}
\DeclareMathOperator{\Spec}{Spec}
\DeclareMathOperator{\rk}{rk}
\DeclareMathOperator{\Pic}{Pic}

\DeclareMathOperator{\Aut}{Aut}

\DeclareMathOperator{\PGL}{PGL}
\DeclareMathOperator{\GL}{GL}

\DeclareMathOperator{\Sym}{\mathfrak{S}}

\DeclareMathOperator{\Am}{Am}

\DeclareMathOperator{\Weil}{R}
\DeclareMathOperator{\Ind}{Ind}
\DeclareMathOperator{\Pl}{Pl}

\DeclareMathOperator{\Br}{Br}

\DeclareMathOperator{\Gal}{Gal}

\DeclareMathOperator{\Exc}{Exc}

\DeclareMathOperator{\Norm}{N}

\DeclareMathOperator{\indexx}{ind}

\DeclareMathOperator{\Cohom}{H}

\DeclareMathOperator{\Cocycles}{Z}
\DeclareMathOperator{\BirMori}{BirMori}

\DeclareMathOperator{\Fields}{\mathcal F}
\DeclareMathOperator{\Graph}{\mathcal G}

\theoremstyle{plain}
\newtheorem{thm}{Theorem}[section]
\newtheorem{lem}[thm]{Lemma}
\newtheorem{cor}[thm]{Corollary}
\newtheorem{prop}[thm]{Proposition}

\newtheorem{maintheorem}{Theorem}

\theoremstyle{definition}
\newtheorem{mydef}[thm]{Definition}
\newtheorem{rem}[thm]{Remark}

\newtheorem{ex}[thm]{Example}

\theoremstyle{definition}

\newtheorem{Notation}[thm]{Notation}

\definecolor{mypink}{RGB}{255,202,202}  
\definecolor{mygreen}{RGB}{212,255,219}  
\definecolor{myblue}{RGB}{212,234,255}  
\definecolor{myyellow}{RGB}{252,229,154}  
\definecolor{darkgreen}{RGB}{50,162,69}  

\usepackage{tikz}
\usetikzlibrary{cd,arrows,positioning}
\tikzset{>=stealth}
\tikzcdset{arrow style=tikz}
\tikzset{link/.style={column sep=1.8cm,row sep=0.16cm}}
\tikzset{map/.style={row sep=0em, column sep=0em}}

\makeatletter
\g@addto@macro{\endabstract}{\@setabstract}
\newcommand{\authorfootnotes}{\renewcommand\thefootnote{\@fnsymbol\c@footnote}}%
\makeatother

\title[Birational geometry of del Pezzo sextics]{Birational geometry of sextic del Pezzo surfaces}

\author{Elias Kurz}
\address{Elias Kurz, Universit\'e de Neuch\^{a}tel, L'Institut de math\'ematiques,
	Rue Emile-Argand 11,
	CH-2000, Neuch\^{a}tel,
	Switzerland \href{elias.kurz@unine.ch}{elias.kurz@unine.ch}}

\author{Egor Yasinsky}
\address{
	Egor Yasinsky, L'Institut de Math\'{e}matiques de Bordeaux, Universit\'{e} de Bordeaux, 351 Cours de la Lib\'{e}ration,
	33405 Talence Cedex, France
\href{egor.yasinsky@u-bordeaux.fr}{egor.yasinsky@u-bordeaux.fr}}

\subjclass[2010]{12G05, 14E07, 14E05, 14J26, 14J50, 14E30, 14J45, 14M22, 20F05}

\newcommand{\I}{\ensuremath{\mathrm{I}}}
\newcommand{\II}{\ensuremath{\mathrm{II}}}

\usepackage{bm}

\newcommand{\CC}{\mathbb C}

\newcommand{\FF}{\mathbf F}

\newcommand{\EE}{\mathbf E}
\newcommand{\id}{\mathrm{id}}

\newcommand{\LL}{\mathbf L}

\newcommand{\SplittingHex}{\mathbf F}

\newcommand{\kk}{\mathbf{k}}

\newcommand{\PP}{\mathbb P}
\newcommand{\ZZ}{\mathbb Z}
\newcommand{\KK}{\mathbf K}

\newcommand{\Dih}{\mathrm D}

\newcommand{\Klein}{\mathrm V}

\makeatletter
\def\@tocline#1#2#3#4#5#6#7{\relax
	\ifnum #1>\c@tocdepth 
	\else
	\par \addpenalty\@secpenalty\addvspace{#2}%
	\begingroup \hyphenpenalty\@M
	\@ifempty{#4}{%
		\@tempdima\csname r@tocindent\number#1\endcsname\relax
	}{%
		\@tempdima#4\relax
	}%
	\parindent\z@ \leftskip#3\relax \advance\leftskip\@tempdima\relax
	\rightskip\@pnumwidth plus4em \parfillskip-\@pnumwidth
	#5\leavevmode\hskip-\@tempdima
	\ifcase #1
	\or\or \hskip 3em \or \hskip 4em \else \hskip 5em \fi%
	#6\nobreak\relax
	\hfill\hbox to\@pnumwidth{\@tocpagenum{#7}}\par
	\nobreak
	\endgroup
	\fi}
\makeatother

\usepackage{dutchcal}

\numberwithin{equation}{section}

\begin{document}
	\raggedbottom

	\maketitle

\begin{abstract}
	We study the biregular and birational geometry of degree 6 del Pezzo surfaces with Picard number 1, defined over an arbitrary perfect field. Using Galois cohomology techniques, we obtain an explicit description of cocycles for such surfaces and describe the Severi–Brauer varieties associated with them, recovering the biregular classification of sextic del Pezzo surfaces. We then compute the automorphism groups of such surfaces, describe their closed points in general position and investigate the structure of Sarkisov links at such points and the corresponding birational models, answering a question of M. Rost. Using this description, we show that degree 6 del Pezzo surfaces are the only solid surfaces that admit infinite pliability. We also find a system of generators and relations for the groups of birational transformations of such surfaces and use it to construct nontrivial quotients of these groups, including free groups on uncountable sets.
\end{abstract}

\setcounter{tocdepth}{1}

{
	\hypersetup{linkcolor=blue}
	\tableofcontents
}

\section{Introduction}\label{sec: intro}

\subsection{Biregular and birational classifications} The main goal of this paper is to study birational properties of del Pezzo surfaces of degree~6 defined over an arbitrary perfect field $\kk$. As it is classically known, over $\overline{\kk}$, every such surface is isomorphic to the blow-up of $\PP^2$ at three non-collinear points, and thus there is only one $\overline{\kk}$-isomorphism class of such surfaces; furthermore, these are del Pezzo surfaces of the smallest anticanonical degree that possess a toric structure.

Despite their simple description over $\overline{\kk}$, the biregular classification of sextic del Pezzo surfaces~is much less trivial and, to the best of our knowledge, was first addressed in the unpublished note \cite[Section 3]{Rost} of M. Rost and the work \cite{CTKarpenkoMerkurjev} of J.-L. Colliot-Thélène, N. Karpenko and A. Merkurjev as an auxiliary result for computing the canonical dimension of $\PGL_6$. A systematic treatment was later given by M. Blunk \cite{Blunk} and M. Blunk, S. Sierra and S. Paul Smith \cite{BlunkSierraSmith}, whose studies also included Quillen K-theory and derived categories of coherent sheaves for such surfaces. From the derived category perspective, sextic del Pezzo surfaces were further studied by A. Auel and M. Bernardara \cite[Section 9]{AuelBernadara} and A. Kuznetsov \cite{KuznetsovDP6}. Models of del Pezzo fibrations of degree $\geqslant 6$ via root stacks were investigated by A. Kresch and Yu. Tschinkel in a recent series of works \cite{KreschTschinkelModels,KreschTschinkel,KreschTschinkelInvolution,KreschTschinkelDP6}, while applications to rationality problems for cubic fourfolds was given by N. Addington, B. Hassett, Yu. Tschinkel and A. Várilly-Alvarado in \cite{Addington}.

\medskip

In all these cases, the descriptions of sextic del Pezzo surfaces were obtained in rather algebraic terms, namely in terms of certain finite dimensional $\kk$-algebras. This can be seen as a continuation of the classical algebraic descriptions of other toric del Pezzo surfaces. Recall that del Pezzo surfaces of degree 9, that is, Severi–Brauer surfaces, are in one-to-one correspondence with central simple $\kk$-algebras of degree 3. Similarly, $\overline{\kk}/\kk$-forms of $\PP^1\times\PP^1$, known as involution surfaces, are in bijection with central simple $\kk$-algebras of degree 4 equipped with generalized orthogonal involutions, see \cite[\S 5B]{Knus}. The development of this viewpoint on toric varieties can be found in the works of A. Duncan \cite{DuncanTwistedForms}, M. Ballard, A. Duncan, A. Lamarche and P.~ McFaddin \cite{BallardDuncanLamarcheMcFaddin} and also in a very recent preprint of A. Duncan and P. Singh \cite{DuncanSingh}, where the authors developed a general framework for classifying torsors of algebraic tori in terms of Brauer groups.

\medskip

Recall \cite[Proposition 5.1]{CTKarpenkoMerkurjev}, \cite[Definition 2.8]{Liedtke} that for any smooth projective geometrically irreducible variety $X$ over $\kk$, there exists an exact sequence of groups
\[
0\to\Pic(X)\to\Pic(X_{\overline{\kk}})^{\Gal(\overline{\kk}/\kk)}\to\Br(\kk)\to\Br(\kk(X)).
\]
The group $\Am(X)=\Pic(X_{\overline{\kk}})^{\Gal(\overline{\kk}/\kk)}/\Pic(X)\subset\Br(\kk)$ is called the \emph{Amitsur subgroup} of $X$ in~$\Br(\kk)$. The inspiration for our first result, Theorem \ref{thm: Biregular} below, is the following characterization of certain twisted forms of quadrics, which uses the Amitsur group.

\begin{thm}[{\cite[Theorem 2.9]{Trepalin2023}}]\label{thm: Trepalin Amitsur}
	Let $S$ and $S'$ be two del Pezzo surfaces of degree 8 with $\Pic(S)\simeq\Pic(S')\simeq\ZZ$. Then $S\simeq S'$ if and only if these surfaces have the same splitting field $\FF$ and $\Am(S_\FF)=\Am(S_\FF')$. In other words, the pair $(\FF,\Am(S_\FF))$ is a biregular invariant.
\end{thm}

So, we begin by recovering the biregular classification of del Pezzo sextic surfaces focusing on the case of Picard rank 1 and using only minimal tools from Galois cohomology, namely an explicit description of cocycles; this will be useful in the remaining part of the paper. We formulate the resulting classification in terms of the Amitsur group, as was done in Theorem \ref{thm: Trepalin Amitsur} from above.

\begin{maintheorem}[cf. {\cite[Theorem 4.2]{CTKarpenkoMerkurjev}, \cite[Theorem 3.4]{Blunk}}]\label{thm: Biregular}
	Let $S$ and $S'$ be two del Pezzo surfaces of degree 6 over a perfect field $\kk$, with $\Pic(S)\simeq\Pic(S')\simeq\ZZ$. Let $\FF$ and $\FF'$ be the minimal extensions of $\kk$ over which all $(-1)$-curves on $S$ and $S'$, respectively, are defined. Then $S\simeq S'$ if and only if the following holds.
	\begin{enumerate}[leftmargin=*, labelindent=20pt, itemsep=5pt]
		\item $\FF=\FF'$.
		\item $\Am(S_\KK)=\Am(S'_\KK)$, where $\KK$ is the unique quadratic extension $\kk\subset\KK\subset\FF$ such that $\Pic(S_{\KK}) \simeq \Pic(S_{\KK}')\simeq\ZZ^2$.
		\item $\Am(S_\LL)=\Am(S'_\LL)$, where $\LL$ is the fixed subfield of the unique central element of order~$2$ of $\Gal(\FF/\kk)$, whenever such element exists (if it does not, the condition (3) is omitted).
	\end{enumerate}
\end{maintheorem}

The starting point of our approach is essentially the same as in all the aforementioned works: namely, a pair of field extensions of the base field over which certain birational contractions on a higher-degree del Pezzo surface are defined. However, we then focus primarily on this geometric data, which we refer to as \emph{Severi–Brauer data}. So, in more geometric terms and less formally, Theorem~\ref{thm: Biregular} means that two del Pezzo surfaces of degree 6 are isomorphic if and only if they have equivalent Severi–Brauer data, see Theorem~\ref{thm: Z6 iso criterion}, \ref{thm: S3 iso criterion} and \ref{thm: D6 iso criterion} for more precise statements. 

Our perspective is birational and situated within the paradigm of the Minimal Model Program, aiming to describe birational transformations of such surfaces and decide when two del Pezzo surfaces of degree 6 are birational, answering in particular  \cite[Question 2]{Rost}. The birational geometry of \emph{non-trivial} Severi–Brauer surfaces is rather poor: if two such surfaces are birational, then they are either isomorphic or opposite (that is, they correspond to opposite central simple algebras in the Brauer group $\Br(\kk)$), see \cite[Theorem 5]{Roquette}, \cite{Weinstein89}, or \cite[Lemma 1]{Weinstein22}. Surprisingly, the birational classification of del Pezzo surfaces of degree 8 was achieved only quite recently by A.~Trepalin \cite{Trepalin2023} and J.-L. Colliot-Thélène \cite{JLCTQuadrics}, with an input from J. Kollár's work \cite[7]{KollarConics}. If $S$ is a del Pezzo surface of degree 8 with $S(\kk)=\varnothing$ and $\Pic(S)\simeq\ZZ$, then \cite[Theorem 1.6]{Trepalin2023} states that any minimal surface birationally equivalent to $S$ is isomorphic to $S$. In particular, the pair $(\FF,\Am(S_\FF))$ from Theorem \ref{thm: Trepalin Amitsur} is not only a biregular but also a \emph{birational invariant}. Our next theorem shows that the birational classification of del Pezzo surfaces of degree 6 is more subtle; we will also see that the natural field extensions associated with such a surface are, in general, not birational invariants.

We use the notation of Theorem~\ref{thm: Biregular},
with the convention $\LL=\FF$ when $\Gal(\FF/\kk)\simeq\Sym_3$, and similarly for $\LL'$ and $\FF'$. In this case, $\Am(S_\LL)=0$ automatically.

\begin{maintheorem}\label{thm: Birational classification}
		Let $S$ and $S'$ be two del Pezzo surfaces of degree $6$ over a perfect field $\kk$, with $\Pic(S) \simeq \Pic(S') \simeq \ZZ$. Then $S$ and $S'$ are birational if and only if one of the following holds:		
		\begin{enumerate}[leftmargin=*, labelindent=20pt, itemsep=5pt]
			\item $\Am(S_{\KK}) =0, \, \Am(S'_{\KK'}) = 0, \, \Am(S_{\LL}) = 0, \, \Am(S'_{\LL'}) = 0$.
			\item $\LL = \LL'$, $\Am(S_{\LL}) = \Am(S'_{\LL}) \ne 0$, $\Am(S_{\KK}) = 0$ and $\Am(S'_{\KK'})=0$.
			\item $\KK = \KK'$, $\Am(S_{\KK}) = \Am(S'_{\KK}) \ne 0$,  $\Am(S_{\LL})=0$ and $\Am(S'_{\LL'})=0$.
			\item $\KK = \KK'$, $\Am(S_{\KK}) = \Am(S'_{\KK}) \neq 0$, and $\LL = \LL'$, $\Am(S_{\LL}) = \Am(S'_{\LL}) \neq 0$.
		\end{enumerate} 
\end{maintheorem}

\subsection{Pliability of del Pezzo surfaces}

The study of birational models of Mori fibre spaces and, in particular, of Fano varieties is a vast and one of the most important topics in modern birational geometry. One of the key concepts in this theory is the notion of \emph{birational rigidity}. It originates in the famous work of V. Iskovskikh and Yu. Manin who showed that a complex smooth quartic 3-fold has no birational maps to Fano varieties, other
than isomorphisms to itself. In this paper, we study birational maps between geometrically rational 2-dimensional Mori fibre spaces. These are either del Pezzo surfaces $S$ with $\Pic(S)\simeq\ZZ$, considered as fibrations over $B=\Spec\kk$, or conic bundles $\pi\colon S\to B$ with $\Pic(S)\simeq\ZZ^2$ and $B_{\overline{\kk}}\simeq\PP_{\overline{\kk}}^1$. 

\begin{mydef}\label{def: BR}
	Let $S$ be a del Pezzo surface over $\kk$ with $\Pic(S)\simeq\ZZ$. It is called \emph{birationally rigid} if for any birational map $S\dashrightarrow S'$ to the total space $S'$ of another Mori fibre space $S'\to B'$ one has $S\simeq S'$. If, moreover, every birational self-map $S\dashrightarrow S$ is actually biregular, then $S$ is called \emph{birationally superrigid}. 
	
	Finally, if $S$ is not birational to (the total space of) any conic bundle $S'$, then $S$ is called \emph{solid}. 
\end{mydef}

Following A. Corti \cite{CortiMella,CortiReid}, one can argue that the notion of birational rigidity is too restrictive and one should study varieties which are birational to only \emph{finitely many} different Mori fibre spaces, up to some natural equivalence. Let $\pi\colon S\to B$ and $\pi'\colon T'\to B'$ be two such Mori fibre spaces of dimension 2. If there is a birational map $\varphi\colon S\dashrightarrow S'$ fitting into the commutative diagram
\[\xymatrix@C+1pc{
		S\ar@{-->}[r]^{\varphi}\ar[d]_{\pi}  & S'\ar[d]^{\pi'}\\
		B\ar@{->}[r]^{\sigma} & B'  
}\]
where $\sigma$ is an isomorphism, and we have an induced isomorphism of generic fibres, then we say that our Mori fibre spaces are \emph{square equivalent}.

\begin{mydef}[cf. {\cite[Definition 1.3]{CortiMella}}]
	The \emph{pliability} of a Mori fibre space $S/B$ is the set (or, in some references, the cardinality of the set)
	\[
	\Pl(S/B)=\{S'\to B'\ \text{is a Mori fibre space}\ |\ \text{$S'$ is birational to $S$} \}\ /\ \text{square equivalence}.
	\]
\end{mydef}

So, a solid del Pezzo surface $S$ is birationally rigid if and only if $\Pl(S)=1$. If we know that our del Pezzo surface cannot be birationally transformed into fibrations with a base of positive dimension, then it is natural to ask into which (and how many) other minimal del Pezzo surfaces it can be transformed, and more generally --- whether the number of such birational models is finite. Our Theorem \ref{thm: pliability} provides an answer to this question in dimension 2, and also shows that del Pezzo surfaces of degree 6 possess, in a certain sense, an exceptional property that sets them apart from all other del Pezzo surfaces.

Let $\indexx(S)$ be the greatest common divisor of closed points on $S$.

\begin{maintheorem}\label{thm: pliability}
	Let $S$ be a birationally solid del Pezzo surface over $\kk$. Suppose that $\Pl(S)=\infty$. Then $S$ is a sextic del Pezzo surface with $\indexx(S)=2$ or $\indexx(S)=3$. Furthermore, there exist a~field $\kk$ and a solid sextic del Pezzo surface $S$ over $\kk$ such that $\Pl(S)=\infty$. 
\end{maintheorem}

\begin{rem}
	In general, it is not obvious how to construct irrational Fano varieties of Picard number 1 with large pliability, including the equivariant setting of so-called $G$-Fano varieties, see e.g. \cite{CheltsovSarikyan,CheltsovShramovIcosahedron, YasinskyRigidity} The pioneering work of A.~  Corti and M.~ Mella \cite{CortiMella} provides an example of a singular quartic threefold $X$ with $\Pl(X)=2$, the other model being a Fano threefold. The work \cite{Sarikyan} of A. Sarikyan gives an example of a Fano--Enriques threefold with pliability 9 (however, these are not solid, as they admit del Pezzo fibrations as models). As far as we know, Theorem \ref{thm: pliability} provides the first example in the literature of a solid Fano variety with infinite pliability.
\end{rem}

In Theorem~\ref{thm: birational rigidity}, we will formulate an explicit criterion for birational rigidity. A surface with infinite pliability will be constructed in Example~\ref{example: main example}. 

\subsection{Groups of birational self-maps}

There is an extensive amount of literature on groups of birational transformations of the projective plane $\PP_\kk^2$ over $\kk$, known as the plane Cremona group $\Bir(\PP_\kk^2)$; here we do not attempt to survey these results and refer the reader elsewhere. All known presentations of these groups are quite involved and probably cannot be significantly simplified, since the projective plane typically admits many different birational transformations into two-dimensional Mori fibrations (see Section \ref{sec: Sarkisov program} on the Sarkisov program). By contrast, as was noticed above, a non-trivial Severi-Brauer surface $S$ admits only two birational models: one has $\Pl(S)=2$. Despite this unpromising property, the group of its birational transformations $\Bir(S)$ enjoys some peculiar features \cite{Shramov1,Shramov2,BlancSchneiderYasinsky,Kurz}. For example, in \cite[Theorem D]{BlancSchneiderYasinsky}, it was shown that it admits a surjective homomorphism 
\[
\Bir(S)\twoheadrightarrow
\bigoplus_I\ZZ/3\ZZ\ast\left (\bigast_J\ZZ \right ),
\]
where $|I|\geqslant 1$, and both $I$ and $J$ can be infinite and even uncountable in certain cases. In particular, unlike plane Cremona groups, they are not generated by involutions and, typically, by elements of finite order; (The generation of plane Cremona groups by involutions was shown in \cite{LamySchneider}.) This is the first known example of del Pezzo surfaces enjoying such property.

In this paper, we show that a similar property holds for certain del Pezzo surfaces of degree~6.

\begin{maintheorem}\label{thm: main theorem}
	Let $S$ be a del Pezzo surface of degree 6 over a perfect field $\kk$ with $\rk\Pic(S)=1$ and $S(\kk)=\varnothing$. Then one and only one of the following holds.
	\begin{enumerate}[leftmargin=*, labelindent=15pt, itemsep=5pt]
		\item If $\indexx(S)=2$, then there is a non-trivial surjective group homomorphism
		\[
		\Bir_\kk(S)\twoheadrightarrow\left (\bigoplus_{I}\ZZ/2\ZZ \right ) \ast \left (\bigast_{J} \ZZ / 2 \ZZ \right ),
		\]
		where $I\ne\varnothing$.
		\item If $\indexx(S)=3$, then there is a non-trivial surjective group homomorphism
		\[ 
		\Bir_{\kk}(S) \twoheadrightarrow \left (\bigast_{I} \ZZ \right ) \ast \left (\bigast_{J} \ZZ / 2 \ZZ \right ),
		\] 
		where at least one of the sets $I$ and $J$ is not empty. In particular, $\Bir_\kk(S)$ is not generated by elements of finite order, once $I\ne \varnothing$. The latter condition is satisfied for every $S$ having $\Pl(S)\geqslant 3$.
		\item If $\indexx(S)=6$, then $\Bir_\kk(S)=\Aut_\kk(S)$ is an abelian group. 
	\end{enumerate}
\end{maintheorem}

Roughly speaking, the sets $I$ in the homomorphisms above correspond to closed points on $S$ that allow transitions to other birational models via Sarkisov links at those points, while the sets~$J$ correspond to points where Sarkisov links lead to an isomorphic surface. For precise definitions, we refer the reader to Section \ref{sec: quotients}.

\vspace{0.3cm}

This paper is organized as follows. In Section \ref{sec: intro}, we recall the language of Galois cohomology, as well as the basic properties of Severi--Brauer surfaces (including some explicit computations of twisted actions on them, which we were unable to find in the literature), and also the birational classification of del Pezzo surfaces of degree 8. We then briefly review the Sarkisov theory. Section~\ref{sec: biregular classification} is devoted to the classification of non-$\kk$-rational sextic del Pezzo surfaces of Picard rank 1 up to isomorphism, using an explicit description of cocycles, and so we prove Theorem \ref{thm: Biregular}. In Section~ \ref{sec: auto}, we compute the automorphism groups of such surfaces. In Section \ref{section: closed points}, we describe closed points of degrees 2 and 3 on our surfaces and study the generality of their positions. In Section~\ref{sec: birational}, we study Sarkisov links between del Pezzo surfaces of degree 6 and, as a consequence, derive Theorem~\ref{thm: pliability}. Section~\ref{sec: quotients} opens with the proof of Theorem~\ref{thm: Birational classification}. The remaining part is devoted to the proof of Theorem~\ref{thm: main theorem}. Along the way, we give the presentation of the groups $\Bir_\kk(S)$, where $S$ is a minimal sextic del Pezzo surface of index 2 or 3.

\subsection*{Acknowledgement} Our special thanks go to Jérémy Blanc for his interest in the project and constant support. We would also like to thank Ivan Cheltsov, Arman Sarikyan, and Evgeny Shinder for useful discussions. Finally, we thank Alexander Duncan and Alexey Elagin for discussions that helped strengthen Theorem \ref{thm: Birational classification}.  

The first author acknowledges support by the Swiss National Science Foundation Grant "Del Pezzo surfaces over perfect fields" (10003572).

\section{Preliminaries}\label{sec: prelim}

\subsection{Fields and Galois theory}\label{subsec: Galois theory} In this paper, we will mostly work with finite Galois extensions of our base field $\kk$, which is assumed to be perfect. For the convenience of the reader, below we recollect some basic algebraic facts which will be systematically used throughout this text. Their proofs can be found in \cite[14.4]{DummitFoote}.

\begin{lem}\label{lem: intersection of galois extensions is galois}
	Let $\EE_1$ and $\EE_2$ be two Galois extensions of a field $\kk$. Fix the algebraic closure $\overline{\kk}$ of ~$\kk$. Then the intersection $\EE_1\cap\EE_2$ in $\overline{\kk}$ is Galois over~$\kk$.
\end{lem}

Recall that, giving two Galois extensions $\EE_1$ and $\EE_2$ of a field $\kk$, their \emph{composite} $\EE_1\EE_2$ in $\overline{\kk}$ is the smallest subfield of $\overline{\kk}$ that contains both $\EE_1$ and $\EE_2$. 

\begin{lem}\label{lem: Galois group of the composite}
	Let $\EE_1$ and $\EE_2$ be two Galois extensions of a field $\kk$. The homomorphism
		\[
		\psi\colon\Gal(\EE_1\EE_2/\kk)\to \Gal(\EE_1/\kk)\times\Gal(\EE_2/\kk),\ \ g\mapsto (g|_{\EE_1},g|_{\EE_2})
		\]
		is injective and induces an isomorphism
		\[
		\psi\colon\Gal(\EE_1\EE_2/\kk)\iso\{(u,v)\in\Gal(\EE_1/\kk)\times(\EE_2/\kk)\colon u|_{\EE_1\cap\EE_2}=v|_{\EE_1\cap\EE_2}\}. 
		\]	
		In particular, $\psi$ is an isomorphism if and only if $\EE_1\cap\EE_2=\kk$.
\end{lem} 

\begin{rem}\label{rem: Galois group of the composite}
	In this paper, Lemma \ref{lem: Galois group of the composite} will be used either in the situation when $\EE_1\cap\EE_2=\kk$, so that $\Gal(\EE_1\EE_2/\kk)$ can be identified with the direct product $\Gal(\EE_1/\kk)\times\Gal(\EE_2/\kk)$, or when $\EE_1\cap\EE_2$ is a quadratic extension of $\kk$. Specifically, assume that $\Gal(\EE_2/\kk)=\langle w,t\rangle\simeq\Sym_3$, where $w$ and $t$ are generators of orders 3 and 2, respectively.
	\begin{itemize}
		\item Suppose that $\Gal(\EE_1/\kk)$ is
		\[
		\langle g,h\ |\ g^3=h^2=\id,\ gh=hg\rangle\simeq\ZZ/6\ZZ\ \ \text{or}\ \ \langle g,f\ |\ g^3=f^2=\id,\ fgf=g^2\rangle\simeq\Sym_3.
		\] 
		Then, by Lemma \ref{lem: intersection of galois extensions is galois}, one has $\EE_1\cap\EE_2=\EE_1^g=\EE_2^w$. The group $\Gal(\EE_1\EE_2/\kk)$, under the above embedding $\psi$, is thus generated by $(g,\id)$, $(\id,w)$ and $(h,t)$ for $\Gal(\EE_1/\kk)\simeq\ZZ/6\ZZ$, and by $(g,\id)$, $(\id,w)$ and $(f,t)$ for $\Gal(\EE_1/\kk)\simeq\Sym_3$.
		\item Suppose that 
		\[
		\Gal(\EE_1/\kk)=\langle g,f,h\ |\ g^3=f^2=h^2=\id,\ hf=fh,\ gh=hg,\ fgf=g^2 \rangle\simeq\Dih_6.
		\]
		Possible quadratic extensions of $\kk$ in this case are $\EE_1^{\langle u,g\rangle}=\EE_2^w$, where $u\in\{h,f,s=hf\}$.  The group $\Gal(\EE_1\EE_2/\kk)$, is generated by $(g,\id)$, $(\id,w)$, $(u,\id)$ and\footnote{In each case, the first entry of this generator can be chosen to be any element whose class in $\Dih_6/\langle g,u\rangle$ is not trivial.} 
		\[
		(f,t)\ \text{if}\ u=h,\ \ \ \ (h,t)\ \text{if}\ u=f,\ \ \ \ (h,t)\ \text{if}\ u=s.
		\]
	\end{itemize}
\end{rem}

Finally, given a finite Galois extension $\EE_1/\EE_2$, an automorphism $g\in\Gal(\EE_1/\EE_2)$ of order $n$, and an element $\lambda\in\EE_1$ we call the element
\[
\Norm_g(\lambda)=\prod_{k=0}^{n-1}g^k(\lambda)\in\EE_1^{\langle g\rangle}
\]
the \emph{$g$-norm} of $\lambda$.

\begin{thm}[Hilbert's Theorem 90, {\cite[Example 2.3.4]{Gille}}]\label{thm: Hilbert 90}
	Let $\EE_1/\EE_2$ be a Galois extension of fields with cyclic Galois group $\Gal(\EE_1/\EE_2)$ generated by an element $g$. If $\lambda\in\EE_1$ is an element of $g$-norm 1, i.e. $\Norm_g(\lambda)=1$, then there is $\mu\in\EE_1$ such that $\lambda=\mu/g(\mu)$.
\end{thm}

\subsection{Galois cohomology}\label{subsec: Galois cohomology}

We recall some basic facts about Galois cohomology and twisted forms of varieties, following \cite[I\S 5, III\S 1]{Serre}. Given a profinite group $G$ and a $G$-group $A$ (where the action will be denoted by $g(a)$ or simply $g a$ for $g\in G$, $a\in A$), one defines a \emph{1-cocycle} as a continuous map $\alpha\colon G\to A$, $g\mapsto \alpha_g$, which satisfies the condition
\[
\alpha_{g_1g_2}=\alpha_{g_1}g_1(\alpha_{g_2}),\ \ g_1,g_2\in G.
\]
The set of these 1-cocycles is denoted $\Cocycles^1(G,A)$. Two cocycles $\alpha$ and $\alpha'$ are said to be \emph{cohomologous} $\alpha \sim \alpha'$ if there is $\beta\in A$ such that $\alpha_{g}'=\beta^{-1}\alpha_gg(\beta)$. This is an equivalence relation and the corresponding quotient of $\Cocycles^1(G,A)$ is the first cohomology (pointed) set $\Cohom^1(G,A)$.

Now let $X$ be an algebraic variety over $\kk$ and consider a Galois extension $\LL\supset\kk$ with the Galois group $G=\Gal(\LL/\kk)$. Set $X_{\LL}=X\times_{\Spec\kk}\Spec{\LL}$. Then $\Aut(X_{\LL})$ is equipped with a continuous action of $G$, given by
\begin{equation}\label{eq: Galois action on automorphisms}
\alpha\mapsto g(\alpha)=(\id\otimes g)\circ\alpha\circ (\id\otimes g^{-1}),\ \ g\in\Gal(\LL/\kk). 
\end{equation}
We then consider 1-cocycles with values in $\Aut(X_\LL)$ and the corresponding cohomology set \[\Cohom^1(\Gal(\LL/\kk),\Aut(X_\LL)).\]
Recall that a $\kk$-variety $Y$ is called an \emph{$\LL/\kk$-form} of $X$ if $Y_\LL\simeq X_\LL$ as $\LL$-varieties. It turns out that, if $X$ is quasiprojective, then the set of $\LL/\kk$ forms of $X$ is in bijection with the set $\Cohom^1(\Gal(\LL/\kk),\Aut(X_\LL))$. More precisely, given an $\LL/\kk$-form $Y$ of $X$, the corresponding 1-cocycle $\alpha$ can be chosen as follows. We fix an isomorphism $\varphi\colon Y_\LL\iso X_\LL$ and take $\alpha\colon G\to\Aut(X_\LL)$ to be the function such that the action of $G$ on $Y_\LL$ induces the \emph{twisted action} on $X_\LL$ via
\[
\varphi\circ (\id\otimes g)\circ\varphi^{-1}=\alpha_g\circ (\id\otimes g),
\]
or, in a slightly less formal form which we will use in the future:
\begin{equation}\label{eq: twisted action}
	\varphi\circ g\circ \varphi^{-1}=\alpha_g\circ g.
\end{equation}
Note that the cohomology class of $\alpha$ does not depend on $\varphi$.

\subsection{Severi-Brauer varieties}\label{subsec: Severi-Brauer varieties}

Recall that a {\it Severi-Brauer variety} over $\kk$ is a projective algebraic variety $X$ such that $X_{\overline{\kk}}=X\times_{\Spec\kk}\Spec\overline{\kk}\simeq\PP^{n}_{\overline{\kk}}$. Severi-Brauer varieties of dimension $n-1$ over~$\kk$ are in one-to-one correspondence with central simple $\kk$-algebras of degree~$n$ as they are both parametrized by the set $\Cohom^1(\Gal(\kk^{\rm sep}/\kk),\PGL_n(\kk^{\rm sep}))$. In particular, if a Severi-Brauer variety $X$ corresponds to a central simple algebra $A$, then one can define the \emph{opposite} variety $X^{\rm op}$ as the variety which corresponds to the inverse of $A$ in the Brauer group $\Br(\kk)$. 

A Severi–Brauer variety $X$ always splits (becomes \emph{trivial}) over a finite Galois extension $\LL$ of the base field~ $\kk$, i.e.~one has $X_\LL\simeq\PP_\LL^n$, see e.g. \cite[Corollary 5.1.5]{Gille}. 

\begin{prop}\label{prop: SB closed points}
	Let $X$ be a non-trivial Severi-Brauer surface over a field $\kk$. Then one has the following.
	\begin{enumerate}[leftmargin=*, labelindent=20pt, itemsep=5pt]
		\item The surface $X$ does not contain points of degree not divisible by~$3$.
		\item The surface $X$ always contains a closed point of degree $3$.
		\item Any point of degree 3 or 6 on $X$ is in general position.
	\end{enumerate}
\end{prop}
\begin{proof}
	See \cite[Theorem 53]{Kollar_SB} for the first two claims, and \cite[Lemma 2.6]{Shramov1} for the third one.
\end{proof}

The following statement is one of the crucial birational properties of non-trivial Severi-Brauer surfaces.

\begin{prop}[{\cite{Weinstein89}, also follows from \cite[Theorem 2.6]{Isk1996}}]\label{prop: birationality of SB}
	Let $S$ be a non-trivial Severi-Brauer surface over $\kk$ and $S\dashrightarrow S'$ be a birational map where $S'$ is a del Pezzo surface with $\rk\Pic(S')=1$. Then either $S'\simeq S$, or $S'\simeq S^{\rm op}$.
\end{prop}

In what follows, we will need several results on explicit parametrization of Severi-Brauer curves and surfaces. Since we were unable to find these results in the literature, we provide their complete proofs below.

\begin{lem}\label{lem: SB curve twisted action}
	Let $\LL/\kk$ be an extension of degree $2$, with $\Gal(\LL/\kk)=\langle h\rangle$. Let  $C$ be a Severi-Brauer curve defined over $\kk$ that splits over $\LL$. There is an isomorphism $\varphi\colon C_{\LL} \rightarrow \PP^1_{\LL}$ such that the twisted action \eqref{eq: twisted action} is given by 
	\[
	\alpha_h= \begin{pmatrix}
		0 & \delta \\
		1 & 0
	\end{pmatrix}\in \PGL_2(\LL)
	\] 
	for $\delta \in \kk^*$ and every  $\delta\in \kk^*$ gives a Severi-Brauer curve $C_{\delta}$. Moreover, $C_{\delta} \simeq C_{\delta'}$ if and only if $\delta'/\delta\in \Norm_h(\LL^*)$. 
\end{lem}
\begin{proof}
	Let $\{p,h(p)\}$ be a point of degree 2 on $C$. Given an isomorphism $\varphi\colon C_{\LL}\iso\PP^1_{\LL}$, we can suppose (possibly composing with an automorphism of $\PP_{\textbf{L}}^1$), that $\varphi(p)=[1:0]$, $\varphi(h(p))=[0:1]$. This implies that the element $\alpha_h\in \PGL_2(\LL)$ satisfying $\varphi \circ h \circ \varphi^{-1} = \alpha_{h} \circ h$ is of the form 
	\[
	\alpha_h=\begin{pmatrix}
		0 & \delta \\
		1 & 0
	\end{pmatrix}
	\] 
	for some $\delta\in \LL^*$. The cocycle condition is $\alpha_h \cdot h(\alpha_h)=\id$, which corresponds to $\delta\in \kk^*$. 
	
	Let $\delta,\delta'\in\kk^*$ be such that $\delta'=\Norm_{h}(\lambda)\delta$ for some $\lambda\in L^*$, and put 
	\[
	\alpha_h'=\begin{pmatrix}
		0 & \delta' \\
		1 & 0
	\end{pmatrix},\ \ 
	\beta=\begin{pmatrix}
		\lambda & 0 \\
		0 & 1
	\end{pmatrix}.
	\]
	Then $\beta \alpha_{h}h(\beta^{-1}) = h(\lambda)^{-1}\alpha_{h}'$, i.e. the two 1-cocycles given by $\alpha_{h}$ and  $\alpha_{h}'$ are equivalent, hence $C_{\delta}\simeq C_{\delta'}$. Conversely, if the two 1-cycles are equivalent then there is matrix $\beta\in\GL_2(\LL)$ and an element $\lambda\in\LL^*$ such that $\lambda \alpha_h=\beta \alpha_{h}'h(\beta^{-1})$. Therefore,
	\[
	\lambda h(\lambda)\delta\I_2=(\lambda \alpha_h)\cdot h(\lambda \alpha_h) = \beta \alpha_h' h(\alpha_h') \beta^{-1}=\delta'\I_2,
	\]
	hence the claim. 
\end{proof}

\begin{lem}[{\cite[Lemmas 2.3.3 and 3.2.9]{BlancSchneiderYasinsky}}]\label{lem: SB surface Z3 twisted action}
Let $\EE/\kk$ be a Galois extension of degree $3$, with $\Gal(\EE/\kk)=\langle g\rangle$.
	There is a one-to-one correspondence between Severi-Brauer surfaces over $\kk$ that split over $\EE$ and $\kk^* /\Norm_g(\EE^*)$. For $\xi \in \kk^*$ and $X_\xi$ the corresponding Severi-Brauer surface, there is an isomorphism $\varphi_{\xi}: (X_\xi)_{\EE} \iso \PP^2_{\EE}$ such that the twisted action \eqref{eq: twisted action} is given by  \[\alpha_g= \begin{pmatrix}
		0 & 0 & \xi \\
		1 & 0 & 0 \\
		0 & 1 & 0
	\end{pmatrix}\in \PGL_3(\EE).\]
	Moreover, if $X$ corresponds to $\xi\in\kk^*$, then $X^{\rm op}$ corresponds to $\xi^{-1}$.
\end{lem}

\begin{lem}\label{lem: SB surface S3 twisted action}
Let $\EE/\kk$ be a Galois extension, with \[\Gal(\EE/\kk)=\langle s,g\ |\ g^3=s^2=(sg)^2=\id\rangle\simeq\Sym_3.\]

\begin{enumerate}\item\label{SB S3 twisted action}
For each $\tau\in \EE^*$ such that $\xi=\tau s(\tau) sg(\tau)\in \EE^g$, there is a unique Severi-Brauer surface $X_{\xi}$ defined over $\kk$ and an isomorphism $\varphi\colon (X_\xi)_\EE\iso\PP_\EE^2$ such that the twisted action~\eqref{eq: twisted action} is given by
\[
		\alpha_g = \begin{pmatrix}
			0 & 0 & \xi \\
			1 & 0 & 0 \\
			0 & 1 & 0
		\end{pmatrix},\ \ \
		\alpha_s = \begin{pmatrix}
			1 & 0 & 0 \\
			0 & 0 & \tau \\
			0 & s(\tau^{-1}) & 0
		\end{pmatrix}.
		\]
		\item\label{Each SB S3} Every Severi-Brauer surface defined over $\kk$ that admits a $3$-point that splits over $\EE$ is isomorphic to a surface $X_\xi$ as in \ref{SB S3 twisted action}.
		\item\label{SB S3 isomorphism} Two surfaces $X_\xi$ and $X_{\xi'}$ as in \ref{SB S3 twisted action} are isomorphic over $\kk$ if and only if \[\frac{\xi'}{\xi}\in \Norm_g(\EE^*).\]
		In particular, $X_\xi$ is trivial if and only $\xi\in \Norm_g(\EE^*)$.
	\end{enumerate}			
\end{lem}

\begin{proof}
	\ref{SB S3 twisted action} To check that such a surface exists, we need to verify that there is a cocycle $r\mapsto \alpha_r$, or equivalently that $r\mapsto  \alpha_r\circ r$ is a group homomorphism from $\Gal(\EE/\kk)$ to $\PGL_3(\EE)\rtimes \Gal(\EE/\kk)$. As $\Gal(\EE/\kk)$ is generated by $g$ and $s$, the elements $\alpha_g$ and $\alpha_s$ determine the group homomorphism (or the cocycle). The relation $g^3=\id$ corresponds to $\alpha_g g(\alpha_g) g^2(\alpha_g)=\alpha_{g^3}=\id\in \PGL_3(\EE)$ and is equivalent to $g(\xi)=\xi$. The relation $\alpha_{s} s(\alpha_s)=\id$ is implied by the form of $\alpha_s$ and the relation $gs=sg^2$ corresponds to
	\[
	\alpha_{gs}=\alpha_g g(\alpha_s)=\alpha_{sg^2}=\alpha_ss(\alpha_{g^2})=\alpha_s s(\alpha_g)\cdot sg(\alpha_g)
	\] and thus, remembering that $g(\xi)=\xi$, to the equality
	\[
	\begin{pmatrix}
		0 & \xi sg^2(\tau^{-1}) & 0\\
		1 & 0 & 0\\
		0 & 0 & g(\tau)
	\end{pmatrix}=
	\begin{pmatrix}
		0 & s(\xi) & 0\\
		\tau & 0 & 0\\
		0 & 0 & s(\xi)s(\tau^{-1})
	\end{pmatrix}\in\PGL_3(\EE).
	\]
	This is equivalent to
	\[\tau\xi =s(\xi) sg^2(\tau), \quad s(\xi)=s(\tau)\tau g(\tau).\] Applying $s$ to the second equality, we obtain $\xi=\tau s(\tau) sg(\tau)$. Replacing it in the first equality, we obtain $\tau^2 s(\tau) sg(\tau)=s(\tau) \tau g(\tau) sg^2(\tau),$ which is equivalent to $\tau  sg(\tau)=  g(\tau) sg^2(\tau),$  and also to the equality $g(\xi)=\xi$.

	\ref{Each SB S3}
	Let $X$ be a Severi--Brauer surface defined over $\kk$.
	If $X$ is trivial, take $\tau=\xi=1$ in
	\ref{SB S3 twisted action}. The corresponding twisted
	action fixes $[1:1:1]$, so the resulting surface has
	a $\kk$-rational point and is therefore isomorphic to $X$.
	We may thus assume that $X$ is non-trivial.
	Let $\{p_1,p_2,p_3\}\subset X(\EE)$ be a $3$-point
	that splits over $\EE$. Its geometric components are
	non-collinear by Proposition~\ref{prop: SB closed points}.  By changing the order, we may suppose that $g(p_1)=p_2$ and $s(p_1)=p_1$, which implies that $g$ and $s$ correspond to $(123)$ and $(23)$. We can choose $\varphi$ so that $p_1,p_2,p_3$ are sent onto $[1:0:0],[0:1:0],[0:0:1]$, respectively. Thus $\alpha_g$ is of the form $\left(\begin{smallmatrix}
			0 & 0 & \xi \\
			\beta & 0 & 0 \\
			0 & 1 & 0
		\end{smallmatrix}\right)$, for some $\xi,\beta\in \EE^*$. Replacing $\varphi$ with its composition by a diagonal automorphism, we may assume that $\beta=1$. We then write $\alpha_s=\left(\begin{smallmatrix}
			1 & 0 & 0 \\
			0 & 0 & \tau \\
			0 & \kappa & 0
		\end{smallmatrix}\right)$ with $\tau,\kappa\in \EE^*$. As $\alpha_s s(\alpha_s)=\alpha_{s^2}=\mathrm{id}\in \PGL_3(\EE)$ we find $\kappa=s(\tau^{-1})$ and obtain the form as in \ref{SB S3 twisted action}, where $\xi=\tau s(\tau) sg(\tau)$ and $g(\xi)=\xi$ are given by the fact that we have a cocycle, as explained above.
	
	\ref{SB S3 isomorphism} We take two Severi-Brauer surfaces $X_\xi$ and $X_{\xi'}$ as in \ref{SB S3 twisted action}, induced by $\xi,\tau\in \EE$ and $\xi',\tau'\in \EE$ respectively, with $\xi=\tau s(\tau) sg(\tau)$, $\xi'=\tau' s(\tau') sg(\tau')\in \EE^g$.
	
	\begin{itemize}[leftmargin=*, labelindent=15pt, itemsep=5pt]
	
	\item {\it Step 1.} We claim that if there is $\lambda \in \EE^*$ such that \[\tau'=\tau s(\lambda)sg(\lambda){\lambda^{-1}},\] then $X_\xi$ and $X_{\xi'}$ are isomorphic\footnote{In fact, this condition is both necessary and sufficient.}. Indeed, we have \[\xi'=\tau' s(\tau') sg(\tau')=\tau s(\tau) sg(\tau) \lambda g(\lambda) g^2(\lambda)=\xi \Norm_g(\lambda).\] Define $M=\left(\begin{smallmatrix}
			1 & 0 & 0\\
			0 & \lambda & 0 \\
			0 & 0 & \lambda g(\lambda)
		\end{smallmatrix}\right)\in \PGL_3(\EE)$ and calculate\footnote{The equality here should be understood as equality in $\PGL_3(\EE)$.} \[M^{-1}\begin{pmatrix}
			0 & 0 & \xi \\
			1 & 0 & 0 \\
			0 & 1 & 0
		\end{pmatrix}g(M)=\begin{pmatrix}
			0 & 0 & \xi' \\
			1 & 0 & 0 \\
			0 & 1 & 0
		\end{pmatrix},\ \ 
		M^{-1} \begin{pmatrix}
			1 & 0 & 0 \\
			0 & 0 & \tau \\
			0 & s(\tau^{-1}) & 0
		\end{pmatrix}s(M)=\begin{pmatrix}
			1 & 0 & 0 \\
			0 & 0 & \tau' \\
			0 & s(\tau'^{-1}) & 0
		\end{pmatrix},\]
		which implies that $X_\xi$ and $X_{\xi'}$ are isomorphic.		
		
	\item {\it Step 2.} If $\xi=\xi'$, we write $\theta=\tau'/\tau\in \EE^*$ and obtain $\theta s(\theta) sg(\theta)=1$, since $\tau s(\tau) sg(\tau)=\xi=\xi'=\tau' s(\tau')sg(\tau')$. By choosing $\lambda=sg(\theta)=\theta^{-1}s(\theta^{-1})\in \EE^s$, we obtain $\lambda=g(\theta)$ and then 
	$s(\lambda)sg(\lambda){\lambda^{-1}}=sg(\lambda)=\theta$. We then obtain $\tau'=\tau s(\lambda)sg(\lambda){\lambda^{-1}}$ as in Step 1 and thus $X_\xi$ and $X_{\xi'}$ are isomorphic. This explains why the surface only depends on $\xi$ and thus motivates the notation $X_\xi$.	
		
	\item {\it Step 3.} We now show the sufficiency in \ref{SB S3 isomorphism}. If there is $\mu\in \EE^*$ such that $\xi'=\Norm_g(\mu)\xi$, we write	
	\[\widetilde{\tau}=\tau'\frac{\mu} {s(\mu)sg(\mu)},\] and deduce \[\widetilde{\tau} s(\widetilde{\tau})sg(\widetilde{\tau})=\tau' s(\tau') sg(\tau')  \frac{\mu} {s(\mu)sg(\mu)}\frac{s(\mu)} {\mu g(\mu)}\frac{sg(\mu)} {g^2(\mu)\mu}= \frac{\xi'} {\Norm_g(\mu)}=\xi\in \EE^g.\]
	So, we obtain a Severi-Brauer $X_\xi$ given by $(\xi,\widetilde{\tau})$. By Step 2, this surface is isomorphic to $X_\xi$, given by $(\xi,\tau)$, and by Step 1 the surface is isomorphic to $X_{\xi'}$.
	
	\item {\it Step 4.} Finally, we show the necessity in \ref{SB S3 isomorphism}. Suppose that $X_\xi$ and $X_{\xi'}$ are isomorphic. We proceed as in the proof of \cite[Lemma 2.3.3 (4)]{BlancSchneiderYasinsky}. Write $R_\xi=\left(\begin{smallmatrix}
			0 & 0 & \xi \\
			1 & 0 & 0 \\
			0 & 1 & 0
		\end{smallmatrix}\right)\in \GL_3(\EE^g)$ and  $R_{\xi'}=\left(\begin{smallmatrix}
			0 & 0 & \xi' \\
			1 & 0 & 0 \\
			0 & 1 & 0
		\end{smallmatrix}\right)\in \GL_3(\EE^g)$ and obtain
	$A\in \GL_3(\EE)$ and $a\in \EE^*$ such that $aR_{\xi}=A^{-1} R_{\xi'} g(A)$. Hence, we find
	\[
	A=a^{-1}R_{\xi'}g(A)R_\xi^{-1}.
	\]
	Applying $g$ and $g^2$ gives
	$g(A)=g(a)^{-1}R_{\xi'}g^2(A)R_\xi^{-1}$ and
	$g^2(A)=g^2(a)^{-1}R_{\xi'}AR_\xi^{-1}$.
	We then replace successively $A$, $g(A)$ and $g^2(A)$ and obtain
	\[
		ag(a)g^2(a)A
		=g(a)g^2(a)R_{\xi'}g(A)R_\xi^{-1}
		=g^2(a)R_{\xi'}^2g^2(A)R_\xi^{-2}
		=R_{\xi'}^3AR_\xi^{-3}
		=\xi'\xi^{-1}A.
	\]
	This yields $\xi'/\xi=ag(a)g^2(a)\in\Norm_g(\EE^*)$,
	as required.
	
	We have proven that $X_\xi$ and $X_{\xi'}$ are isomorphic if and only if $\xi'/\xi\in \Norm_g(\EE^*)$. To finish the proof, it remains to check that $X_1$ is trivial. For this, use the representative $\xi=\tau=1$, and observe that $[1:1:1]$ is fixed by $\varphi \Gal(\EE/\kk) \varphi^{-1}$, and thus $\varphi^{-1}([1:1:1])$ is an $\EE$-point of $X_1$ fixed by the Galois group $\Gal(\EE/\kk)$, so it is a $\kk$-rational point.	
	\end{itemize}
	
\end{proof}

\subsection{Involution surfaces}\label{subsec: involution surfaces}

In this section, we recall the birational properties of del Pezzo surfaces~$S$, such that $S_{\overline{\kk}}\simeq\PP_{\overline{\kk}}^1\times\PP_{\overline{\kk}}^1$, also known as \emph{involution surfaces}. We follow recent works of A.~Trepain \cite{Trepalin2023} and J.-L. Colliot-Th\'{e}l\`{e}ne \cite{JLCTQuadrics}. First of all, let us state the biregular classification of such surfaces. Note that we have $\rk\Pic(S)\in\{1,2\}$, as $\Pic(S)\subseteq\Pic(S_{\overline{\kk}})$.

\begin{prop}[{\cite[Lemma 7.3]{ShramovVologodskyPointless}}]\label{prop: biregular classification of dP8}
	Let $S$ be a del Pezzo surface of degree 8 such that $S_{\overline{\kk}}\simeq\PP_{\overline{\kk}}^1\times\PP_{\overline{\kk}}^1$. Then one of the following cases holds:
	\begin{enumerate}[leftmargin=*, labelindent=20pt, itemsep=5pt]
		\item If $\rk\Pic(S)=1$, then $S\simeq\Weil_{\LL/\kk}C$, the Weil restriction of scalars for a conic $C$ over a quadratic extension $\LL/\kk$. The extension $\LL/\kk$ and the conic $C$ are uniquely determined by $S$ up to conjugation by $\Gal(\LL/\kk)$.
		\item If $\rk\Pic(S)=2$, then $S\simeq C_1\times C_2$, where $C_1$ and $C_2$ are two conics over $\kk$, uniquely determined by $S$. 
	\end{enumerate}
	Moreover, $S$ is isomorphic to a quadric in $\PP_\kk^3$ if and only if $S\simeq C\times C$ or $S\simeq\Weil_{\LL/\kk}C_\LL$ for some conic $C$ over $\kk$.
\end{prop}

We then focus on del Pezzo surfaces $S$ of degree 8 with $\rk\Pic(S)=2$. By Proposition~\ref{prop: biregular classification of dP8}, one has $S\simeq C_1\times C_2$, where $C_1$ and $C_2$ are two conics over $\kk$, uniquely determined by $S$. The description of possible birational models of $S$ in the non-rational case (i.e. when $S(\kk)=\varnothing$) is due to J. Koll\`{a}r \cite{KollarConics} and A. Trepalin \cite{Trepalin2023}. 

\begin{mydef}[{\cite[7]{KollarConics}}]
	Let $C_1$ and $C_2$ be two smooth conics over $\kk$, and $b(C_1)$ and $b(C_2)$ be their classes in the Brauer group $\Br(\kk)$. If there exists a conic $C_3$ such that $b(C_3)=b(C_1)+b(C_2)$, then $C_3$ is called the \emph{Brauer product} $C_1*C_2$ of $C_1$ and $C_2$.
\end{mydef}

\begin{prop}[{\cite[Lemma 3.2]{Trepalin2023} and \cite[Lemma 8]{KollarConics}}]\label{prop: When Brauer product is defined}
	Suppose that at least one of the two conics $C_1$ and $C_2$ is not trivial. The Brauer product $C_1*C_2$ is defined if and only if there exists a quadratic extension $\LL/\kk$, such that both $C_1$ and $C_2$ have closed points of degree 2 whose splitting field is $\LL$; equivalently, the Severi-Brauer curves $C_1$ and $C_2$ are trivial over $\LL$.
\end{prop}

\begin{thm}[{\cite[Theorem 1.5]{Trepalin2023}, \cite[Example 6]{KollarConics}}]\label{thm: birational classification of dP8}
	Let $S$ be a del Pezzo surface of degree~8 over a perfect field $\kk$ such that $S_{\overline{\kk}}\simeq\PP_{\overline{\kk}}^1\times\PP_{\overline{\kk}}^1$. Suppose that $S(\kk)=\varnothing$ and $\rk\Pic(S)=2$. Then $S$ is a product of two smooth conics and one of the following holds:
	\begin{enumerate}[leftmargin=*, labelindent=20pt, itemsep=5pt]
		\item $S\simeq C\times C$ or $S\simeq C\times\PP_\kk^1$, where $C(\kk)=\varnothing$. Let $S\dashrightarrow S'$ be a birational map to another Mori fibre space $S'$. Then $S'$ is isomorphic to $C\times C$, or to $C\times\PP^1_\kk$, or to a $\kk$-form of a Hirzebruch surface $\mathbb{F}_{2k}$ over $C$. 
		\item $S\simeq C_1\times C_2$, where $C_1$ and $C_2$ are two conics with no $\kk$-points, not isomorphic to each other. Let $S\dashrightarrow S'$ be a birational map to another Mori fibre space $S'$. If the Brauer product $C_3=C_1*C_2$ is not defined, then $S'\simeq C_1\times C_2$. Otherwise $S'$ is isomorphic to $C_1\times C_2$, or to $C_1\times C_3$, or to $C_2\times C_3$. 
	\end{enumerate}	
\end{thm}

\subsection{The Sarkisov program}\label{sec: Sarkisov program}

The main tool for studying groups of birational self-maps $\Bir_\kk(S)$ is the so-called \emph{Sarkisov program}, which allows to decompose every birational map $S\dashrightarrow S'$ between Mori fiber spaces (see below) into a composition of simpler maps, known as \emph{Sarkisov links}. We recall the basics of this theory very briefly and refer to \cite{LamyZimmermann} and \cite{BLZ} for more details. 

Let $S$ be a surface, and $r \geqslant 1$ an integer.
We say that $S$ is a \textit{rank $r$ fibration} if there exists a surjective morphism $\pi \colon S \to B$ with connected fibers, where $B$ is a point or a smooth curve, with relative Picard number equal to $r$, and such that the anticanonical divisor $-K_S$ is $\pi$-ample. We will write $S/B$ if we want to emphasize the base of the fibration. An isomorphism between two fibrations $S/B$ and $S'/B'$ of the same rank $r$ is an isomorphism $S \stackrel{\sim}{\to} S'$ such that there exists an isomorphism on the bases (necessarily uniquely defined) that makes the following diagram commutes: 
\[
\xymatrix{
	S \ar[rr]^{\sim} \ar[d]_{\pi} && S ' \ar[d]^{\pi'} \\
	B \ar[rr]^{\sim} && B'
}
\]
We say that a rank~$r$ fibration $S/B$ \emph{factorises through} a rank~$r'$ fibration $S'/B'$, or that \emph{$S'/B'$ is dominated by $S/B$}, if the fibrations $S/B$ and $S'/B'$ fit in a commutative diagram 
\[
\begin{tikzcd}[link]
	S \ar[rrr] \ar[dr,dashed] &&& B \\
	& S' \ar[r] & B' \ar[ur]
\end{tikzcd}
\]
where $S\rat S'$ is a birational contraction, and $B' \to B$ is a morphism with connected fibers. Note that $r \geqslant r'$.

Rank 1 fibrations are simply 2-dimensional Mori fiber spaces in the sense of Mori theory (see e.g. \cite[Lemma 3.3]{BLZ}). So, geometrically rational Mori fiber spaces in dimension 2 are exactly del Pezzo surfaces $S$ with $\Pic(S)\simeq\ZZ$ and conic bundles $\pi\colon S\to B$ with $\Pic(S)\simeq\ZZ^2$.

The usual notion of Sarkisov links between 2-dimensional Mori fiber spaces can be rephrased in terms of rank $2$ fibrations. Namely, a 2-ray game allows to show \cite[Lemma 3.7]{BLZ} that a rank~$2$ fibration $T/B$ factorises through exactly two rank~$1$ fibrations $S/B_1$ and $S'/B_2$ (up to isomorphisms). The induced birational map $\chi: S\rat S'$ is called a {\it Sarkisov link}. In this paper, we will mostly consider \emph{links of type \II} between del Pezzo surfaces. These are diagrams
\begin{equation}\label{eq: Sarkisov link of type II}
	\xymatrix{
		&T\ar@{->}[dl]_{\eta}\ar@{->}[dr]^{\eta'}&\\
		S\ar@{-->}[rr]^{\chi}\ar[dr]&& S'\ar[dl]\\
		& \Spec\kk &}
\end{equation}
where $S$ and $S'$ are both minimal del Pezzo surfaces, and $B=\Spec\kk$.

The notion of rank~$3$ fibration recovers the notion of an {\it elementary relation} between Sarkisov links.

\begin{prop}[{\cite[Proposition 4.3]{BLZ}}]\label{pro:from T3}
	Let $Z/B$ be a rank~$3$ fibration.
	Then there are only finitely many Sarkisov links $\chi_i$ dominated by $Z/B$, and they fit in a relation
	\begin{equation}\label{eq: relation}
		\chi_t \circ \dots \circ \chi_1 = \id.
	\end{equation}
	In this situation, we say that (\ref{eq: relation})	is an \emph{elementary relation} between Sarkisov links, coming from the rank $3$ fibration $Z/B$.
	Observe that the elementary relation is uniquely defined by $Z/B$, up to taking the inverse, cyclic permutations and insertion of isomorphisms. 
\end{prop}

Let $S/B$ a Mori fiber space. We denote by $\BirMori(S)$ the {\it groupoid} of birational maps between Mori fiber spaces birational to $S$. Note that $\Bir(S)$ is a subgroupoid of $\BirMori(S)$. We now state the crucial technical result that we are going to use in this paper. Its first part is known as the {\it Sarkisov program}, which in most general form was proven by C. Hacon and J. McKernan \cite{HMcK}, and by Iskovskikh in dimension 2 \cite{Isk1996}. The second part is inspired by A.-S. Kaloghiros \cite[Theorem 1.3]{Kaloghiros} and was proven in \cite{LamyZimmermann,BLZ}.

\begin{thm}[\bf Sarkisov program]\label{thm: sarkisov} 
	Let $S/B$ be a Mori fiber space (of dimension 2). 
	\begin{enumerate}
		\item\label{sarkisov1} The groupoid $\BirMori(S)$ is generated by Sarkisov links and automorphisms.
		\item\label{sarkisov2} Any relation between Sarkisov links in $\BirMori(S)$ is generated by trivial and elementary relations.
	\end{enumerate}
\end{thm}

Trivial relations will be defined in Section \ref{sec: quotients}.

\section{Biregular classification}\label{sec: biregular classification}

Building on the construction \cite{CTKarpenkoMerkurjev} of J.-L. Colliot-Thélène, N. Karpenko and A. Merkurjev, in \cite[Theorem 3.4]{Blunk} M. Blunk gives an algebraic way to classify sextic del Pezzo surfaces $S$ over a field $\kk$. To every such surface, one associates a pair of separable $\kk$-algebras $B$ and $Q$, defined as endomorphism rings of
certain locally free sheaves on $S$. The algebras $B$ and $Q$ are Azumaya over their centers, which are étale quadratic and cubic extensions of $\kk$, respectively, and can be recovered from the action of $\Gal(\overline{\kk}/\kk)$ on the $(-1)$-curves of $S$. We also refer to \cite[Table 4]{AuelBernadara} for the list of possible Blunk's data and  corresponding arithmetic invariants of $S$. 

Likewise, starting from the action of the Galois group on the $(-1)$-curves and some extensions of $\kk$ associated with this action, we provide below a particularly transparent description of sextic del Pezzo surfaces $S$ with $\Pic(S)\simeq\ZZ$ from a birational perspective. Our primary motivation for this approach is to understand the birational models of $S$, their Sarkisov links and birational self-maps.

\subsection{The blow-up model over $\overline{\kk}$} Let $\kk$ be a perfect field and $S$ be a del Pezzo surface of degree $K_S^2=6$ over $\kk$. Recall that $S_{\overline{\kk}}$ is the blow-up $\pi\colon S_{\overline{\kk}}\to\PP^2_{\overline{\kk}}$ in three non-collinear points $p_1,p_2,p_3$, which we may assume to be $[1:0:0]$, $[0:1:0]$ and $[0:0:1]$, respectively. In particular, there is a unique del Pezzo surface of degree 6 over $\overline{\kk}$, and it can be given by the equation
\begin{equation}\label{eq: del Pezzo of degree 6}
	S_{\overline{\kk}}=\big\{([x_1:x_2:x_3],[y_1:y_2:y_3])\in\PP^2_{\overline{\kk}}\times\PP^2_{\overline{\kk}}:\ x_1y_1=x_2y_2=x_3y_3 \big\}.
\end{equation}
The projection to the first factor $\PP^2_{\overline{\kk}}$ in (\ref{eq: del Pezzo of degree 6}) is the blow down of three lines $E_1=\{x_2=x_3=0\}$, $E_2=\{x_1=x_3=0\}$ and $E_3=\{x_1=x_2=0\}$ onto the standard coordinate points, while the projection to the second factor is the blow down of $F_1=\{y_2=y_3=0\}$, $F_2=\{y_1=y_3=0\}$ and $F_3=\{y_1=y_2=0\}$.

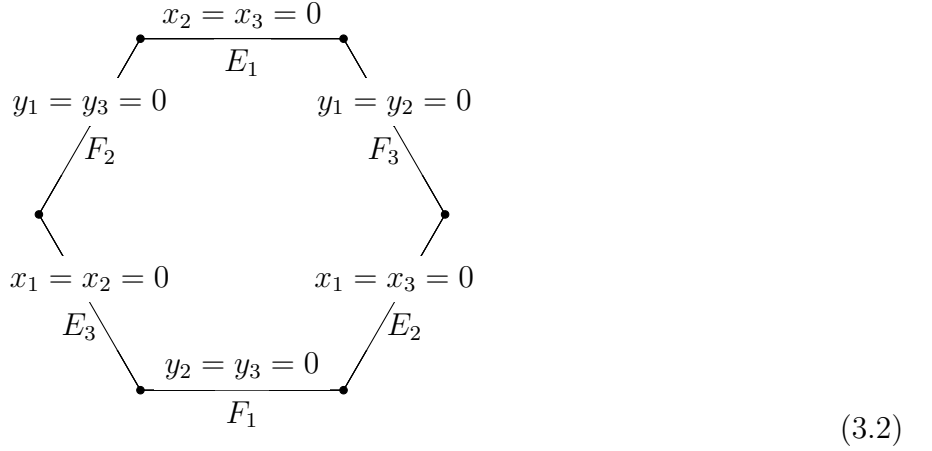
\begin{equation}\label{pic: hexagon}
	\begin{tikzpicture}
		\newdimen\R
		\R=2.7cm
		\draw (0:\R) \foreach \x in {60,120,...,360} {  -- (\x:\R) };
		
		\foreach \x/\l/\p in
		{ 60/{}/above,
			120/{}/above,
			180/{}/left,
			240/{}/below,
			300/{}/below,
			360/{}/right
		}
		\node[inner sep=1pt,circle,draw,fill,label={\p:\l}] at (\x:\R) {};
		
		\foreach \start/\end/\label in
		{ 0/60/$F_3\ \ $,
			60/120/$E_1$,
			120/180/$\ \ F_2$,
			180/240/$E_3\ \ $,
			240/300/$F_1$,
			300/360/$\ \ E_2$
		}
		\draw (\start:\R) -- (\end:\R) node[midway, below, fill=white] {\label};
		\foreach \start/\end/\label in
		{ 0/60/$y_1=y_2=0$,
			60/120/$x_2=x_3=0$,
			120/180/$y_1=y_3=0$,
			180/240/$x_1=x_2=0$,
			240/300/$y_2=y_3=0$,
			300/360/$x_1=x_3=0$
		}
		\draw (\start:\R) -- (\end:\R) node[midway, above, fill=white] {\label};
	\end{tikzpicture}
\end{equation}

In the anticanonical embedding $S_{\overline{\kk}}\hookrightarrow\PP^6_{\overline{\kk}}$ these $(-1)$-curves $\{E_i,F_i\}_{i=1,2,3}$ form a regular hexagon $\Sigma$. It is naturally acted on by $\Aut(S_{\overline{\kk}})$, so there is a homomorphism
\[
\Phi\colon \Aut(S_{\overline{\kk}})\to \Aut(\Sigma)\simeq\Dih_{6}=\langle r,s\ |\ r^6=s^2=\id,\ srs=r^{-1}\rangle,
\]
where $r$ is a rotation by $\pi/3$ and $s$ is a reflection. (On the figure \ref{pic: hexagon}, we choose $s$ to swap $E_1$ with $F_1$, $E_3$ with $F_2$ and $E_2$ with $F_3$.) The kernel of $\Phi$ is the maximal torus $T_{\overline{\kk}}\subset\PGL_3(\overline{\kk})$, isomorphic to $\overline{\kk}^*\times\overline{\kk}^*$, and it acts on $S$ by
\begin{equation}\label{eq: del Pezzo 6 torus action}
	(\lambda_0,\lambda_1,\lambda_2)\cdot ([x_1:x_2:x_3],[y_1:y_2:y_3])=([\lambda_0 x_1:\lambda_1x_2:\lambda_2x_3],[\lambda_0^{-1}y_1:\lambda_1^{-1}y_2:\lambda_2^{-1}y_3]).
\end{equation}
In what follows, we set $\lambda_0=1$ and identify such toric automorphism with a pair $(\lambda_1,\lambda_2)$. The action of $T_{\overline{\kk}}$ on $S\setminus\Sigma$ is faithful and transitive, and the automorphism group of $\Aut(S_{\overline{\kk}})$ fits into the short exact sequence 
\begin{equation}\label{eq: SES for Aut of DP6}
	\begin{tikzcd}
		1 
		\ar{r}
		& 
		{T_{\overline{\kk}}} 
		\ar{r}
		& 
		\Aut(S_{\overline{\kk}}) 
		\ar{r}{\Phi}
		& 
		\Dih_6
		\ar{r}
		& 
		1
	\end{tikzcd}
\end{equation}
with $\Phi(\Aut(S_{\overline{\kk}}))\simeq\Dih_6\simeq\Sym_3\times\ZZ/2$, where $\ZZ/2$ is generated by the lift of the standard quadratic involution, acting by
\begin{equation}\label{eq: del Pezzo 6 Cremona}
	\iota\colon ([x_1:x_2:x_3],[y_1:y_2:y_3])\mapsto ([y_1:y_2:y_3],[x_1:x_2:x_3]),
\end{equation}
and $\Sym_3$ acts naturally by permutations on each of the two triples $x_1,x_2,x_3$ and $y_1,y_2,y_3$. In what follows, we denote
\begin{gather}\label{eq: dP6 order 3 rotation}
	\theta\colon ([x_1:x_2:x_3],[y_1:y_2:y_3])\mapsto ([x_3:x_1:x_2],[y_3:y_1:y_2]),\\
	\sigma\colon ([x_1:x_2:x_3],[y_1:y_2:y_3])\mapsto ([y_1:y_3:y_2],[x_1:x_3:x_2])
\end{gather}
Finally, notice that the short exact sequence \eqref{eq: SES for Aut of DP6} splits and one has $\Aut(S_{\overline{\kk}})\simeq T_{\overline{\kk}}\rtimes\Dih_6$, where the action is given by
\begin{equation}\label{eq: semidirect structure}
	\theta\circ (\lambda_1,\lambda_2)\circ\theta^{-1}=(\lambda_2^{-1},\lambda_1\lambda_2^{-1}),\ \ \iota\circ (\lambda_1,\lambda_2)\circ\iota=(\lambda_1^{-1},\lambda_2^{-1}),\ \ \sigma\circ (\lambda_1,\lambda_2)\circ\sigma=(\lambda_2^{-1},\lambda_1^{-1})
\end{equation}
Now let $S$ be a sextic del Pezzo surface over $\kk$ and $\SplittingHex$ be the splitting field of the hexagon~ $\Sigma$, i.e. the smallest subfield $\kk\subset\SplittingHex\subset\overline{\kk}$, over which all six $(-1)$-curves are defined. Over $\SplittingHex$, there is a birational morphism $S_{\SplittingHex}\to\PP_{\SplittingHex}^2$ which blows down the triple $\{E_1,E_2,E_3\}$, hence $S_{\SplittingHex}$ can be given by the standard equation (\ref{eq: del Pezzo of degree 6}). We denote $M$ the surface given by this equation.

As was recalled in Section \ref{subsec: Galois cohomology}, the set of $\kk$-surfaces $S$ that are $\SplittingHex$-isomorphic to $M$, up to $\kk$-isomorphism, is parametrised by the cohomology set \[
\Cohom^1(\Gal(\SplittingHex/\kk),\Aut_{\SplittingHex}(M))=\Cohom^1(\Gal(\SplittingHex/\kk),(\SplittingHex^*)^2\rtimes\Dih_6).
\]
By the minimality of $\SplittingHex$, we may assume that there is an embedding $\epsilon\colon G=\Gal(\SplittingHex/\kk) \hookrightarrow\ZZ/2\times\Sym_3$ and denote the elements of $\Aut(\Sigma)\simeq\ZZ/2\times\Sym_3$ by $(i,\tau)$, where $i\in\{0,1\}$ and $\tau\in\Sym_3$. We suppose that the cycle~$(123)$ corresponds to the automorphism of $\Sigma$ which acts as $E_1\mapsto E_2\mapsto E_3\mapsto E_1$ on the hexagon~\eqref{pic: hexagon}, i.e. as $\Phi(\theta)$. 
	
The element $(1,\id)\in\ZZ/2\times\Sym_3$ sends the faces of $\Sigma$ to their opposites, i.e. acts as $\Phi(\iota)$. The element $(1,(23))$ acts as $\Phi(\sigma)$. We take, if they exist, 
\[
h=\epsilon^{-1}((1,\id)),\ \ g=\epsilon^{-1}((0,(123))),\ \ f=\epsilon^{-1}((1,(23))).
\]	
\begin{mydef}
	Given a subgroup $G\subset\Dih_6$, we call \emph{a sextic $G$-del Pezzo surface} a del Pezzo surface $S$ of degree 6 over $\kk$ such that $\Pic(S)\simeq\ZZ$ and the splitting field $\SplittingHex$ of its hexagon has $G$ as its Galois group over $\kk$, if such surface exists.
\end{mydef}

\begin{rem}
	Recall that by Iskovskikh's rationality criterion \cite[p. 642]{Isk1996}, a geometrically rational $\kk$-minimal surface $S$ is $\kk$-rational if and only if $K_S^2\geqslant 5$ and $S(\kk)\ne\varnothing$. In particular, a~sextic $G$-del Pezzo surface $S$ is $\kk$-rational if and only if $S(\kk)\ne\varnothing$.
\end{rem}

\begin{Notation}\label{notation: Galois groups}
	Since we assume $\Pic(S)\simeq\ZZ$, there are only 3 possibilities for the group $G$, namely $G=\langle r\rangle\simeq\ZZ/6\ZZ$, $G=\langle r^2,s\rangle\simeq\Sym_3$ and $G=\langle r,s\rangle\simeq\Dih_6$. In what follows, it will be convenient to choose the following presentations\footnote{It is not immediately clear why we take 2 (respectively, 3) generators for $\ZZ/6$ (respectively, for $\Dih_6$). However, this will simplify many formulas in our further calculations.} of these groups, using the Galois automorphisms from above (by abuse of notation, we denote $\epsilon(u)$ by $u$):
	
	\begin{itemize}
		\item $G=\langle g,h\ |\ g^3=h^2=\id,\ gh=hg\rangle$ when $G\simeq\ZZ/6$. 
		\item $G=\langle g,f\ |\ g^3=f^2=\id,\ fgf=g^2\rangle$ when $G\simeq\Sym_3$. 
		\item $G=\langle g,f,h\ |\ g^3=f^2=h^2=\id,\ hf=fh,\ gh=hg,\ fgf=g^2 \rangle$ when $G\simeq\Dih_6$.
	\end{itemize}
	
	In the last case, we set $s=fh$.
\end{Notation}

The corresponding actions of $G$ are shown on Figure \ref{fig: possible Galois actions}. 

\begin{figure}[ht]
	\centering
	\begin{subfigure}{0.32\textwidth}
		\centering
		\begin{tikzpicture}
			\newdimen\R
			\R=2.6cm
			\newdimen\A
			\A=1.8cm
			\newdimen\B
			\B=2.9cm
			\newdimen\C
			\C=2.2cm
			\draw (0:\R) \foreach \x in {60,120,...,360} {  -- (\x:\R) };
			\foreach \x/\l/\p in
			{ 60/{}/above, 120/{}/above, 180/{}/left,
				240/{}/below, 300/{}/below, 360/{}/right }
			\node[inner sep=1pt,circle,draw,fill,label={\p:\l}] at (\x:\R) {};
			\foreach \start/\end/\label in
			{ 0/60/$F_3\ \ $,
				60/120/$E_1$,
				120/180/$\ \ F_2$,
				180/240/$E_3\ \ $,
				300/360/$\ \ E_2$ }
			\draw (\start:\R) -- (\end:\R) node[midway, above, fill=white] {\label};
			\draw (240:\R) -- (300:\R) node[midway, below, fill=white] {$F_1$};
			
			\draw[<->, darkgreen] 
			($(60:\C)!0.5!(120:\C)$) -- ($(240:\C)!0.5!(300:\C)$)
			node[pos=0.6, below, fill=white] {\( h \)};
			\draw[<->, darkgreen] 
			($(120:\C)!0.5!(180:\C)$) -- ($(300:\C)!0.5!(360:\C)$);
			\draw[<->, darkgreen] 
			($(0:\C)!0.5!(60:\C)$) --  ($(180:\C)!0.5!(240:\C)$);
			\draw[->, red] 
			($(120:\A)!0.5!(180:\A)$) -- ($(0:\A)!0.5!(60:\A)$)
			node[midway, above,fill=white] {\( g \)};
			\draw[->, red] 
			($(0:\A)!0.5!(60:\A)$) -- ($(240:\A)!0.5!(300:\A)$);
			\draw[->, red] 
			($(240:\A)!0.5!(300:\A)$) -- ($(120:\A)!0.5!(180:\A)$);
		\end{tikzpicture}
		\caption{$\Gal(\FF/\kk)\simeq\ZZ/6\ZZ$}
	\end{subfigure}
	\hfill
	\begin{subfigure}{0.32\textwidth}
		\centering
		\begin{tikzpicture}
			\newdimen\R
			\R=2.6cm
			\newdimen\A
			\A=1.8cm
			\newdimen\B
			\B=2.4cm
			\newdimen\C
			\C=2.2cm
			\draw (0:\R) \foreach \x in {60,120,...,360} {  -- (\x:\R) };
			\foreach \x/\l/\p in
			{ 60/{}/above, 120/{}/above, 180/{}/left,
				240/{}/below, 300/{}/below, 360/{}/right }
			\node[inner sep=1pt,circle,draw,fill,label={\p:\l}] at (\x:\R) {};
			\foreach \start/\end/\label in
			{ 0/60/$F_3\ \ $,
				60/120/$E_1$,
				120/180/$\ \ F_2$,
				180/240/$E_3\ \ $,
				300/360/$\ \ E_2$ }
			\draw (\start:\R) -- (\end:\R) node[midway, above, fill=white] {\label};
			\draw (240:\R) -- (300:\R) node[midway, below, fill=white] {$F_1$};
			
			
			\draw[<->, blue] 
			($(60:\B)!0.5!(120:\B)$) -- ($(240:\B)!0.5!(300:\B)$)
			node[pos=0.1, left, fill=white] {\( f \)};  
			
			\draw[<->, blue] 
			($(120:\B)!0.5!(180:\B)$) -- ($(180:\B)!0.5!(240:\B)$);  
			
			\draw[<->, blue] 
			($(0:\B)!0.5!(60:\B)$) -- ($(300:\B)!0.5!(360:\B)$);  
			
			\draw[->, red] 
			($(120:\A)!0.5!(180:\A)$) -- ($(0:\A)!0.5!(60:\A)$)
			node[pos=0.3, midway, fill=white] {\( g \)};
			\draw[->, red] 
			($(0:\A)!0.5!(60:\A)$) -- ($(240:\A)!0.5!(300:\A)$);
			\draw[->, red] 
			($(240:\A)!0.5!(300:\A)$) -- ($(120:\A)!0.5!(180:\A)$);
		\end{tikzpicture}
		\caption{$\Gal(\FF/\kk)\simeq\Sym_3$}
	\end{subfigure}
	\hfill
	\begin{subfigure}{0.32\textwidth}
		\centering
		\begin{tikzpicture}
			\newdimen\R
			\R=2.6cm
			\newdimen\A
			\A=1.8cm
			\newdimen\B
			\B=2.4cm
			\newdimen\C
			\C=2.1cm
			\draw (0:\R) \foreach \x in {60,120,...,360} {  -- (\x:\R) };
			\foreach \x/\l/\p in
			{ 60/{}/above, 120/{}/above, 180/{}/left,
				240/{}/below, 300/{}/below, 360/{}/right }
			\node[inner sep=1pt,circle,draw,fill,label={\p:\l}] at (\x:\R) {};
			\foreach \start/\end/\label in
			{ 0/60/$F_3\ \ $,
				60/120/$E_1$,
				120/180/$\ \ F_2$,
				180/240/$E_3\ \ $,
				300/360/$\ \ E_2$ }
			\draw (\start:\R) -- (\end:\R) node[midway, above, fill=white] {\label};
			\draw (240:\R) -- (300:\R) node[midway, below, fill=white] {$F_1$};
			
			
			\draw[<->, blue] 
			($(60:\B)!0.5!(120:\B)$) -- ($(240:\B)!0.5!(300:\B)$)
			node[pos=0.1, left, fill=white] {\( f \)};  
			
			\draw[<->, blue] 
			($(120:\B)!0.5!(180:\B)$) -- ($(180:\B)!0.5!(240:\B)$);  
			
			\draw[<->, blue] 
			($(0:\B)!0.5!(60:\B)$) -- ($(300:\B)!0.5!(360:\B)$);  
			
			\draw[<->, darkgreen] 
			($(60:\C)!0.5!(120:\C)$) -- ($(240:\C)!0.5!(300:\C)$)
			node[pos=0.6, below, fill=white] {\( h \)};
			\draw[<->, darkgreen] 
			($(120:\C)!0.5!(180:\C)$) -- ($(300:\C)!0.5!(360:\C)$);
			\draw[<->, darkgreen] 
			($(0:\C)!0.5!(60:\C)$) --  ($(180:\C)!0.5!(240:\C)$);
			\draw[->, red] 
			($(120:\A)!0.5!(180:\A)$) -- ($(0:\A)!0.5!(60:\A)$)
			node[pos=0.3, midway, fill=white] {\( g \)};
			\draw[->, red] 
			($(0:\A)!0.5!(60:\A)$) -- ($(240:\A)!0.5!(300:\A)$);
			\draw[->, red] 
			($(240:\A)!0.5!(300:\A)$) -- ($(120:\A)!0.5!(180:\A)$);
		\end{tikzpicture}
		\caption{$\Gal(\FF/\kk)\simeq\Dih_6$}
	\end{subfigure}
	\caption{Possible actions of $\Gal(\SplittingHex/\kk)$ on $\Sigma$.}
	\label{fig: possible Galois actions}
\end{figure}

\subsection{Severi-Brauer data}\label{subsec: SB data}

The action of the Galois group $G=\Gal(\FF/\kk)$ on $\Sigma$ induces the following two actions:
\begin{itemize}
	\item The group $G$ acts on the set of triples $\{\{E_1,E_2,E_3\},\{F_1,F_2,F_3\}\}$, with a transitive action, as we assume $\Pic(S)\simeq\ZZ$. The stabilizer in $G$ of $\{E_1,E_2,E_3\}$ gives rise to a quadratic extension $\KK/\kk$, over which this triple is defined. Note that $\{F_1,F_2,F_3\}$ is then also defined over $\KK$. There are two $\KK$-blow-downs
	\[\xymatrix{
		&S\ar@{->}[dl]_{\eta}\ar@{->}[dr]^{\eta'}&\\	X\ar@{-->}[rr]^{\chi}&& X^{\rm op}\\
	}
	\]
	where $X$ is a Severi-Brauer surface, the birational morphism $\eta$ is the contraction of the triple $\{E_1,E_2,E_3\}$ onto a zero-dimensional $\KK$-subscheme of $X$, and $\eta'$ is the contraction of the triple $\{F_1,F_2,F_3\}$ onto a zero-dimensional $\KK$-subscheme of $X^{\rm op}$, see \cite[Proposition 2.1, Theorem 2.3]{Corn} and notice that characteristic zero assumption of loc. cit. is redundant, as del Pezzo surfaces are
	separably split. Note that $X$ uniquely determines $X^{\rm op}$.
	\item Furthermore, $G$ acts on the set of pairs $\{\{E_1,F_1\}, \{ E_2,F_2\},\{E_3,F_3\}\}$, with a transitive action, as $\Pic(S)\simeq\ZZ$. For each $i\in \{1,2,3\}$, there exists thus a cubic extension $\LL_i/\kk$ over which some pair $\{E_i,F_i\}$ is defined. Over this extension, we get a contraction $S\to Y_i$, where $Y_i$ is an involution surface over the field $\LL_i$. The extension $\LL_i/\kk$ is not necessarily Galois. Let $\LL$ be the compositum of the fields $\LL_i$.	Then $\LL=\FF^{\langle h\rangle}$ when
	$G\simeq\ZZ/6\ZZ$ or $G\simeq\Dih_6$, where $h$ is
	the unique central element of order~$2$.
	When $G\simeq\Sym_3$, one has $\LL=\FF$.
\end{itemize}

The described birational contractions serve as a starting point in the constructions of \cite{Blunk,CTKarpenkoMerkurjev}. To work with Galois cohomology, we will need to define similar data in more concrete terms, using Notation~\ref{notation: Galois groups}. 

\begin{mydef}\label{def: SB data}
	Let $S$ be a $G$-del Pezzo surface of degree 6, where $G$ is as in Notation \ref{notation: Galois groups}. A~\emph{Severi-Brauer data} for $S$ consists of:
	\begin{itemize}[leftmargin=*, labelindent=20pt, itemsep=5pt]
		\item $G\simeq\ZZ/6\ZZ$: A Severi-Brauer surface $X$ over the field $\KK=\FF^g$, such that $S_\KK$ is birational to~$X$, and an involution surface $Y$ with $\rk\Pic(Y)=2$ over the field $\LL=\FF^h$ such that $S_\LL$ is birational to $Y$.
		\item $G\simeq\Sym_3$: A Severi-Brauer surface $X$ over the field $\KK=\FF^g$ such that $S_\KK$ is birational to $X$.
		\item $G\simeq\Dih_6$: A Severi-Brauer surface $X$ over the field $\KK=\FF^{\langle g,s\rangle}$, such that $S_\KK$ is birational to~$X$, and an involution surface $Y$ with $\rk\Pic(Y)=2$ over the field $\LL=\FF^h$ such that $S_\LL$ is birational to $Y$.
	\end{itemize}
Let $S$ and $S'$ be two sextic $G$-del Pezzo surfaces with
the same splitting field $\SplittingHex$.
We say that their Severi--Brauer data are \emph{equivalent}
if $\KK=\KK'$, $X$ is $\KK$-birational to $X'$, and,
when $G\not\simeq\Sym_3$, $Y$ is $\LL$-birational to $Y'$.
\end{mydef}

\begin{rem}
	Let $S$ be a sextic $G$-del Pezzo surface over $\kk$.
	\begin{enumerate}[leftmargin=*, labelindent=20pt, itemsep=5pt]
		\item A Severi-Brauer data for $S$ exists, as follows from the discussion in the beginning of the paragraph, but is not uniquely defined: for example, there are two possible choices for the Severi-Brauer surface, $X$ and $X^{\rm op}$, and also several choices for an involution surface $Y$, as will be discussed below. 
		\item The extensions $\KK/\kk$ in Definition \ref{def: SB data} are all of degree~2. The extension $\LL/\kk$ is of degree~3 for $G\simeq\ZZ/6\ZZ$ and of degree~6 for $G\simeq\Dih_6$.
		Note that $\rk_\LL\Pic(Y)=2$ for $G\simeq\ZZ/6\ZZ$ and $G\simeq\Dih_6$. 
		\item Let $G\simeq\Dih_6$. Consider the central subgroup $\langle h\rangle$ and three conjugate subgroups $\langle f_i\rangle$, $i\in\{1,2,3\}$, where $f_1=f$, $f_2=gf$, $f_3=g^2f=fg$. Then 
		\[
		\begin{aligned}
			f_1 &= f \quad & \text{swaps the sections}\ x_2 = x_3 = 0\ \text{and}\ y_2 = y_3 = 0, \\
			f_2 &= gf \quad & \text{swaps the sections}\ x_1 = x_2 = 0\ \text{and}\ y_1 = y_2 = 0, \\
			f_3 &= fg \quad & \text{swaps the sections}\ x_1 = x_3 = 0\ \text{and}\ y_1 = y_3 = 0.
		\end{aligned}
		\]		
		We denote $\Klein_i=\langle h,f_i\rangle$ the three Klein 4-groups. Hence, for $\Dih_6$-del Pezzo surfaces there are three cubic extensions $\LL_i=\FF^{\Klein_i}\supset\kk$ over which we have contractions to involution surfaces $Y_1$, $Y_2$ and $Y_3$. In this case, $\rk_{\LL_i}\Pic(Y_i)=1$ for $i\in \{1,2,3\}$.
	\end{enumerate}
\end{rem}

\begin{rem} \label{rem: EmbedUniq}
In the case of $G\simeq \ZZ / 6 \ZZ$ and $G\simeq \Sym_3$ there is only one embedding $\epsilon$ of $G=\Gal(\SplittingHex / \kk)$ into $ \Dih_6$, up to conjugacy, among embeddings acting transitively on the hexagon. In the case of $G = \Dih_6$ there are up to conjugation two embeddings which corresponds to switching the roles of the elements called $f,s$. Since $\KK = \SplittingHex^{\langle g, s \rangle}$, the Severi-Brauer data uniquely gives the embedding.
\end{rem}

\subsection{Sextic $\ZZ/6\ZZ$-del Pezzo surfaces}

In this Section, we suppose that $\Gal(\FF/\kk)\simeq\ZZ/6\ZZ$ and find explicit parametrization of such del Pezzo surfaces via Galois cohomology. Throughout this paper, we use the split exact sequence \ref{eq: SES for Aut of DP6} and denote the automorphisms of $S_{\overline{\kk}}$ by pairs $((\lambda_1,\lambda_2),\delta)\in T_{\overline{\kk}}\rtimes\Dih_6$.

\begin{prop}\label{prop: Z6 twisted actions}
	Let $S$ be a sextic $\ZZ/6$-del Pezzo surface. Then the twisted action \eqref{eq: twisted action} can be chosen so that 
	\[
	\alpha_g = ((\xi^{-1}, \xi^{-1}),\epsilon(g)),\ \ \alpha_{h} = ((\rho, \rho g(\rho)), \epsilon(h))
	\]
	for 
	\begin{equation}\label{eq: Z6 coefficient conditions}
		\xi \in (\SplittingHex^g)^*,\ \ \rho \in (\SplittingHex^{h})^*,\ \ \Norm_{h}(\xi) \Norm_g(\rho) = 1.
	\end{equation}
	Moreover every such pair $\xi, \rho$ yields a del Pezzo surface with said properties.
\end{prop}
\begin{proof}
	The Galois group $G\simeq\ZZ/2\times\ZZ/3$ is generated by $(h,g)$. Letting 
	\[
	\alpha_g = ((\xi_1, \xi_2), \epsilon(g)),\ \ \alpha_{h} = ((\rho_1, \rho_2), \epsilon(h))
	\]
	we easily observe that this association defines a $1$-cocycle. We want to use the relations to get more information on the coefficients $\xi_i, \rho_i$. 
	
	First, we use that $g^3=\id$. One has
	$
	g(\alpha_g)=((g(\xi_1),g(\xi_2)),\epsilon(g)),
	$
	hence, using the semidirect product structure \eqref{eq: semidirect structure}, one gets
	\[
	\alpha_{g^2} = \alpha_g g(\alpha_g) = \left ( \left (\frac{\xi_1}{g(\xi_2)}, \frac{\xi_2 g(\xi_1)}{g(\xi_2)} \right ),\epsilon(g)^2 \right).
	\]
	Therefore, 
	\[
	((1,1),\id) = \alpha_{g^3} = \alpha_g g(\alpha_{g^2}) = \left (\left (\frac{\xi_1 g^2(\xi_2)}{g^2(\xi_1)g(\xi_2)}, \frac{\xi_2 g(\xi_1)}{g(\xi_2)g^2(\xi_1)}\right ),\id\right )
	\] 
	which is equivalent to $\xi_2 g(\xi_1) \in (\SplittingHex^g)^*$. Put $\xi= \xi_2 g(\xi_1)$ and replace $\varphi$ with $((\xi_1^{-1},\xi^{-1}),\id)\circ \varphi$, which replaces $\alpha_g$ with $((\xi^{-1}, \xi^{-1}),\epsilon(g))$.
	 
	Next, we use $h^2=\id$. One has $h(\alpha_h)=((h(\rho_1),h(\rho_2)),\epsilon(h))$, which gives 
	\[
	((1,1),\id) = \alpha_{h^2} = \alpha_{h}h(\alpha_{h}) = \left (\left (\frac{\rho_1}{h(\rho_1)}, \frac{\rho_2}{h(\rho_2)}\right ), \id\right )
	\]
	and thus $\rho_1,\rho_2\in ({\SplittingHex}^{h})^*$. 
	
	Finally, we use the relation $gh=hg$, which gives $\alpha_g g(\alpha_{h})=\alpha_{h}h(\alpha_g)$. Hence, by using the actions \eqref{eq: semidirect structure}, we get 
	\[
	\left(\left( \rho_1h(\xi), \rho_2 h(\xi)\right ), \epsilon(g h)\right ) = \left (\left (\frac{1}{\xi g(\rho_2)}, \frac{g(\rho_1)}{\xi g(\rho_2)}\right ) ,\epsilon(hg)\right )	
	\] 
	which is equivalent to $\rho_2 = \rho_1 g(\rho_1)$, $\Norm_{h}(\xi)\Norm_{g}(\rho_1) = 1$. Letting $\rho= \rho_1$, we get a pair $(\xi, \rho) \in (\SplittingHex^g)^* \times (\SplittingHex^{h})^*$ as in the statement. 
\end{proof}

\begin{prop}\label{prop: Z6 equivalence}
	Let $S$ and $S'$ be sextic $\ZZ/6$-del Pezzo surfaces with the same splitting field $\SplittingHex$, parametrized by $(\rho, \xi)$ and $(\rho',\xi')$, respectively. Then $S \simeq S'$ if and only if $(\rho, \xi) \sim (\rho', \xi')$ where the equivalence relation is generated by:
	\begin{equation}\label{eq: Z6 equivalences}
	(\rho, \xi) \sim (g(\rho), \xi), \, (\rho, \xi) \sim (\rho^{-1}, \xi^{-1}), \, (\rho, \xi) \sim (\Norm_{h}(\lambda)\rho, \Norm_g(\lambda^{-1}) \xi)\ \text{for}\ \lambda \in \SplittingHex.
	\end{equation}
\end{prop}
\begin{proof}
	Fix the twisted actions $\varphi\colon S_\SplittingHex\iso M$ and $\varphi'\colon S_\SplittingHex'\iso M$, and let $\alpha,\alpha'\colon G\to\Aut_{\SplittingHex}(M)$ be the corresponding 1-cocylces. Then $S \simeq S'$ if and only if these cocycles are cohomologous, i.e. there is $\beta \in\Aut_{\SplittingHex}(M)$ such that 
	\begin{equation}\label{eq: Z6 cohomologous cycles}
		\alpha_g' = \beta \alpha_g g(\beta^{-1}),\ \  \alpha_{h}' = \beta \alpha_{h} h(\beta^{-1})
	\end{equation}
	Thus, we start by looking at possible $\beta$s that keep the parametrization found in Lemma \ref{prop: Z6 twisted actions}
	
	\vspace{0.3cm}
	
	$\beta = ((\lambda_1,\lambda_2), \id)$: 
	\[
	((\xi'^{-1}, \xi'^{-1}),\epsilon(g)) = \alpha_g' = \beta \alpha_g g(\beta^{-1}) = \left (\left (\frac{\lambda_1 g(\lambda_2)}{\xi}, \frac{\lambda_2 g(\lambda_2)}{g(\lambda_1) \xi}\right ),\epsilon(g)\right )
	\]
	This is the case if and only if $\lambda_2=\lambda_1 g(\lambda_1)$ and $\xi=\xi'\Norm_g(\lambda_1)$. Now for $\alpha_{h}$ we find:
	\[
	((\rho', \rho' g(\rho')), \epsilon(h)) = \alpha_{h}'=\beta \alpha_{h} h(\beta^{-1})= ((\Norm_{h}(\lambda_1) \rho, \Norm_{h}(\lambda_2) \rho g(\rho)), \epsilon(h))
	\] 
	which is equivalent to $\rho' =\Norm_{h}(\lambda_1)\rho$. 
	
	\vspace{0.3cm}
	
	$\beta = ((1,1),\epsilon(h))$:
	\begin{gather*}
	((\xi'^{-1}, \xi'^{-1}),\epsilon(g)) = \alpha_g' = \beta \alpha_g g(\beta^{-1}) = ((\xi,\xi),\epsilon(g))
	\\
	((\rho', \rho' g(\rho')), \epsilon(h)) = \alpha_{h}'=\beta \alpha_{h} h(\beta^{-1}) = ((\rho^{-1}, \rho^{-1}g(\rho^{-1})),\epsilon(h))
	\end{gather*}
	Therefore we get, that $\rho' = \rho^{-1}, \xi' = \xi^{-1}$. 
	
	\vspace{0.3cm}	
	$\beta = ((1,1),\epsilon(g))$:
	
	Put $\gamma=((\xi^{-1},\xi^{-1}),\id) \circ \beta=((\xi^{-1},\xi^{-1}),\epsilon(g))$. Then
	\begin{gather*}
	\alpha_g' = \gamma \alpha_g g(\gamma^{-1}) = \alpha_g,\ \ \ \
	\alpha_{h}' = \gamma\alpha_{h} h(\gamma^{-1}) = ((g^2(\rho), g^2(\rho)\rho),\epsilon(h)),
	\end{gather*}
	where we used that $\Norm_{h}(\xi) \Norm_g(\rho) = 1$. This leads us to $(\rho', \xi') = (g^2(\rho),\xi)$, but this is the same as the first equivalence in \eqref{eq: Z6 equivalences}.
	
	\vspace{0.3cm}
	
	Now take a general $\beta = ((\lambda_1, \lambda_2),\tau)=((\lambda_1,\lambda_2),\id)\circ ((1,1),\tau)$, $\tau\in \Dih_6$, that gives an equivalence of 1-cycles. The equations of \eqref{eq: Z6 cohomologous cycles} imply that $\tau$ belongs to the centralizer of $\epsilon(g)$ in $\Dih_6$, hence $\tau \in \langle \epsilon(g), \epsilon(h) \rangle\simeq\ZZ/6$. Therefore, $\beta$ is a composition of one of the maps from above, which finishes the proof.
\end{proof}

Recall that we assume $\KK=\SplittingHex^g$ and $\LL=\SplittingHex^h$ in our Severi-Brauer data of $S$. By definition, $S$ is $\KK$-birational to a Severi-Brauer surface $X$ and $\LL$-birational to an involution surface $Y$, but the choice of these surfaces is not unique. However, we are able to list all possibilities for them.

\begin{prop}\label{prop: Z6 parametrization of SB curves and surfaces}
	Let $S$ be a sextic $\ZZ/6$-del Pezzo surface, parametrized as in Proposition \ref{prop: Z6 twisted actions}. Then one has the following:
	\begin{enumerate}[leftmargin=*, labelindent=20pt, itemsep=5pt]
		\item\label{Z6 SB data surfaces} The Severi-Brauer surfaces over~$\KK=\SplittingHex^g$ given by $S$ are parametrized by $\xi$ and $\xi^{-1}$.
		\item\label{Z6 SB data curves} The involution surfaces over $\LL=\SplittingHex^h$ given by $S$ are $C_{\rho}\times C_{g(\rho)}$, $C_{\rho}\times C_{g^2(\rho)}$ and $C_{g(\rho)}\times C_{g^2(\rho)}$.
		\item\label{Z6 SB data curve iso} Moreover, the conics $C_{\rho}$, $C_{g(\rho)}$ and $C_{g^2(\rho)}$ are either all isomorphic to each other, or pairwise non-isomorphic and each of them is the Brauer product of the other two.
	\end{enumerate}
\end{prop}
\begin{proof}
	\ref{Z6 SB data surfaces} Recall that the projection to the first factor in 
	\[
	M=\big\{([x_0:x_1:x_2],[y_0:y_1:y_2])\in\PP^2_{\SplittingHex}\times\PP^2_{\SplittingHex}:\ x_0y_0=x_1y_1=x_2y_2 \big\}.
	\]
	realizes $M$ as the standard blow-up of $\PP_{\SplittingHex}^2$. Since by Proposition \ref{prop: Z6 twisted actions} the twisted action is given by the automorphism
	\[
	\alpha_g\colon ([x_0:x_1:x_2],[y_0:y_1:y_2])\mapsto ([x_2:\xi^{-1}x_0:\xi^{-1}x_1],[y_2:\xi y_0:\xi y_1]),
	\] 
	we see, by projecting to the first and to the second factors, that the twisted actions on a Severi-Brauer surfaces are given, respectively, by the matrices
	\[
	\begin{pmatrix}
		0 & 0 & 1\\
		\xi^{-1} & 0 & 0\\
		0 & \xi^{-1} & 0
	\end{pmatrix}\sim
	\begin{pmatrix}
		0 & 0 & \xi\\
		1 & 0 & 0\\
		0 & 1 & 0
	\end{pmatrix}\ \ \text{and}\ \ \
	\begin{pmatrix}
		0 & 0 & 1\\
		\xi & 0 & 0\\
		0 & \xi & 0
	\end{pmatrix}\sim
	\begin{pmatrix}
		0 & 0 & \xi^{-1}\\
		1 & 0 & 0\\
		0 & 1 & 0
	\end{pmatrix},
	\]
	and we are done by Lemma \ref{lem: SB surface Z3 twisted action}.
	
	\ref{Z6 SB data curves} We consider 3 rational maps  $M \rat \PP^1_{\SplittingHex} \times \PP^1_{\SplittingHex}$ which define the contractions of the 2-sections in 3 conic bundles on $M$. Write $x=[x_0:x_1:x_2]$ and $y=[y_0:y_1:y_2]$. Then
	\begin{flalign}
		& \pi_{01}(x,y)=
		\begin{cases}
			([x_0:x_2],[y_1:y_2]), & x_0\ne0 \text{ or } y_1\ne0,\\
			([y_2:y_0],[x_2:x_1]), & x_1\ne0 \text{ or } y_0\ne0
		\end{cases}
		\quad\text{contracts }x_0=x_1=0,\ y_0=y_1=0
		&\label{eq: contraction of x0=x1}\\[3pt]
		& \pi_{12}(x,y)=
		\begin{cases}
			([x_0:x_1],[y_0:y_2]), & x_1\ne0 \text{ or } y_2\ne0,\\
			([y_1:y_0],[x_2:x_0]), & x_2\ne0 \text{ or } y_1\ne0
		\end{cases}
		\quad\text{contracts }x_1=x_2=0,\ y_1=y_2=0
		&\label{eq: contraction of x1=x2}\\[3pt]
		& \pi_{02}(x,y)=
		\begin{cases}
			([x_0:x_1],[y_1:y_2]), & x_0\ne0 \text{ or } y_2\ne0,\\
			([y_1:y_0],[x_2:x_1]), & x_2\ne0 \text{ or } y_0\ne0
		\end{cases}
		\quad\text{contracts }x_0=x_2=0,\ y_0=y_2=0
		&\label{eq: contraction of x0=x2}
	\end{flalign}
	For the map $\pi_{01}$, we let $([u_0:u_1],[v_0:v_1])=([x_0:x_2],[y_1:y_2])$ be the coordinates on $\PP_{\SplittingHex}^1\times\PP_{\SplittingHex}^1$, and find that the induced twisted action is
	\[
	([u_0:u_1],[v_0:v_1])\mapsto ([u_1:\rho g(\rho)u_0],[v_1:g(\rho^{-1})v_0])=([g^2(\rho)\Norm_h(\xi)u_1:u_0],[g(\rho)v_1:v_0]),
	\] 
	so the surface is isomorphic to $C_{g^2(\rho)}\times C_{g(\rho)}$. We leave to the reader to write down the other two contractions and verify the induced twisted action. Similarly, for $\pi_{12}$ and $\pi_{02}$ we find that the induced twisted action defines the surfaces $C_{\rho}\times C_{g^2(\rho)}$ and $C_{\rho}\times C_{g(\rho)}$, respectively.
	
	\ref{Z6 SB data curve iso} If any element among $\{\rho,g(\rho),g^2(\rho)\}$ is an $h$-norm, then the other two are also $h$-norms, since $h$ and $g$ commute. Thus, if any of these conics is $\LL$-trivial, then the other two are also $\LL$-trivial. Now suppose that two of these conics are isomorphic, say $\rho=\Norm_h(\lambda)g(\rho)$. Then $g(\rho)=\Norm_h(g(\lambda))g^2(\rho)$, and hence all 3 conics are isomorphic. Finally, if all 3 conics are pairwise non-isomorphic, then the claim follows from Theorem \ref{thm: birational classification of dP8} applied to $\kk=\LL$.
\end{proof}

\begin{thm}\label{thm: Z6 iso criterion}
	Let $S$ and $S'$ be two sextic $\ZZ/6\ZZ$-del Pezzo surfaces. Then $S\simeq S'$ if and only if $\FF=\FF'$ and their Severi-Brauer data are equivalent.
\end{thm}
\begin{proof}
	If $S\simeq S'$ then $\SplittingHex=\SplittingHex'$ and therefore $\KK=\KK'$, $\LL=\LL'$. Since $S_\KK$ is $\KK$-isomorphic to $S_{\KK}'$, we find that $X$ is $\KK$-birational to $X'$. Similarly, $Y$ is $\LL$-birational to $Y'$. Conversely, suppose that $X$ is $\KK$-birational to $X'$ and $Y$ is $\LL$-birational to $Y'$. Further, suppose that $S$ and $S'$ are parametrized by the pairs $(\rho,\xi)$ and $(\rho',\xi')$, respectively. 
	We then need to show that these pairs are equivalent in the sense of Proposition~\ref{prop: Z6 equivalence}.   
		
	By Proposition \ref{prop: Z6 parametrization of SB curves and surfaces}, one has $X\simeq X_{\xi}$ or $X\simeq X_{\xi^{-1}}$. Similarly, $X'\simeq X_{\xi'}$ or $X'\simeq X_{{\xi'}^{-1}}$. Since $X$ is $\KK$-birational to $X'$, we get that $\xi=\xi'$ or $\xi={{\xi'}^{-1}}$ in $\KK^*/\Norm_g(\FF^*)$, see Lemma \ref{lem: SB surface Z3 twisted action} and Proposition \ref{prop: birationality of SB}. By using the equivalences of Proposition~\ref{prop: Z6 equivalence}, we can replace the pair $(\rho,\xi)$ with $(\rho^{-1},\xi^{-1})$, and hence from the very beginning assume that $\xi$ and $\xi'$ parametrize isomorphic Severi-Brauer surfaces; note that $\rho$ and $\rho^{-1}$ parametrize isomorphic curves. Further, by Proposition \ref{prop: Z6 parametrization of SB curves and surfaces} one has $Y\simeq C_{g^i(\rho)}\times C_{g^j(\rho)}$, where $i,j\in\{0,1,2\}$ and $i\ne j$. Similarly, $Y'\simeq C_{g^k(\rho')}\times C_{g^l(\rho')}$, where $k,l\in\{0,1,2\}$ and $k\ne l$. But since $Y$ and $Y'$ are $\LL$-birational, Theorem \ref{thm: birational classification of dP8} implies that $\rho'=g^t(\rho)$ in $\LL^*/\Norm_h(\FF^*)$ for $t\in\{0,1,2\}$. 
	
	So, by using the equivalences of Proposition~\ref{prop: Z6 equivalence}, we can switch, if necessary, between $ \rho'$, $g(\rho')$, and $g^2(\rho')$ to make sure that $\rho$ and $\rho'$ parametrize the same curve. This gives us $\lambda,\lambda' \in \SplittingHex$ such that $\xi' = \Norm_g(\lambda) \xi$ and $\rho' = \Norm_{h}(\lambda') \rho$. Using the last equivalence of \eqref{eq: Z6 equivalences}, we may assume that $\lambda' = 1$ and hence $\rho'=\rho$. Thus we need to show that $(\rho, \xi)\sim (\rho, \Norm_g(\lambda) \xi)$. The equalities $\Norm_{h}(\xi^{-1}) = \Norm_g(\rho)$ and $\Norm_{h}(\xi'^{-1}) = \Norm_g(\rho')$ imply
	\[
	\Norm_g(\rho)=\Norm_h(\xi^{-1})=\Norm_h\left ({\xi'}^{-1}\Norm_g(\lambda)\right )=\Norm_g(\rho')\Norm_h(\Norm_g(\lambda)),
	\]
	hence $\Norm_{h} (\Norm_g(\lambda))=1$. Define $\mu= \lambda h(\lambda^{-1})\Norm_g(\lambda^{-1})$ and observe that $\Norm_{h}(\mu)=1$. Since $\Norm_g(\lambda) = \Norm_g(h(\lambda)^{-1})$, we get $\Norm_g(\mu) = \Norm_g(\lambda^{-1})$. Using the last equivalence of \eqref{eq: Z6 equivalences} with coefficient~$\mu^{-1}$, we obtain the desired equivalence. 
\end{proof}

\subsection{Sextic $\Sym_3$-del Pezzo surfaces}

The group $G\simeq\Sym_3$ is generated by $\epsilon(g)$ and $\epsilon(f)$. 

\begin{lem}\label{lem: S3 twisted actions}
	Let $S$ be a sextic $\Sym_3$-del Pezzo surface. Then the twisted action \eqref{eq: twisted action} can be chosen so that 
	\[
	\alpha_g = ((\xi^{-1},\xi^{-1} ),\epsilon(g)), \, \alpha_{f} = ((\eta, f(\eta) ), \epsilon(f))
	\]
	for $\eta\in \SplittingHex^*$, and
	\begin{equation}\label{eq: S3 coefficients conditions short}
	\xi=\frac{\eta}{f(\eta)fg(\eta)}\in\SplittingHex^g.
	\end{equation}
	Moreover every $\eta\in F^*$ such that $\frac{\eta}{f(\eta)fg(\eta)}\in\SplittingHex^g$ yields a del Pezzo surface with said properties.
\end{lem}
\begin{proof}
	We start with the same $\xi\in(\SplittingHex^*)^g$, $\alpha_{g}$ that we found in the $\ZZ/6$-case. Let $\alpha_{f} = ((\lambda_1, \lambda_2), \epsilon(f))$. We again use relations to extract the information about coefficients.
	
	Firstly, $f^2=\id$, hence
	\[ 
	((1,1), \id) = \alpha_{f}f(\alpha_{f}) = \left (\left (\frac{\lambda_1}{f(\lambda_2)},\frac{\lambda_2}{f(\lambda_1)}\right ),\id\right ).
	\]
	Put $\eta= \lambda_1$, then $\lambda_2 = f(\eta)$.
	
	\vspace{0.3cm}
		
	Secondly, $fg=g^2f$. Then $\alpha_{f}f(\alpha_g) = \alpha_{g^2} g^2(\alpha_{f})$, which gives 
	\[ 
	\left (\left (\eta f(\xi), f(\eta\xi) \right ), \epsilon(fg)\right ) = \left (\left (g^2 (f(\eta)\eta^{-1}),g^2 (\eta^{-1}) \xi^{-1} \right ), \epsilon(g^2f)\right ).
	\]
	This is equivalent to 
	\begin{equation}\label{eq: S3 coefficients conditions}
		\Norm_{f}(g(\eta))=\Norm_{gf}(\eta),\ \ \xi=\frac{\eta}{f(\eta)fg(\eta)}.
	\end{equation}
	It remains to notice that, adding to these two conditions the requirement $\xi\in\SplittingHex^g$ is equivalent to the condition \eqref{eq: S3 coefficients conditions short}. 
\end{proof} 

\begin{prop}\label{prop: S3 equivalence}
	Let $S$ and $S'$ be $\Sym_3$-del Pezzo surfaces of degree $6$, parametrized by~$\eta$ and~$\eta'$, respectively. Then $S \simeq S'$ if and only if $\eta\sim\eta'$, where the equivalence relation is generated by\footnote{This implies $(\xi,\eta)\sim (\Norm_g(\lambda^{-1})\xi, \Norm_{f}(\lambda)fg(\lambda)\eta),\ (\xi,\eta)\sim (\xi^{-1},\eta^{-1}).$}:
	\[ 
	\eta\sim\Norm_{f}(\lambda)fg(\lambda)\eta,\ \ \ \eta\sim\eta^{-1}.
	\]
\end{prop}
\begin{proof}
	As in the previous proof, we notice that if $\alpha_g$ and $\alpha_g'$ are equivalent via the toric automorphism, it is of the form $\beta=((\lambda,\lambda g(\lambda)),\id)$, where $\xi=\xi'\Norm_g(\lambda)$. Now for $\alpha_f$ we get
	\[
	\alpha_f'=((\eta',f(\eta')),\epsilon(f))=\beta\alpha_f f(\beta^{-1})=\left ( (\Norm_f(\lambda)\eta fg(\lambda), \Norm_f(\lambda) g(\lambda) f(\eta)),\epsilon(f) \right ),
	\]	
	which gives $\eta'=\Norm_{f}(\lambda)fg(\lambda)\eta$. 
	
	Further, equivalence via $\beta=((1,1), \epsilon(h))$ gives us $\eta'=\eta^{-1}$. (Here and below, by abuse of notation, we write $\epsilon(h)$ to denote the central involution of $\Dih_6$, although $G$ does not contain $h$.) Since the centralizer of $\langle \epsilon(g),\epsilon(f)\rangle$ in $\Dih_6$ is $\epsilon(h)$, we are done. 
\end{proof}

We now get a better reparametrization.

\begin{prop}\label{prop: S3 final parametrization}
	In the setting of Lemma \ref{lem: S3 twisted actions}, the twisted action $\varphi$ can be chosen so that
	\[
	\alpha_g = ((\xi^{-1},\xi^{-1}), \epsilon(g)),\ \ \ \alpha_f = ((\xi^{-1},\xi^{-1}),\epsilon(f)),
	\]
	where $\xi\in\kk^*$. Two such parametrizations yield the same surface if and only if $\xi \sim \xi'$ where the equivalence is generated by $\xi \sim \xi^{-1}, \, \xi \sim\Norm_g(\mu)\xi$, with $\mu\in{(\SplittingHex^*)}^{f}$
\end{prop}
\begin{proof}
	We use equivalence found in Proposition \ref{prop: S3 equivalence} with coefficient $\lambda= fg(\eta^{-1})$ to repalce $\eta$ by $\Norm_f(fg(\eta^{-1})))fg(fg(\eta^{-1})))\eta = \Norm_f(g(\eta^{-1}))$. Hence we can assume that $\eta\in\SplittingHex^f$. But now the first equality in \eqref{eq: S3 coefficients conditions} implies $\eta\in\SplittingHex^g$. So, $\eta\in\kk^*$. Then the second equality in \eqref{eq: S3 coefficients conditions} implies $\eta=\xi^{-1}$.
	
	The equivalence via the automorphism $\beta=((1,1),\epsilon(h))$ gives that $(\xi,\xi)\sim (\xi^{-1},\xi^{-1})$. If $\beta=((\lambda,\lambda g(\lambda)),\id)$, then
	\[
	\alpha_g'=((\xi'^{-1}, \xi'^{-1}),\epsilon(g)) = \beta \alpha_g g(\beta^{-1})
	\]
	implies $\xi'=\xi\Norm_g(\lambda^{-1})$, as in the proof of Proposition \ref{prop: S3 equivalence}. On the other hand,
	\[
	\alpha_f'=(({\xi'}^{-1},{\xi'}^{-1}),\epsilon(f))=\beta\alpha_f f(\beta^{-1})=((\lambda f(\lambda)fg(\lambda)\xi^{-1},\lambda g(\lambda)f(\lambda)\xi^{-1}),\epsilon(f)).
	\]
	Then $\lambda f(\lambda)fg(\lambda)\xi^{-1}=\lambda g(\lambda)f(\lambda)\xi^{-1}$, so $fg(\lambda)=g(\lambda)$. On the other hand, since $\lambda g(\lambda)f(\lambda)\xi^{-1}$ must be $f,g$-invariant, we get $\lambda=gf(\lambda)$. So, $\lambda g(\lambda)f(\lambda)=\lambda g(\lambda) g^2(\lambda)=\Norm_g(\lambda)$, hence
	\[
	\alpha_f'=(({\xi'}^{-1},{\xi'}^{-1}),\epsilon(f))=\beta\alpha_f f(\beta^{-1})=((\Norm_g(\lambda)\xi^{-1},\Norm_g(\lambda)\xi^{-1}),\epsilon(f)),
	\]
	i.e. we again get $\xi'=\Norm_g(\lambda^{-1})\xi$. Putting $\mu=g(\lambda^{-1})$, we get the result.
\end{proof}

\begin{thm}\label{thm: S3 iso criterion}
	Let $S$ and $S'$ be two sextic $\Sym_3$-del Pezzo surfaces. Then $S\simeq S'$ if and only if $\FF=\FF'$ and their Severi-Brauer data are equivalent.
\end{thm}
\begin{proof}
	The necessity is clear. Suppose that $\xi$ and $\xi'$ parametrize $S$ and $S'$, respectively. If the corresponding Severi-Brauer surfaces $X$ and $X'$ are $\KK$-birational, then $\xi'=\xi$ or $\xi'=\xi^{-1}$ in $\KK^*/\Norm_g(\SplittingHex^*)$ and  Proposition \ref{prop: S3 final parametrization} implies that we can replace $\xi$ with $\xi^{-1}$, if needed, and assume that $\xi$ and $\xi'$ parametrize isomorphic Severi-Brauer surfaces, thus $\xi' =\Norm_g(\mu')\xi$. Since both $\xi$ and $\xi'$ are fixed by $f$, we see that $\Norm_g(\mu')$ is fixed by $f$. Then 
	\[
	\mu= \frac{\Norm_g(\mu')}{\mu'f(\mu')}
	\] 
	is fixed by $f$ and has the same $\Norm_g$-norm as $\mu'$. This gives that the two sextic del Pezzos are isomorphic.
\end{proof}

\subsection{Sextic $\Dih_6$-del Pezzo surfaces}

By abuse of notation, denoting $\epsilon(g)$, $\epsilon(f)$ and $\epsilon(h)$ by $g$, $f$ and $h$, respectively, we have  
\[
\Dih_6=\langle g,f,h\ |\ g^3=f^2=h^2=\id,\ hf=fh,\ gh=hg,\ fgf=g^2 \rangle.
\]

\begin{lem}
	Let $S$ be a sextic $\Dih_6$-del Pezzo surface. Then the twisted action \eqref{eq: twisted action} can be chosen so that 
	\[
	\alpha_g=((\xi^{-1}, \xi^{-1}), \epsilon(g)),\  \ \alpha_{f} = ((\eta, f(\eta) ), \epsilon(f)),\ \  \alpha_{h} = ((\rho, \rho g(\rho)), \epsilon(h))
	\]
	with conditions
	\begin{equation}\label{eq: D6 coefficients conditions short}
		\xi=\frac{\eta}{f(\eta)fg(\eta)},\ \ \Norm_{h}(\xi) \Norm_g(\rho) = 1,\ \ g(\rho)\Norm_f(\rho)=\Norm_h(f(\eta)),\ \  \xi\in(\SplittingHex^*)^g,\ \ \rho\in(\SplittingHex^*)^h.
	\end{equation}
	Moreover every such pair $(\rho,\eta)$ yields a del Pezzo surface with said properties.
\end{lem}
\begin{proof}
	By combining the proofs of Proposition \ref{prop: Z6 twisted actions} and Lemma \ref{lem: S3 twisted actions}, we get the required parametrizations by $\xi$, $\eta$ and $\rho$, satisfying the condition \eqref{eq: Z6 coefficient conditions} and \eqref{eq: S3 coefficients conditions}. The condition $\alpha_{hf}=\alpha_{fh}$ is then equivalent to
	\begin{equation}\label{eq: D6 key relation}
	g(\rho)\Norm_f(\rho)=\Norm_h(f(\eta)).
	\end{equation}
\end{proof}

\begin{lem}\label{lem: D6 pre-equivalence}
	Let $S$ and $S'$ be sextic $\Dih_6$-del Pezzo surfaces, parametrized by~$(\rho,\eta)$ and~$(\rho',\eta')$, respectively. Then $S \simeq S'$ if and only if $(\rho,\eta)\sim(\rho',\eta')$, where the equivalence relation is generated by
	\[ 
	(\rho,\eta)\sim (\rho^{-1},\eta^{-1}),\ \  (\rho,\eta)\sim (\Norm_h(\lambda)\rho,\Norm_f(\lambda)fg(\lambda)\eta),\ \ \lambda\in\SplittingHex^*.
	\]
\end{lem}
\begin{proof}
	Indeed, if two 1-cocylces $\alpha$ and $\alpha'$ are cohomologous via $\beta=((\lambda_1,\lambda_2),\tau)$, then
	\[
	\tau\in\Centralizer_{\Dih_6}(\epsilon(g))\cap \Centralizer_{\Dih_6}(\epsilon(h))\cap \Centralizer_{\Dih_6}(\epsilon(f))=\langle\epsilon(h)\rangle.
	\]
	Now the result follows from the proofs of Proposition \ref{prop: Z6 equivalence} and \ref{prop: S3 equivalence}.
\end{proof}

As in the previous case, we are now able to get rid of some parameters:

\begin{prop}\label{prop: Second D6 param}
	Let $S$ be a sextic $\Dih_6$-del Pezzo surface. Then the twisted action \eqref{eq: twisted action} can be chosen so that 
	\[
	\alpha_g=((\xi^{-1}, \xi^{-1}), \epsilon(g)),\  \ \alpha_{f} = ((\xi^{-1}, \xi^{-1} ), \epsilon(f)),\ \  \alpha_{h} = ((\rho, \rho g(\rho)), \epsilon(h))
	\]
	with conditions
	\begin{equation}\label{eq: D6 conditions on coefficients}
	\Norm_{h}(\xi) \Norm_g(\rho) = 1,\ \ g(\rho)\Norm_f(\rho)\Norm_h(\xi)=1,\ \  \xi\in(\SplittingHex^*)^{\langle g,f\rangle},\ \ \rho\in(\SplittingHex^*)^h.
	\end{equation}
	Moreover every such pair $(\rho,\xi)$ yields a del Pezzo surface with said properties.
	
	Two such parametrizations $(\rho,\xi)$ and $(\rho',\xi')$ yield the same surface if and only if $(\rho,\xi)\sim(\rho',\xi')$ where the equivalence is generated by 
	\[
	(\rho,\xi)\sim (\rho^{-1},\xi^{-1}),\ \ (\rho,\xi)\sim (\Norm_h(\mu)\rho,\Norm_g(\mu^{-1})\xi),\ \ \mu\in(\SplittingHex^*)^{gf}.
	\]
\end{prop}
\begin{proof}
	We again use the second equivalence found in Proposition \ref{lem: D6 pre-equivalence} with coefficient $\lambda= fg(\eta^{-1})$ to repalce $\eta$ by $\Norm_f(fg(\eta^{-1})))fg(fg(\eta^{-1})))\eta = \Norm_f(g(\eta^{-1}))$ to assume that $\eta\in\SplittingHex^f$. The first and the fourth equalities in \eqref{eq: D6 coefficients conditions short} imply 
	\[
	\eta=\frac{g(\eta)fg(\eta)}{gf(\eta)}=\frac{g(\eta)g^2(\eta)}{g(\eta)}=g^2(\eta),
	\]
	so $\eta\in\SplittingHex^g$. Now the first equality in \eqref{eq: D6 coefficients conditions short} gives $\eta=\xi^{-1}$. 
	
	The equivalence via the automorphism $\beta=((1,1),\epsilon(h))$ translates to $(\rho,\xi)\sim (\rho^{-1},\xi^{-1})$, as in the proof of Proposition \ref{prop: Z6 equivalence}. The equivalence via a toric automorphism, necessarily of the form $\beta=((\lambda,\lambda g(\lambda)),\id)$, translates to $(\rho,\xi)\sim (\Norm_h(\mu)\rho,\Norm_g(\mu^{-1})\xi)$, as in the proof of Propositions \ref{prop: Z6 equivalence} and \ref{prop: S3 final parametrization}, where $\mu$ is fixed by $gf$.
\end{proof}

\begin{prop}\label{prop: D6 parametrization of SB curves and surfaces}
	Let $S$ be a sextic $\Dih_6$-del Pezzo surface, parametrized as in Proposition \ref{prop: Second D6 param}. Then one has the following:
	\begin{enumerate}[leftmargin=*, labelindent=20pt, itemsep=5pt]
		\item\label{D6 SB data surfaces} The Severi-Brauer surfaces over~$\KK=\SplittingHex^{\langle g,s\rangle}$ given by $S$ are parametrized by $\xi$ and $\xi^{-1}$.
		\item\label{D6 SB data curves} The involution surfaces over $\LL=\SplittingHex^h$ given by $S$ are $C_{\rho}\times C_{g(\rho)}$, $C_{\rho}\times C_{g^2(\rho)}$ and $C_{g(\rho)}\times C_{g^2(\rho)}$.
		\item\label{D6 SB data curve iso} Moreover, the conics $C_{\rho}$, $C_{g(\rho)}$ and $C_{g^2(\rho)}$ are either all isomorphic to each other, or pairwise non-isomorphic and each of them is the Brauer product of the other two.
	\end{enumerate}
\end{prop}
\begin{proof}
	We only prove the statement \ref{D6 SB data surfaces} about surfaces, because for involution surfaces and curves the calculation is the same as in Proposition \ref{prop: Z6 parametrization of SB curves and surfaces}. We use the the parametrization found in Proposition \ref{prop: Second D6 param}. Note that the contraction of the 3 curves $E_1$, $E_2$ and $E_3$ is now defined over a field $\SplittingHex^{\langle g,s\rangle}$. We have
	\[
	\alpha_s=\alpha_h h(\alpha_f)=((\rho h(\xi),\rho g(\rho) h(\xi)),\epsilon(hf)),
	\] 
	hence the induced twisted action of $\alpha_g$ and $\alpha_s$ after contraction is given by the matrices
	\[
	\begin{pmatrix}
		0 & 0 & \xi\\
		1 & 0 & 0\\
		0 & 1 & 0
	\end{pmatrix} \ \ \ \text{and}\ \ \ 
	\begin{pmatrix}
		1 & 0 & 0\\
		0 & 0 & \rho h(\xi)\\
		0 & \rho g(\rho) h(\xi) & 0
	\end{pmatrix}.
	\]
	Notice that
	\[
	\rho g(\rho) h(\xi)= s(\rho^{-1}h(\xi^{-1})),
	\]
	because this is equivalent (in view of $\rho\in\SplittingHex^h$ and $\xi\in\SplittingHex^{\langle g,f \rangle}$) to $\Norm_f(\rho)\Norm_h(\xi)g(\rho)=1$, the equality insured by \ref{eq: D6 conditions on coefficients}. On the other hand, computing $\tau s(\tau)sg(\tau)$ with $\tau=\rho h(\xi)$ gives $\xi$. We conclude by using Lemma \ref{lem: SB surface S3 twisted action}.
\end{proof}

\begin{thm}\label{thm: D6 iso criterion}
	Let $S$ and $S'$ be two sextic $\Dih_6$-del Pezzo surfaces. Then $S\simeq S'$ if and only if $\FF=\FF'$ and their Severi-Brauer data are equivalent.
\end{thm}
\begin{proof}
The necessity is obvious, so we prove the sufficiency. Suppose that $S$ and $S'$ are parametrized by the pairs $(\rho,\xi)$ and $(\rho',\xi')$, respectively. We then show that these pairs are equivalent in the sense of Proposition~\ref{prop: Second D6 param}. In what follows we use that 
\[
g(\rho) \Norm_f(\rho)\Norm_h(\xi)=\Norm_g(\rho)\Norm_h(\xi)=1
\] 
implies $\rho = gf(\rho)$.   

By Proposition \ref{prop: D6 parametrization of SB curves and surfaces}, one has $X\simeq X_{\xi}$ or $X\simeq X_{\xi^{-1}}$, and $X'\simeq X_{\xi'}$ or $X'\simeq X_{{\xi'}^{-1}}$. Since $X$ is $\KK$-birational to $X'$, we get that $\xi=\xi'$ or $\xi={{\xi'}^{-1}}$ in $\FF^{\langle g\rangle*}/\Norm_g(\FF^*)$. By using the equivalences of Proposition~\ref{prop: Second D6 param}, we can replace the pair $(\rho,\xi)$ with $(\rho^{-1},\xi^{-1})$, and hence assume that $\xi$ and $\xi'$ parametrize isomorphic Severi-Brauer surfaces. As in the proof of Theorem \ref{thm: Z6 iso criterion}, we get $\rho'=g^t(\rho)$ in $\LL^*/\Norm_h(\FF^*)$ for $t\in\{0,1,2\}$.	

If $\rho'$ parametrizes the same curve as $g(\rho)$ we find that $\rho' = g(\rho) \Norm_h(\lambda)$, where $\lambda\in \SplittingHex$. Since $\rho'$ and $\rho$ are fixed by $gs = sg^2$, we get that 
\[
s(\rho)\Norm_h(sg^2(\lambda))=g(\rho)\Norm_h(\lambda),
\] 
hence, by applying $s$ to both sides, we have
\[
g(\rho) = \rho \Norm_h\left (\frac{g^2(\lambda)}{s(\lambda)} \right).
\]
Similarly, if $\rho'$ parametrizes the same curve as $g^2(\rho)$ we get that $\rho' = g^2(\rho) \Norm_h(\lambda)$. Since $\rho'$ is $sg^2$-invariant, we get $sg(\rho)\Norm_h(sg^2(\lambda)) = g^2(\rho)\Norm_h(\lambda)$. Applying $sg$ to this yields $\rho \Norm_h(g(\lambda)) = s(\rho)\Norm_h(sg(\lambda))$, and by using $sg^2(\rho) = \rho$ we obtain that $\rho \Norm_h(g(\lambda)) = s(sg^2(\rho))\Norm_h(sg(\lambda)) = g^2(\rho)\Norm_h(sg(\lambda))$. Hence 
\[
g^2(\rho) = \rho N_h\left (\frac{g(\lambda)}{sg(\lambda)}\right ).
\]
Thus, we may assume that $\rho'$ parametrizes the same conic as $\rho$. Then there are $\delta,\theta \in \SplittingHex$ such that $\rho' = \Norm_h(\delta)\rho$ and $\xi' = \Norm_g(\theta) \xi$. Since $\xi,\xi' \in \SplittingHex^f$ we find that $\Norm_g(f(\theta)) = f(\Norm_g(\theta)) = \Norm_g(\theta)$. Let $\theta' = \Norm_g(\theta)/(\theta gf(\theta))$, then $\theta' \in (\SplittingHex^{gf})^*$ and $\Norm_g(\theta') = \Norm_g(\theta)$, so we can use a norm-equivalence with $\mu=\theta'$ to get
\[
(\rho',\xi')=(\Norm_h(\delta)\rho,\Norm_g(\theta)\xi)\sim (\Norm_h(\theta')\Norm_h(\delta)\rho,\Norm_g({\theta'}^{-1})\Norm_g(\theta)\xi)=(\Norm_h(\theta'\delta)\rho,\xi)
\]
and hence assume $\xi' = \xi$ from the very beginning. So, it remains to show that the pair $(\Norm_h(\delta)\rho,\xi)$ is equivalent to $(\rho,\xi)$. Note that the conditions \eqref{eq: D6 conditions on coefficients} imply  
\[
\Norm_g(\Norm_h(\delta))=1,\ \ g\Norm_h(\delta)\Norm_f(\Norm_h(\delta))=1.
\]
We claim that there exists $\zeta\in\SplittingHex^{gf}$ such that $\Norm_g(\zeta) = 1$, $\Norm_h(\zeta)=\Norm_h(\delta)$. This will achieve the proof by using the norm-equivalence of Proposition \ref{prop: Second D6 param}.

We can write every $u \in\Gal(\SplittingHex/\kk)$ as $h^iv$, where $v \in \langle g,f \rangle, i \in \{ 0,1 \} $. It is easy to check that setting 
\[
\alpha_u=((\Norm_h(\delta^i), \Norm_h(\delta^i g(\delta^i))),\epsilon(u)).
\]
defines a sextic $\Dih_6$-surface $S_0$ parametrized by $\rho_0=\Norm_h(\delta)$, $\xi_0=1$. In particular,
\[
\alpha_g=((1,1),\epsilon(g)),\ \ \alpha_h=
((\Norm_h(\delta),\Norm_h(\delta g(\delta))),\epsilon(h)),\ \ \alpha_f=((1,1),\epsilon(f)).
\]
A pair $(\zeta_1,\zeta_2)$ corresponds to a $\kk$-rational point on $S_0$ if and only if $\alpha_u \circ u (\zeta_1,\zeta_2) = (\zeta_1,\zeta_2)$ for $u = g,h,f$. For $u=g$ we get
\begin{flalign}\label{eq: D6 surface u=g}
(\zeta_1,\zeta_2) = \alpha_g(g(\zeta_1,\zeta_2)) = (g(\zeta_2^{-1}),g(\zeta_1 \zeta_2^{-1}))
\end{flalign}
which is equivalent to $\zeta_2 = \zeta_1g(\zeta_1)$, $\Norm_g(\zeta_1) = 1$.

For $u=h$ we get
\begin{flalign}\label{eq: D6 surface u=h}
(\zeta_1,\zeta_2) = \alpha_h (h(\zeta_1,\zeta_2)) = (\Norm_h(\delta) h(\zeta_1^{-1}), \Norm_h(\delta g(\delta)) h(\zeta_2^{-1}))
\end{flalign}
which is equivalent to $\Norm_h(\zeta_1) = \Norm_h(\delta)$, $\Norm_h(\zeta_2) = \Norm_h(\delta g(\delta))$.

For $u=f$ we get
\begin{flalign}\label{eq: D6 surface u=f}
(\zeta_1,\zeta_2) = \alpha_f (f(\zeta_1,\zeta_2)) = (f(\zeta_2^{-1}),f(\zeta_1^{-1}))
\end{flalign}
which is equivalent to $\zeta_2 = f(\zeta_1)^{-1}$. Thus, the existence of a $\kk$-rational point on $S_0$ is equivalent to the existence of $\zeta \in \SplittingHex^*$ such that $\Norm_g(\zeta) = 1$, $\zeta g(\zeta)f(\zeta) = 1$, $\Norm_h(\zeta) = \Norm_h(\delta)$ which is further equivalent to $\Norm_g(\zeta) = 1$, $\Norm_h(\zeta) = \Norm_h(\delta)$, $gf(\zeta) =\zeta$. This is exactly $\zeta$ we are looking for. It remains to notice that $S_0$ is $\kk$-rational. This can be proven using the results of Section \ref{section: closed points}, but we sketch a direct argument below.

Let us prove that $S_0$ has points of degrees dividing $2$
and $3$. This implies that $\indexx(S_0)$ divides both $2$
and $3$, and hence equals~$1$. By
Remark~\ref{rem: index 1 equivalent to rationality},
$S_0$ is $\kk$-rational. Moreover, every $\kk$-rational
point lies outside the hexagon: otherwise, Galois
transitivity on its six components would force all six
components to pass through that point. Its coordinates
therefore provide the required $\zeta$.

Points of degree dividing 2 and 3 can be found directly by taking suitable orbits under the actions \eqref{eq: D6 surface u=g}, \eqref{eq: D6 surface u=h} and \eqref{eq: D6 surface u=f}. The point $(1,1)$ is invariant under $\alpha_g\circ g$ and $\alpha_f\circ f$, hence its $\Dih_6$-orbit is $\{(1,1),(\Norm_h(\delta), \Norm_h(\delta g(\delta)))\}$, which gives a point of degree dividing 2 on $S_0$. The point $(\delta,f(\delta^{-1}))$ is obviously $\alpha_f\circ f$-invariant, and 
\[
\{(\delta,f(\delta^{-1})), (g(f(\delta)),g(\delta f(\delta))), (g^2(f(\delta^{-1})\delta^{-1}),g^2(\delta^{-1}))\}
\]
is its $\Dih_6$-orbit, giving a point of degree dividing 3 on $S_0$. Indeed, the $\alpha_g\circ g$-invariance of this set is straightforward. As was noticed in the beginning of the proof, $\rho_0=\Norm_h(\delta)$ is $gf=fg^2$-invariant. We hence get
\[
\alpha_h\circ h(\delta,f(\delta^{-1}))=(\Norm_h(\delta)h(\delta^{-1}),g^2\Norm_h(\delta^{-1})hf(\delta))=(\delta,ffg^2\Norm_h(\delta^{-1})hf(\delta))=(\delta,f(\delta^{-1})),
\]
hence $\alpha_h\circ h$ fixes each element of the set, as $h$ commutes with $g$. Finally, we observe that
\[
\alpha_f\circ f(gf(\delta),g(\delta f(\delta)))=(fg(\delta^{-1}f(\delta^{-1})),fgf(\delta^{-1}))=(g^2(f(\delta^{-1})\delta^{-1}),g^2(\delta^{-1})). 
\]
\end{proof}

\subsection{The Amitsur invariant} As was already recalled in the Introduction, for any smooth projective geometrically irreducible variety $X$ over $\kk$, there is an exact sequence of groups
\[
0\to\Pic(X)\to\Pic(X_{\overline{\kk}})^{\Gal(\overline{\kk}/\kk)}\to\Br(\kk)\to\Br(\kk(X)).
\]
The group $\Am(X)=\Pic(X_{\overline{\kk}})^{\Gal(\overline{\kk}/\kk)}/\Pic(X)\subset\Br(\kk)$ is called the \emph{Amitsur subgroup} of $X$ in~$\Br(\kk)$. Note that this is a birational invariant. 
	
\begin{ex}\label{ex: Amitsur}
		By \cite[Proposition 5.1]{CTKarpenkoMerkurjev} and \cite[Proposition 2.5]{Trepalin2023}, one has the following.
		\begin{enumerate}[leftmargin=*, labelindent=20pt, itemsep=5pt]
			\item If $S$ is a Severi-Brauer surface over a field $\kk$, then $\Am(S)$ is spanned by the class of a central simple algebra of degree 3 over $\kk$. In particular, this group is either trivial, or isomorphic to $\ZZ/3\ZZ$, and the latter holds if and only if $S$ (or the corresponding algebra) is not trivial.
			\item Suppose that $S\simeq C_1\times C_2$, where $C_1$ and $C_2$ are two smooth conics. If $C_1\simeq C_2\simeq\PP_\kk^1$, then $\Am(S)=0$. If $C_1\simeq C_2$ are both not trivial, then $\Am(S)\simeq\ZZ/2$ is generated by the class of~  $C_1$ in $\Br(\kk)[2]$, the 2-torsion subgroup of the Brauer group. If $C_1$ and $C_2$ are not trivial and not isomorphic to each other, then $\Am(S)\simeq\ZZ/2\times\ZZ/2$ is generated by the classes of $C_1$ and $C_2$ in $\Br(\kk)[2]$.
		\end{enumerate}
\end{ex}
	
\begin{proof}[Proof of Theorem \ref{thm: Biregular}]
		The necessity is obvious, so we will prove that the three conditions from the statement are sufficient. The fields $\KK$ and $\LL$ are uniquely defined by $\FF=\FF'$, $S$ and $S'$ as the fields from the Severi-Brauer data. By Example \ref{ex: Amitsur}, if $\Am(S_\KK)=\Am(S'_\KK)=0$, then $X\simeq X'\simeq \PP_\KK^2$. If $\Am(S_\KK)=\Am(S'_\KK)\simeq\ZZ/3\ZZ$, then $X$ and $X'$ both generate the same subgroup of $\Br(\KK)$ isomorphic to $\ZZ/3\ZZ$, hence are birational. Similarly, the groups $\Am(S_\LL)$ and $\Am(S'_\LL)$ are both either trivial, or isomorphic to $\ZZ/2\times\ZZ/2$. In the fist case, $Y\simeq Y'\simeq\PP_\LL^1\times\PP_\LL^1$. In the second case, the non-trivial elements of the Amitsur subgroup correspond to 3 non-trivial conics of Propositions~\ref{prop: Z6 parametrization of SB curves and surfaces} and~\ref{prop: D6 parametrization of SB curves and surfaces}.
\end{proof}

We leave to the reader the reformulation this biregular classification in terms of two Brauer elements $\alpha\in\Br(\KK)[3]$ and $\beta\in\Br(\LL)[2]$, as in \cite{Blunk}.

\section{Automorphism groups}\label{sec: auto}

In this Section, we use the parametrization of sextic $G$-del Pezzo surfaces found in Section~\ref{sec: biregular classification}, to compute their automorphism groups. This complements the results of J.~Schneider and S.~Zimmermann \cite{SchneiderZimmermann} who computed the automorphism groups of $\kk$-rational surfaces, and also the works of C.~Shramov and V.~Vologodsky \cite{ShramovVologodskyPointless}, and C.~ Shramov and A.~Vikulova \cite{ShramovVikulova} which deal with automorphism groups of del Pezzo surfaces without $\kk$-rational points; see \cite{Boitrel,SmithCubicSurfaces,SmithQuartic,ZaitsevDegree5,ZaitsevQuadrics} for related discussion.

\medskip

Let $\varphi\colon S_\SplittingHex\iso M$ be the twisted action \eqref{eq: twisted action}. Consider the group homomorphism
\[
\Phi\colon\Aut_\kk(S)\to\Aut_\SplittingHex(M),\ \ \gamma\mapsto \varphi\circ\gamma\circ\varphi^{-1}.
\] 
Note that, since $\gamma$ commutes with all $u\in G$, one has
\[
\alpha_u\circ u\circ \Phi(\gamma)=\varphi\circ u\circ\varphi^{-1}\circ \varphi\circ\gamma\circ\varphi^{-1}=\varphi\circ \gamma\circ u\circ\varphi^{-1}=\Phi(\gamma)\circ\alpha_u\circ u,
\]
so $\Phi$ induces an isomorphism
\[
\Phi\colon\Aut_{\kk}(S) \iso \{\gamma \in \Aut_{\SplittingHex}(M) \colon u(\gamma) = \alpha_u^{-1}\circ\gamma\circ\alpha_u\ \text{for all}\ u\in G\},
\]
where the action of $u$ on $\gamma$ is given by the conjugation \eqref{eq: Galois action on automorphisms}.

In the following theorem, we stick to notation used in Section \ref{sec: biregular classification}. 

\begin{thm}\label{thm: Autom}
Let $S$ be a non-$\kk$-rational sextic $G$-del Pezzo surface, $\SplittingHex$ be its splitting field and $\varphi: S_{\SplittingHex}\iso M$ be the isomorphism \eqref{eq: twisted action}. 

\begin{enumerate}[leftmargin=*, labelindent=20pt, itemsep=5pt]
	
\item\label{Auto Case Z6} Suppose that $G\simeq\ZZ/6$. Put
\[
T_1=\big \{\big((\lambda,\lambda g(\lambda)),\id\big ) \colon \lambda\in\SplittingHex^*,\ \Norm_g(\lambda)=\Norm_h(\lambda)=1 \big \}\subset T_{\SplittingHex}.
\] 
Then one of the following holds:

\vspace{0.3cm}

\begin{enumerate}[label=(\alph*), left=1em]
	\item $\Phi(\Aut_\kk(S))\simeq T_1$ if $\xi,\rho$ are non trivial;
	\item $\Phi(\Aut_\kk(S))\simeq T_1\rtimes \langle\alpha_h\rangle$ if $\xi$ is trivial;
	\item $\Phi(\Aut_\kk(S))\simeq T_1\rtimes \langle\alpha_g\rangle$ if $\rho$ is trivial.
\end{enumerate}

\vspace{0.3cm}

\item\label{Auto Case S3} Suppose that $G\simeq\Sym_3$. Put 
\[
T_2=\big \{ \big ((\lambda,f(\lambda^{-1})),\id\big )\colon \lambda\in({\SplittingHex^*})^{gf},\  \Norm_g(\lambda)=1\big \}\subset T_{\SplittingHex}.
\]
Then $\Phi(\Aut_\kk(S))\simeq T_2$ (note that $\xi$ is not trivial here).

\item\label{Auto Case D6} Suppose that $G\simeq \Dih_6$. Put 
\[
T_3=\big \{ \big ((\lambda,f(\lambda^{-1})),\id\big )\colon \lambda\in({\SplittingHex^*})^{gf},\  \Norm_g(\lambda)=\Norm_h(\lambda) = 1\big \}\subset T_{\SplittingHex}.
\]
Then one of the following holds:

\begin{enumerate}[label=(\alph*), left=1em]
	\item[(a)] $\Phi(\Aut_\kk(S))\simeq T_3$ if $\xi$ is not trivial;
	\item[(b)] $\Phi(\Aut_\kk(S))\simeq T_3 \rtimes \langle\alpha_h\rangle$ if $\xi$ is trivial.
\end{enumerate}
\end{enumerate}
\end{thm}
\begin{proof}
Take $\gamma=((\lambda,\mu),v)\in\Phi(\Aut_\kk(S))$. Recall that we have the split short exact sequence~\eqref{eq: SES for Aut of DP6} and the subgroup group $\Dih_6\subset\Aut(S_\SplittingHex)$ is generated by the automorphisms $\iota$, $\theta$ and $\sigma$, thus $v\circ u = u\circ v$ for all $ u\in G$. Therefore,  $v\in\langle\epsilon(g),\epsilon(h)\rangle$ if $G\simeq\ZZ/6$, and $v\in\langle\epsilon(h)\rangle$ if $G\simeq\Sym_3$ or $G\simeq\Dih_6$. We then consider 3 cases.

\ref{Auto Case Z6} Suppose that $G\simeq\ZZ/6$. We first show that if $v$ is of order 3, then $S$ is $\LL$-trivial, i.e. $\rho$ is an $h$-Norm. Indeed, in this case we may assume that $v=\epsilon(g)$ and then the condition $u(\gamma) = \alpha_u^{-1}\circ\gamma\circ\alpha_u$ translates, for $u=g$ and $u=h$ respectively, to
\begin{flalign}
		& ((1,\xi^{-1}),\epsilon(g^{-1}))\circ ((\lambda,\mu),\epsilon(g))\circ ((\xi^{-1},\xi^{-1}),\epsilon(g)) = ((\xi^{-1}\lambda^{-1}\mu,\xi^{-2}\lambda^{-1}),\epsilon(g)) = ((g(\lambda),g(\mu)),\epsilon(g)) \hspace{1cm} & \label{eq: Z6 Auto conjugation of g by g} \\
		& ((\rho,\rho g(\rho)),\epsilon(h))\circ ((\lambda,\mu),\epsilon(g))\circ ((\rho,\rho g(\rho)),\epsilon(h)) = ((\lambda^{-1}\rho^2 g(\rho),\mu^{-1}\rho g(\rho^2)),\epsilon(g)) = ((h(\lambda),h(\mu)),\epsilon(g)) \hspace{1cm} & \label{eq: Z6 Auto conjugation of g by h}
\end{flalign}
which is equivalent to
\begin{equation}\label{eq: Z6 Auto conjugation of g by g and h - corollary}
		\begin{cases}
			\mu=\lambda g(\lambda)\xi,\\
			\lambda g(\mu)\xi^2=1
		\end{cases}\ \ \text{and}\ \ \
		\begin{cases}
			\Norm_h(\lambda)=\rho^2g(\rho)\\
			\Norm_h(\mu)=\rho g(\rho^2).
		\end{cases}
\end{equation}
Since $\rho\in\SplittingHex^h$, we get $\Norm_h(\lambda)=\Norm_h(\rho) g(\rho)$, so $\rho=\Norm_h(g^2(\lambda/\rho))$. Therefore, we may assume that $\rho=1$.	

So, if $S$ is not $\LL$-{rational}, then $\Phi(\Aut_\kk(S))$ cannot contain elements $((\lambda,\mu),v)$ with $v$ or order 3 or 6, hence $v=\id$, or $v=\epsilon(h)$. In the first case, we get
\begin{flalign}
	& ((1,\xi^{-1}),\epsilon(g^{-1}))\circ ((\lambda,\mu),\id)\circ ((\xi^{-1},\xi^{-1}),\epsilon(g))=((\lambda^{-1}\mu,\lambda^{-1}),\id)=((g(\lambda),g(\mu)),\id)\ \ \text{for}\ u=g \hspace{1cm} & \label{eq: Z6 Auto conjugation of toric by g}\\
	& ((\rho,\rho g(\rho)),\epsilon(h))\circ ((\lambda,\mu),\id)\circ ((\rho,\rho g(\rho)),\epsilon(h)) = ((\lambda^{-1},\mu^{-1}),\id) = ((h(\lambda),h(\mu)),\id)\ \ \text{for}\ u=h. \hspace{1cm} & \label{eq: Z6 Auto conjugation of toric by h}
\end{flalign}
The first equality is equivalent to $\mu=\lambda g(\lambda)$, $\Norm_g(\lambda)=1$, while the second is equivalent to $\Norm_h(\lambda)=\Norm_h(\mu)=1$. We conclude that all toric elements are of the form
\begin{equation}\label{eq: Z6 Auto toric elemements}
\big \{\big((\lambda,\lambda g(\lambda)),\id\big ) \colon \lambda\in\SplittingHex^*,\ \Norm_g(\lambda)=\Norm_h(\lambda)=1 \big \}.
\end{equation}
In the second case $v=\epsilon(h)$ we obtain	
\begin{flalign}
	& ((1,\xi^{-1}),\epsilon(g^{-1}))\circ ((\lambda,\mu),\epsilon(h))\circ ((\xi^{-1},\xi^{-1}),\epsilon(g))=((\lambda^{-1}\mu,\lambda^{-1}\xi^{-2}),\epsilon(h))=((g(\lambda),g(\mu)),\epsilon(h)), \hspace{1cm} & \label{eq: Z6 Auto conjugation of h by g} \\
	& ((\rho,\rho g(\rho)),\epsilon(h))\circ ((\lambda,\mu),\epsilon(h))\circ ((\rho,\rho g(\rho)),\epsilon(h)) = ((\lambda^{-1}\rho^2,\mu^{-1}\rho^2g(\rho^2)),\epsilon(h)) = ((h(\lambda),h(\mu)),\epsilon(h)). \hspace{1cm} & \label{eq: Z6 Auto conjugation of h by h}
\end{flalign}	
which is equivalent to
\begin{equation}\label{eq: Z6 Auto conjugation of h by g and h - corollary}
	\begin{cases}
		\mu=\lambda g(\lambda),\\
		\Norm_g(\lambda)\xi^2=1
	\end{cases}\ \ \text{and}\ \ \
	\begin{cases}
		\Norm_h(\lambda)=\rho^2\\
		\Norm_h(\mu)=\rho^2 g(\rho^2).
	\end{cases}
\end{equation}
Now $\Norm_g(\lambda)\xi^2=1$ is equivalent (recall that $\xi\in\SplittingHex^g$) to $\Norm_g(\lambda\xi)=\xi$, hence $S$ is $\KK$-trivial.
We conclude that if $S$ is neither $\KK$-trivial, nor $\LL$-trivial, then $\Phi(\Aut_\kk(S))$ consists only of toric automorphisms. If $S$ is $\KK$-trivial, then we put $\xi=1$. Then $\alpha_h\in\Phi(\Aut_\kk(S))$ and together with toric automorphisms from above generates the whole group $\Phi(\Aut_\kk(S))$. As was shown above, it acts on toric automorphism by inversion. 

Finally, suppose that $S$ is $\LL$-trivial, i.e. $\rho=1$. Note that $S$ is not $\KK$-trivial, since we assume $S$ not $\kk$-rational. Therefore, $v$ cannot be of order 2 (and hence of order 6 either), since otherwise \eqref{eq: Z6 Auto conjugation of h by g}, \eqref{eq: Z6 Auto conjugation of h by h} and \eqref{eq: Z6 Auto conjugation of h by g and h - corollary} again imply that $\xi=\Norm_g(\lambda\xi)$, a contradiction.  As before, we find that toric automorphisms are of the form \eqref{eq: Z6 Auto toric elemements}, so it remains to determine the form of an automorphism with $v$ of order 3. But here we simply notice, using $\Norm_h(\xi)=1$, that $\alpha_g\in\Phi(\Aut_\kk(S))$, hence the result.
	
\ref{Auto Case S3} Here every automorphism in $\Phi(\Aut_\kk(S))$ is of the form $\gamma=((\lambda,\mu),\id)$, or $\gamma=((\lambda,\mu),\epsilon(h))$. In the first case, the condition $u(\gamma)=\alpha_u^{-1}\circ\gamma\circ\alpha_u$
translates, for $u=g$ and $u=f$, respectively, to
	\begin{flalign}
		& ((1,\xi^{-1}),\epsilon(g^{-1}))\circ ((\lambda,\mu),\id)\circ ((\xi^{-1},\xi^{-1}),\epsilon(g))=((\lambda^{-1}\mu,\lambda^{-1}),\id)=((g(\lambda),g(\mu)),\id),& \label{eq: S3 Auto conjugation of toric by g}\\
		& ((\xi^{-1},\xi^{-1}),\epsilon(f))\circ ((\lambda,\mu),\id)\circ ((\xi^{-1},\xi^{-1}),\epsilon(f))=((\mu^{-1},\lambda^{-1}),\id)=((f(\lambda),f(\mu)),\id),& \label{eq: S3 Auto conjugation of toric by f}
	\end{flalign}
which gives $\mu=f(\lambda^{-1})=\lambda g(\lambda)$ and $\Norm_g(\lambda)=1$, which is equivalent to $\lambda \in (\textbf{F}^*)^{gf}$, $\Norm_g(\lambda) = 1$. In the second case, we respectively get
	\begin{flalign}
		& ((1,\xi^{-1}),\epsilon(g^{-1}))\circ ((\lambda,\mu),\epsilon(h))\circ ((\xi^{-1},\xi^{-1}),\epsilon(g))=((\lambda^{-1}\mu,\lambda^{-1}\xi^{-2}),\epsilon(h))=((g(\lambda),g(\mu)),\epsilon(h)),& \label{eq: S3 Auto conjugation of h by g}\\
		& ((\xi^{-1},\xi^{-1}),\epsilon(f))\circ ((\lambda,\mu),\epsilon(h))\circ ((\xi^{-1},\xi^{-1}),\epsilon(f))=((\xi^{-2}\mu^{-1},\xi^{-2}\lambda^{-1}),\epsilon(h))=((f(\lambda),f(\mu)),\epsilon(h)),& \label{eq: S3 Auto conjugation of h by f}
	\end{flalign}
which implies $\mu=\lambda g(\lambda),\, \xi^{-2} = \lambda g(\mu) = \Norm_g(\lambda)$.
Thus $\Norm_g(\xi \lambda)=\Norm_g(\xi)\Norm_g(\lambda)=\xi^3\cdot\xi^{-2} =\xi$, so $S$ is $\KK$-trivial. Therefore, if $S$ is not $\KK$-trivial then $\Phi(\Aut_\kk(S))$ consists only of toric automorphisms.

\ref{Auto Case D6}
Once again, every automorphism in $\Phi(\Aut_\kk(S))$ is of the form $\gamma=((\lambda,\mu),\id)$, or $\gamma=((\lambda,\mu),\epsilon(h))$. Suppose that our automorphism group contains a toric automorphism; then  the conjugation by $\alpha_f$ gives:
\begin{flalign}
	& ( (\xi^{-1},\xi^{-1}),\epsilon(f))\circ ((\lambda,\mu),\id)\circ ( (\xi^{-1},\xi^{-1}),\epsilon(f))=((\mu^{-1},\lambda^{-1}),\id)=((f(\lambda),f(\mu)),\id),& 
\end{flalign}
 together with \eqref{eq: Z6 Auto conjugation of toric by g} and \eqref{eq: Z6 Auto conjugation of toric by h} this gives that all the toric automorphisms are as in $T_3$ similar to the toric case in \ref{Auto Case S3} with the only difference being the additional assumption that the $h$-Norms are $1$. As in the previous case, we observe that if $\Phi(\Aut_\kk(S))$ contains an element of the form $((\lambda,\mu),\epsilon(h))$, then \eqref{eq: S3 Auto conjugation of h by g} implies that $\xi$ is a $g$-norm. Hence, by Lemma~\ref{lem: SB surface S3 twisted action} \ref{SB S3 isomorphism}, we may put $\xi = 1$ without loss of generality. \\
Further, the conjugation by $\alpha_f$ gives
	\begin{flalign}
	& ((1,1),\epsilon(f))\circ ((\lambda,\mu),\epsilon(h))\circ  ((1,1),\epsilon(f))=((\mu^{-1} ,\lambda^{-1}),\epsilon(h))=((f(\lambda),f(\mu)),\epsilon(h)),& 
\end{flalign}
which is equivalent to $\mu = f(\lambda^{-1})$. We claim that $\alpha_h$ belongs to $\Phi(\Aut_\kk(S))$. Indeed, the two conditions required by \eqref{eq: S3 Auto conjugation of h by g}
are satisfied, because $\xi=1$ and hence $\Norm_g(\rho)=1$. Further, the second system in \eqref{eq: Z6 Auto conjugation of h by g and h - corollary} is satisfied because $\lambda=\rho$ is $h$-invariant (and $\mu=\lambda g(\lambda)=\rho g(\rho)$ for $\alpha_h$). So, it remains to check that $\mu = f(\lambda^{-1})$, which is equivalent to $g(\rho)\Norm_f(\rho)=1$. But $g(\rho)\Norm_f(\rho)N_h(\xi)=1$, as was found in our parametrization of $S$ and since $\xi = 1$ this implies that $\alpha_h$ belongs to $\Phi(\Aut_\kk(S))$.	
\end{proof}

\section{Closed points}\label{section: closed points}

\subsection{Generalities}\label{subsec: closed points general}

We start with recalling some general facts. The \emph{index} $\indexx(X)$ of an algebraic variety $X$ over $\kk$ is the greatest common divisor of the degrees of all closed points on $X$, or equivalently, the greatest common divisor of the degrees of all finite field extensions $\EE/\kk$ such that $X$ has a point over $\EE$. It is well known that index is a birational invariant \cite[Proposition 6.8]{GabberLiuLorenzini} for smooth proper varieties. 

\begin{rem}\label{rem: index 1 equivalent to rationality} Let $S$ be a del Pezzo surface of degree 6 over $\kk$.
	\begin{enumerate}
		\item Recall that $|-K_S|$ gives an embedding of $S$ into $\PP_\kk^6$ as a surface of degree 6. By intersecting this surface by two hyperplanes in general position, we see that $S$ has a point of degree dividing $6$ (alternatively, we can consider the Galois action on the vertices of $\Sigma$). In particular, the index of a sextic del Pezzo surface always divides 6. If $G\simeq\ZZ/6\ZZ$ or $G\simeq\Dih_6$, the six vertices
		of $\Sigma$ form a closed point of degree~$6$.
		If $G\simeq\Sym_3$, they form two closed points of degree~$3$.
		
		\item If $S(\kk)\ne\varnothing$, then $\indexx(S)=1$.
		Conversely, a del Pezzo surface of degree~$6$ admitting a zero-cycle of degree prime to~$6$ has a rational point by \cite[Proposition~8]{Addington}. Hence $\indexx(S)=1$ implies $S(\kk)\ne\varnothing$.
	\end{enumerate}
\end{rem}

The following classical statement is known as the Lang-Nishimura lemma.

\begin{lem}[{\cite{LangSomeApplications,Nishimura}}]\label{lem: Lang-Nishimura}
	Let $\kk$ be a field, and let $\phi\colon X'\dashrightarrow X$ be a rational map of
	$\kk$-schemes. Assume that $X'$ has a smooth $\kk$-point and that $X$ is proper. Then $X(\kk)\ne\varnothing$.
\end{lem} 

Finally, there is an explicit criterion of $\kk$-rationality for a minimal geometrically rational surfaces:

\begin{thm}[{\cite{Isk1996}}]\label{Iskovskikh Criterion}
	A minimal geometrically rational surface $S$ over a perfect field $\kk$ is $\kk$-rational if and only if the following two conditions are satisfied:
	\begin{description}
		\item[(i)] $S(\kk)\ne\varnothing$;
		\item[(ii)] $K_S^2\geqslant 5$.
	\end{description}
\end{thm}

\begin{cor}
	A sextic $G$-del Pezzo surface $S$ is $\kk$-rational if and only if $S(\kk)\ne\varnothing$.
\end{cor}

\begin{mydef}
Let $S$ be a non-$\kk$-rational $G$-del Pezzo surface of degree 6. Consider its Severi-Brauer data from Definition \ref{def: SB data}. If the Severi-Brauer surface $X$ is trivial, i.e. $X\simeq\PP_\KK^2$, we say that \emph{$S$ is $\KK$-trivial}. If $G\not\simeq\Sym_3$ and
$Y\simeq\PP_\LL^1\times\PP_\LL^1$, we say that
\emph{$S$ is $\LL$-trivial}.
When $G\simeq\Sym_3$, we also regard $S$ as $\LL$-trivial,
since $\LL=\FF$ and $S_\FF$ is split. Note that these properties do not depend on the choice of $X$ or $Y$.
\end{mydef}

\begin{prop}\label{prop: index 2}
	Let $S$ be a $G$-del Pezzo surface of degree 6 over $\kk$, and assume that $S$ is not $\kk$-rational $($equivalently, $S(\kk)=\varnothing$$)$. The following statements are equivalent.
	\begin{enumerate}
		\item\label{closed points K-trivial 1} $S$ is $\KK$-trivial;
		\item\label{closed points K-trivial 2} $\indexx(S)=2$;
		\item\label{closed points K-trivial 3} $S$ has a degree 2 point.
	\end{enumerate}
\end{prop}

\begin{proof}
	First, \ref{closed points K-trivial 1} implies \ref{closed points K-trivial 3}, because $S$ being $\KK$-trivial implies, by the Lang-Nishimura lemma, that $S(\KK)\ne\varnothing$. Since $S(\kk)=\varnothing$, we get a quadratic point on $S$. Next, \ref{closed points K-trivial 3} implies \ref{closed points K-trivial 2}, because $\indexx(S)\ne 1$ by Remark \ref{rem: index 1 equivalent to rationality}. Finally, \ref{closed points K-trivial 2} implies \ref{closed points K-trivial 1}. Indeed, if $\indexx(S)=2$ then $S$ has a point whose degree is not divisible by 3. But then the same holds for $X$, hence $X\simeq\PP_\KK^2$ by Proposition~\ref{prop: SB closed points}.
\end{proof}

\begin{prop}\label{prop: index 3}
	Let $S$ be a $G$-del Pezzo surface of degree 6 over $\kk$, and assume that $S$ is not $\kk$-rational. The following statements are equivalent.
	\begin{enumerate}
		\item\label{closed points L-trivial 1} $S$ is $\LL$-trivial.
		\item\label{closed points L-trivial 2} $\indexx(S)=3$.
		\item\label{closed points L-trivial 3} $S$ has a degree 3 point.
	\end{enumerate}
\end{prop}
\begin{proof}
	If $G\simeq\Sym_3$, the vertices of $\Sigma$ form two
	closed points of degree~$3$. Since $S(\kk)=\varnothing$,
	Remark~\ref{rem: index 1 equivalent to rationality}
	gives $\indexx(S)=3$, so all three conditions hold.
	We may therefore assume that $G\simeq\ZZ/6\ZZ$
	or $G\simeq\Dih_6$.
	First, \ref{closed points L-trivial 1} implies \ref{closed points L-trivial 3}. Indeed, suppose that $S_\LL$ is birational to $Y=\PP_\LL^1\times\PP_\LL^1$. Then $S(\LL)\ne\varnothing$ by the Lang-Nishimura lemma and hence in the case $\Gal(\FF/\kk)\simeq\ZZ/6\ZZ$ the $g$-orbit of an $\LL$-rational point gives a point of degree 3 on $S$, as $S(\kk)=\varnothing$. Assume that $\Gal(\FF/\kk)\simeq\Dih_6$. In this case, we consider the birational contraction $S_\EE\to Z$ over the field $\EE=\FF^{\langle h,f\rangle}\subset\LL$, see Section \ref{subsec: SB data}. Then $Z$ is a del Pezzo surface of degree 8 with $\rk\Pic_\EE(Z)=1$ such that $Z_\LL\simeq\PP_\LL^1\times\PP_\LL^1$ and by Proposition~\ref{prop: biregular classification of dP8} one has $Z\simeq\Weil_{\LL/\EE}\PP_{\LL}^1$, see \cite[Lemma 7.3]{ShramovVologodskyPointless} for details. Therefore, $Z(\EE)\ne\varnothing$ and thus $S(\EE)\ne\varnothing$. Now the $g$-orbit of an $\EE$-rational point gives a 3-point on $S$.
	Next, \ref{closed points L-trivial 3} implies \ref{closed points L-trivial 2}, because $\indexx(S)\ne 1$ by Remark \ref{rem: index 1 equivalent to rationality}. Finally, \ref{closed points L-trivial 2} implies \ref{closed points L-trivial 1}. Indeed, if $\indexx(S)=3$ then $S$ has a point of odd degree, hence the same is true for $Y$. But then $Y(\LL)\ne\varnothing$ by Springer's theorem \cite{Springer}.
\end{proof}

Finally, we have the following property. 

\begin{prop}\label{prop: index 6}
	Let $S$ be a $G$-del Pezzo surface of degree 6 over $\kk$, and assume that $S$ is not $\kk$-rational. The following statements are equivalent.
	\begin{enumerate}
		\item\label{closed points L,K-nontrivial 1} $S$ is neither $\KK$-trivial, nor $\LL$-trivial.
		\item\label{closed points L,K-nontrivial 2} $\indexx(S)=6$.
		\item\label{closed points L,K-nontrivial rigidity} $S$ is birationally super-rigid, i.e. $\Bir_\kk(S)=\Aut_\kk(S)$.
	\end{enumerate}
\end{prop}
\begin{proof}
The equivalence of \ref{closed points L,K-nontrivial 1} and \ref{closed points L,K-nontrivial 2} follows from Propositions \ref{prop: index 2} and \ref{prop: index 3}, and the fact that $\indexx(S)$ divides 6. Their equivalence to \ref{closed points L,K-nontrivial rigidity} follows from the Sarkisov program for $S$.
\end{proof}

\subsection{Explicit parametrization}

Now we want to look at $d$-points on $S$. For their image under~$\varphi$ we can use the coordinates $(\lambda_1,\lambda_2)$ defined over some field that lies over $\SplittingHex$. This point then corresponds to $([1:\lambda_1:\lambda_2],[1:\lambda_1^{-1}:\lambda_2^{-1}])$. 

\begin{Notation}
	In what follows, we fix the twisted action $\varphi\colon S_{\SplittingHex}\iso M$ as in \eqref{eq: twisted action}. Further, $\EE$ will denote the splitting field of a degree $d$ point on $S$. If $\EE\nsubseteq\SplittingHex$ then the generators of  $\Gal(\EE/\kk)$ of order 3 and 2 will be denoted $w$ and $t$, respectively, as in Remark \ref{rem: Galois group of the composite}. We will consider the $\varphi$-image of a $d$-point $p$ on $S$, whose geometric components are denoted $p_1,\ldots,p_d$, $d\in\{2,3\}$. By a slight abuse of notation, we mean that $\varphi$ is naturally extended to the composite field $\FF\EE$ when $\EE\nsubseteq\SplittingHex$. We set $\varphi(p_1)=(\lambda_1,\lambda_2)$, where $\lambda_1,\lambda_2\in\SplittingHex\EE^*$. When $d=3$, $[\EE:\kk]=6$, and $\EE\nsubseteq\SplittingHex$, we will assume that $t$ is chosen so that it fixes $p_1$.
\end{Notation}

\begin{prop}\label{prop: Z6 closed points}
    Let $S$ be a non-$\kk$-rational sextic $\ZZ/6\ZZ$-del Pezzo surface. 
    \begin{enumerate}[leftmargin=*, labelindent=20pt, itemsep=5pt]
        \item\label{Z6 2-points} If $S$ is $\KK$-trivial, then there is a $2$-point on $S$. If a 2-point splits over a subfield of $\SplittingHex$, then this field must be $\KK=\SplittingHex^g$ and the $\varphi$-image of this point is given by the $\alpha_h\circ h$-orbit of the point $(\lambda,\lambda g(\lambda))$, where $\lambda\in \SplittingHex^*$, $\Norm_g(\lambda)=1$. 
        
        The $\varphi$-image of $2$-points over another quadratic splitting field $\EE\nsubseteq\SplittingHex$ is given by the $t$-orbit of the point $(\lambda,\lambda g(\lambda))$, where $\lambda \in \SplittingHex \EE^*$, $\ \Norm_g(\lambda) = 1$,  $\Norm_h(\lambda) = \rho$.
        
        \item\label{Z6 3-points} If $S$ is $\LL$-trivial, then there is a $3$-point on $S$. If a 3-point splits over a subfield of $\SplittingHex$, then this field must be $\LL=\SplittingHex^h$ and the $\varphi$-image of this point is given by the $\alpha_g\circ g$-orbit of $(\lambda_1,\lambda_2)$, where $\lambda_1,\lambda_2 \in \SplittingHex^*$, $\Norm_h(\lambda_1)=\Norm_h(\lambda_2)=1.$
        
        The $\varphi$-image of $3$-points with some other splitting field $\EE\nsubseteq\SplittingHex$ (of degree $3$ or $6$) is given by the $w$-orbit of the point $(\lambda,\lambda g(\lambda))$, where $\lambda \in \SplittingHex\EE^*$, $\Norm_g(\lambda)=\xi^{-1}$, $\Norm_h(\lambda)=1$. If $\EE/\kk$ is of degree $6$, then we have two more conditions.  If $ \SplittingHex \cap \EE = \kk$, then $t(\lambda)=\lambda$. If $\SplittingHex \cap \EE = \KK$, then the second and third components of $\varphi(p)$ are permuted by $\alpha_h \circ h$.
        \item\label{Z6 index 6} If $S$ is neither $\KK$-trivial, nor $\LL$-trivial, then $\indexx(S)=6$.
    \end{enumerate}
\end{prop}
\begin{proof}
	If $p$ is a 2-point (respectively, a 3-point) split over $\KK=\SplittingHex^g$ (respectively, over $\LL=\SplittingHex^h$), then $q=\varphi(p)$ satisfies
	\[
	q=\varphi\circ u\circ \varphi^{-1}(q)=\alpha_u u(q),
	\]
	where $u=g$ for 2-points and $u=h$ for 3-points. So, we start by looking at pairs $(\lambda_1,\lambda_2)$ which are fixed by either $\alpha_g\circ g$, or $\alpha_h\circ h$. We have
    \begin{equation}\label{eq: Z6 closed points g}
    (\lambda_1,\lambda_2) = \alpha_g (g (\lambda_1,\lambda_2)) = \left (\frac{1}{\xi g(\lambda_2)}, \frac{g(\lambda_1)}{\xi g(\lambda_2)} \right )
    \end{equation}
    if and only if $\lambda_2 = \lambda_1 g(\lambda_1)$ and $\Norm_g(\lambda_1^{-1}) = \xi$. We have 
    \begin{equation}\label{eq: Z6 closed points h}
    (\lambda_1,\lambda_2) = \alpha_h (h(\lambda_1,\lambda_2)) = \left (\frac{\rho}{h(\lambda_1)}, \frac{\rho g(\rho)}{h(\lambda_2)} \right )
    \end{equation}
    if and only if $\Norm_h(\lambda_1) = \rho$, $\Norm_h(\lambda_2) = \rho g(\rho)$. We now consider 3 cases from the statement.    
    
    \ref{Z6 2-points} 
    A $\KK$-rational point on the Severi-Brauer surface  gives rise to a $2$-point on $S$. Let $p=\{p_1,p_2\}$ be a $2$-point on $S$ that splits over $\KK$ and put $(\lambda_1,\lambda_2) = \varphi(p_1)$. Since $p_1$ is fixed by $g$, the condition \eqref{eq: Z6 closed points g} holds, hence we get, using $\xi=1$, that $\lambda_1=\lambda \in \SplittingHex^*,$ $\Norm_g(\lambda)=1$, $\lambda_2 = \lambda g(\lambda)$. On the other hand, $p_2=h(p_1)$ and thus \eqref{eq: Z6 closed points h} implies $\varphi(p_2) = \alpha_h (h(p_1)) = (\rho h(\lambda^{-1}),\rho h(\lambda^{-1})g(\rho h(\lambda^{-1})))$. 
    
    Now let $p$ be a $2$-point that does not split over $\KK$. Then its splitting field $\EE$ is quadratic extension of $\kk$ and $\Gal(\EE/\kk)\simeq\ZZ/2$. So, $(\lambda_1,\lambda_2)=\varphi(p_1)$ is invariant under $\alpha_g\circ g$ and $\alpha_h\circ h$, so $\Norm_g(\lambda_1) =1$, $\Norm_h(\lambda_1) = \rho$, and $\lambda_2 = \lambda_1 g(\lambda_1)$, where $\lambda_1 \in \SplittingHex\EE$. Note that $p_2=t(p_1)$, but since $t$ acts trivially on $\SplittingHex$ and $\varphi$ is defined over $\SplittingHex$, we have that the action of $t$ is not twisted by $\varphi$.     
    
    \ref{Z6 3-points} 
    An $\LL$-rational point on an involution surface gives rise to a $3$-point on $S$. Let $p=\{p_1,p_2,p_3\}$ be a $3$-point, and as before $p_1=(\lambda_1,\lambda_2)$. Since $p_1$ is fixed by $\alpha_h\circ h$, we can use the condition \eqref{eq: Z6 closed points h} and $\rho=1$ to deduce that $\Norm_h(\lambda_1)=\Norm_h(\lambda_2)=1$. The other two components $p_2$ and $p_3$ are represented by the image of $(\lambda_1,\lambda_2)$ under $\alpha_g\circ g$ and $\alpha_{g^2}\circ g^2$ which is exactly what is stated in the Proposition. 
    
    Now let $p$ be a 3-point that does not split over $\LL$, and let $\EE$ be its splitting field. Then $\Gal(\EE/\kk)\simeq\ZZ/3\ZZ$ or $\Gal(\EE/\kk)\simeq\Sym_3$. In the first case, we have either $\EE\cap\FF=\EE$ or $\EE\cap\FF=\kk$, but the former case forces $\EE=\LL$, which was already discussed above. So, we assume $\EE\cap\FF=\kk$ and let $w$ be a generator of $\Gal(\EE/\kk)$. As before, letting $(\lambda_1,\lambda_2)=\varphi(p_1)$, we notice that $g$ and $h$ act trivially on $p$, so \eqref{eq: Z6 closed points g} gives $\lambda_1=\lambda$, $\lambda_2=\lambda g(\lambda)$ and $\Norm_g(\lambda)=\xi^{-1}$, while \eqref{eq: Z6 closed points h} gives $\Norm_h(\lambda)=1$ (recall that $\rho=1$). Taking a $w$-orbit of $p_1$, we get the 3-point from the statement.
    
    In the second case $\Gal(\EE/\kk)\simeq\Sym_3$, let $w$ be a generator of order 3 and $t$ be a generator of order 2. If $\FF\cap\EE=\kk$ then we argue as in the previous paragraph to get that a 3-point is of the said form (note that $g$ and $h$ act trivially on $p$), and choose $t$ so that it fixes~$p_1$, and permutes $p_2$ and $p_3$. Suppose that $\FF\cap\EE=\KK=\FF^g$. As was mentioned in Remark \ref{rem: Galois group of the composite}, the group $\Gal(\FF\EE/\kk)$, viewed as a subgroup of $\Gal(\FF/\kk)\times\Gal(\EE/\kk)$, is generated by $(g,\id)$, $(\id,w)$ and $(h,t)$. So, let $\varphi(p_1)=(\lambda_1,\lambda_2)$, where $\lambda_1,\lambda_2\in\FF\EE^*$. Note that $(\lambda_1,\lambda_2)$ is $(g,\id)$-invariant, which gives $\lambda_1=\lambda$, $\lambda_2=\lambda g(\lambda)$ and $\Norm_g(\lambda)=\xi^{-1}$ as before. As we choose $t$ to preserve $p_1$ and permute $p_2$ with $p_3$, we find that $(\lambda,\lambda g(\lambda))$ is $\alpha_h\circ h$-invariant, hence $\Norm_h(\lambda)=1$. By contrast, $p_2$ is sent onto $p_3$ by $(h,t)$, hence the claim.
    
    \ref{Z6 index 6} Follows from Proposition \ref{prop: index 6}.
\end{proof}

\begin{prop}\label{prop: S3 closed points}
    Let $S$ be a non-$\kk$-rational sextic $\Sym_3$-del Pezzo surface. Then the six vertices of $\Sigma$ form two $3$-points, both with splitting field $\SplittingHex$.
    Every other $3$-point lies in $S\setminus\Sigma$;
    we describe these points below. If such point splits over a subfield of $\SplittingHex$, then it must be $\SplittingHex$ itself. The $\varphi$-image of such 3-point is given by the $\alpha_g\circ g$-orbit of the point $\left (\lambda, f(\lambda^{-1})\xi^{-1} \right)$, where $\lambda\in\SplittingHex^*$. 
    
    The $\varphi$-image of 3-points with some other splitting field $\EE\nsubseteq\SplittingHex$ $($of degree $3$ or $6$$)$ is given by the $w$-orbit of the point $(\lambda,\lambda g(\lambda))$, where $w$ is an order $3$ element in $\Gal(\EE/\kk)$, $\lambda \in (\SplittingHex \EE^*)^{gf}$, $\Norm_g(\lambda)=\xi^{-1}$.
    
    If $\EE/\kk$ is of degree $6$ $($thus the Galois group is $\Sym_3$$)$, then we have two more conditions.  If $ \SplittingHex \cap \EE = \kk$, then $t(\lambda)=\lambda$. If $\SplittingHex \cap \EE = \KK$, then the second and third components are permuted by $\alpha_f \circ f$.
\end{prop}
\begin{proof}
    Any three vertices at distance 2 from each other in the graph $\Sigma$ give a $3$-point on $S$, hence $\indexx(S)=3$ by Proposition~\ref{prop: index 3}. Let $\EE$ be the splitting field of a 3-point on $S$. Then one of the following holds:
    \begin{enumerate}[label=\alph*)]
    	\item $\EE\cap\FF=\FF$, hence $\EE=\FF$.
    	\item $\EE\cap\FF=\kk$.
    	\item $\EE\cap\FF$ is quadratic over $\kk$, i.e. $\EE\cap\FF=\FF^g=\KK$. Then $\Gal(\EE/\kk)\simeq\Sym_3$.
    	\item $\EE\cap\FF$ is cubic over $\kk$. This case is not possible since $\EE\cap\FF$ must be Galois over $\kk$.
    \end{enumerate}    
    Suppose that we are in case (a). A 3-point that splits over $\SplittingHex$ has one component fixed by~$f$. Thus we solve $(\lambda_1,\lambda_2) = \alpha_f\circ f(\lambda_1,\lambda_2)$ and we get \[
    (\lambda_1,\lambda_2) = \alpha_f\circ f(\lambda_1,\lambda_2) = \left (\frac{1}{\xi f(\lambda_2)}, \frac{1}{\xi f(\lambda_1)}\right )
    \] 
    which is equivalent to $\lambda_2 = \xi^{-1}f(\lambda_1^{-1})$; we set $\lambda_1=\lambda$. Calculating further the $\alpha_g\circ g$-orbit of this pair, we get the points from the statement.     
    
    Now assume that there is a $3$-point that splits over a field $\EE\ne\SplittingHex$, i.e. we are in one of the cases (b) or (c), so either $\SplittingHex \cap \EE = \kk$ or $\SplittingHex \cap \EE = \SplittingHex^g=\KK$ and $\Gal(\EE/\kk)\simeq \Sym_3$. As before, we let $w$ be an order 3 generator in $\Gal(\EE/\kk)$ and $t$ be an order 2 generator, if exists. If $\EE\cap\SplittingHex=\kk$ then we argue as in $\ZZ/6$ case, taking $t\in\Gal(\EE/\kk)$ a transposition that fixes $p_1$. Suppose that $\FF\cap\EE=\KK$. By Remark \ref{rem: Galois group of the composite}, the group $\Gal(\FF\EE/\kk)$, viewed as a subgroup of $\Gal(\FF/\kk)\times\Gal(\EE/\kk)$, is generated by $(g,\id)$, $(\id,w)$ and $(f,t)$. The proof finishes as in Proposition \ref{prop: Z6 closed points} by noticing that 
    \[
    \left (\lambda,\frac{1}{f(\lambda)\xi} \right )=\alpha_g\circ g\left (\lambda,\frac{1}{f(\lambda)\xi} \right )=(gf(\lambda),g(\lambda)gf(\lambda)),
    \]
    is equivalent to $gf(\lambda) = \lambda$, $\Norm_g(\lambda) = \xi^{-1}$.
\end{proof}

\begin{prop}\label{prop: D6 closed points}
    Let $S$ be a non-$\kk$-rational sextic $\Dih_6$-del Pezzo surface. Then one of the following holds. 
    \begin{enumerate}[leftmargin=*, labelindent=20pt, itemsep=5pt]
        \item \label{D6 2-points} If $S$ is $\KK$-trivial then there is a $2$-point on $S$. If a 2-point splits over a subfield of $\SplittingHex$, then this field is either $\KK=\SplittingHex^{\langle g,s\rangle}$, or $\SplittingHex^{\langle g,f\rangle}$, i.e. there are no $2$-points that split over $\SplittingHex^{\langle g,h\rangle}$. The $\varphi$-image of all $2$-points split over $\KK$ is given by the $\alpha_h\circ h$-orbit of the point $(\lambda,\lambda g(\lambda))$, where $\Norm_g(\lambda)=1$, $\Norm_{gs}(\lambda) = \rho$. The $\varphi$-image of all $2$-points split over $\SplittingHex^{\langle g,f\rangle}$ is given by the $\alpha_h\circ h$-orbit of the point $(\lambda,\lambda g(\lambda))$, where $\Norm_g(\lambda)=1$, $\lambda\in\SplittingHex^{gf}$. 
        
        The $\varphi$-image of 2-points with some other quadratic splitting field $\EE\nsubseteq\SplittingHex$ is given by the $t$-orbit of $(\lambda,\lambda g(\lambda))$, where $\lambda \in\SplittingHex \EE^*$, $\Norm_g(\lambda) = 1$, $\Norm_h(\lambda)=\rho$, $\Norm_{gs}(\lambda)=\rho$. 
        
        \item \label{D6 3-points} If $S$ is $\LL$-trivial, then there is a $3$-point on $S$. If a 3-point splits over a subfield of $\SplittingHex$, then this field must be $\SplittingHex^h$. The $\varphi$-image of this point is given by the $\alpha_g\circ g$-orbit of $(\lambda,f(\lambda^{-1})\xi^{-1})$, where $\lambda\in\SplittingHex^*$ and $\Norm_h(\lambda)=1$.
        
        If a 3-point splits over some other field $\EE\nsubseteq\SplittingHex$, then its $\varphi$-image is given by the $w$-orbit of $(\lambda,\lambda g(\lambda))$, where $\lambda\in\SplittingHex\EE^*$, $\Norm_g(\lambda)=\xi^{-1}$, $\rho=\Norm_h(\lambda)=\Norm_{gs}(\lambda) = 1$. If $\EE/\kk$ is of degree 6, then we have two more conditions. If $\SplittingHex \cap \EE = \kk$ then $t(\lambda)=\lambda$. If $\SplittingHex\cap \EE$ is of degree $2$ over~$\kk$ with the Galois group generated by $v \in \{ f,s,h\}$ then the second and third components are permuted by $\alpha_v \circ v$.
        \item If $S$ is neither $\KK$-trivial, nor $\LL$-trivial, then $\indexx(S)=6$.
    \end{enumerate}
\end{prop}
\begin{proof}
	
\ref{D6 2-points} We assume that $\xi = 1$. If $\EE/\kk$ is a splitting field of a point of degree 2, then either $\EE\cap\FF=\kk$, or $\EE\subset\FF$ and in this case $\EE=\FF^{\langle g,u\rangle}$, where $u\in\{h,f,s\}$. Suppose that we are in the latter case, let $p=\{p_1,p_2\}$ be a 2-point split over $\EE$ and let $\varphi(p_1)=(\lambda_1,\lambda_2)$. Then, as before, we notice that $p_1$ is $\alpha_g\circ g$-invariant, which gives the conditions $\lambda_1=\lambda$, $\lambda_2=\lambda g(\lambda)$, $\Norm_g(\lambda)=1$. Suppose that $u=f$. Then $\alpha_f\circ f$-invariance translates to 
\[
\lambda g(\lambda)f(\lambda)=1
\]
which implies (since $\Norm_g(\lambda)=1$) that $f(\lambda)=g^2(\lambda)$, hence $\lambda\in\FF^{gf}$. If $u=s$ then, using parametrization of Proposition \ref{prop: Second D6 param} with $\xi=1$, we get
\[
\alpha_s=\alpha_h h(\alpha_f)=((\rho,\rho g(\rho)),\epsilon(hf)).
\]
Therefore, the $\alpha_s\circ s$-invariance of $(\lambda,\lambda g(\lambda))$ means
\[
(\lambda,\lambda g(\lambda))=(\rho s(\lambda) sg(\lambda), \rho g(\rho)s(\lambda)),
\]
hence $\lambda=\rho s(\lambda)sg(\lambda)$. By using $\Norm_g(s(\lambda))=1$, we notice that two equalities $\Norm_g(\lambda)=1$ and $\lambda=\rho s(\lambda)sg(\lambda)$ are equivalent to $\Norm_g(\lambda)=1$ and $\Norm_{gs}(\lambda)=\rho$.

Finally, the $\alpha_h\circ h$-invariance of $(\lambda,\lambda g(\lambda))$ implies 
\[
(\lambda,\lambda g(\lambda))=(\rho h(\lambda^{-1}),\rho g(\rho)h(\lambda^{-1})hg(\lambda^{-1})),
\]
which is equivalent to $\rho=\Norm_h(\lambda)$. But this implies that $S$ is $\LL$-trivial, hence $\kk$-rational, a contradiction. 

Now, in the former case $\EE\cap\SplittingHex=\kk$, we again let $\varphi(p_1)=(\lambda_1,\lambda_2)$ and notice that this point is $\alpha_u\circ u$-invariant for $u\in\{g,h,f\}$, and is sent to $\varphi(p_2)$ by $t$.

\ref{D6 3-points}: We assume that $\rho = 1$. Let $\EE$ be the splitting field of a 3-point on $S$. Since $\EE\cap\SplittingHex$ is Galois over $\kk$, one of the following holds:
\begin{enumerate}[label=\alph*)]
	\item $\EE\subseteq\FF$, in which case
	$\EE=\FF^{\langle h\rangle}$;
	\item $\EE\cap\FF=\kk$;
	\item $\EE\cap\FF=\FF^{\langle g,u\rangle}$
	for some $u\in\{h,f,s\}$.
\end{enumerate}

In the case (a), we notice that $\varphi(p_1)=(\lambda_1,\lambda_2)$ is $\alpha_h\circ h$-invariant, hence we get the conditions $\Norm_h(\lambda_1)=\Norm_h(\lambda_2)=1$. Assuming that $\alpha_f\circ f$ fixes $p_1$ and swaps $p_2$ with $p_3$, we get $\lambda_1f(\lambda_2)\xi=1$. Now the $\alpha_g\circ g$-orbit of $(\lambda_1,\lambda_2)$ gives a 3-point on $S$.

In the case (b), we notice that $(\lambda_1,\lambda_2)$ is $\alpha_u\circ u$-invariant for $u\in\{g,h,f\}$, hence we get conditions on $\lambda$ from the statement. The 3-point is then generated by the $w$-orbit of $(\lambda,\lambda g(\lambda))$, and we have an extra condition $t(\lambda)=\lambda$ if $[\EE:\kk]=6$. Finally, in the case (c) where $\FF\cap\EE$ is quadratic over~$\kk$, Remark \ref{rem: Galois group of the composite} gives the generators of $\Gal(\EE\FF/\kk)$. Notice that $(\lambda_1,\lambda_2)$ is both $(g,\id)$- and $(u,\id)$-invariant, while $(f,t)$ (for $u=h$) or $(h,t)$ (for $u=s$ and $u=f$) preserves $p_1$ and swaps $p_2$ with $p_3$. Hence the result.      
\end{proof}

\subsection{Closed points in general position} Given $n\leqslant 6$ distinct points $p_1,\ldots,p_n$ on $\PP^2_{\overline{\kk}}$, we
will say that they are in general position if no three of them are contained in a line, and for $n=6$
all six are not contained in a conic. 
If $p\in\PP_{\overline{\kk}}^2$ is a point of degree $d$ in general position, then its blow-up is a del Pezzo surface of degree $9-d$, see for example \cite[Theorem IV.2.6]{ManinCubicForms}. Given a sextic $G$-del Pezzo surface $S$, we will say that a point $p\in S$ of degree $d$ is \emph{in general position} on $S$ if its blow-up is a del Pezzo surface (whose $G$-invariant Picard rank is then 2). An interesting feature of degree 3 and degree 6 points on non-trivial Severi-Brauer surfaces is that they are automatically in general position and hence can serve as a centre of a Sarkisov link \cite[Lemma 2.8]{Shramov1}. Let us investigate the position of points on sextic del Pezzo surfaces.

\begin{prop}\label{prop: position of points}
	Let $S$ be a non-$\kk$-rational sextic $G$-del Pezzo surface. Then one has the following. 
	\begin{enumerate}[leftmargin=*, labelindent=20pt, itemsep=5pt]
		\item\label{general position index 2} Suppose that $\indexx(S)=2$. Then any 2-point on $S$ is in general position. 
		\item\label{general position index 3} Suppose that $\indexx(S)=3$. If $p\in S$ is a 3-point that does not split over $\SplittingHex$, then $p$ is in general position. The two $3$-points formed by the vertices of $\Sigma$
		when $G\simeq\Sym_3$ are not in general position.
		For all other $3$-points splitting over $\SplittingHex$,
		we have the following subcases (in the notation of Propositions \ref{prop: Z6 closed points}, \ref{prop: S3 closed points} and \ref{prop: D6 closed points}, and the restrictions on $\lambda_1$, $\lambda_2$, given therein). 
		\begin{enumerate}[leftmargin=*, labelindent=3pt, itemsep=5pt]
			\item \label{genPos Z6} If $G\simeq\ZZ / 6 \ZZ$, then $p$ is in general position if and only if
			\[ 
			\lambda_2 \neq \lambda_1 g(\lambda_1), \, \lambda_1 \neq \xi \lambda_2 g^2(\lambda_2),\ \xi g(\lambda_1)g^2(\lambda_2)\ne 1.
			\] 
			\item \label{genPos S3 D6} If $G\simeq \Sym_3$ or $G\simeq\Dih_6$, then $p$ is in general position if and only if
			\[ 
			\xi\lambda g(\lambda) f(\lambda)\ne 1,\ \lambda \not \in \SplittingHex^{gf}.
			\]
		\end{enumerate}
	\end{enumerate}
\end{prop}
\begin{proof}	
\ref{general position index 2} If $p$ splits over $\SplittingHex$, then $\varphi(p)$ is the $\alpha_h \circ h$ orbit of $ (\lambda, \lambda g(\lambda))$, $\lambda \in \SplittingHex^*$ with different conditions on $\lambda$ depending on the group $G$, see Propositions \ref{prop: Z6 closed points} and \ref{prop: D6 closed points}. By viewing $S_{\overline{\kk}}$ as the blow-up of $\PP^2_{\overline{\kk}}$, we should study the position of the points 
\[
\left [1:\lambda: \lambda g(\lambda))\right ],\  \left [1: \frac{\rho}{h(\lambda)}:  \frac{\rho}{h(\lambda)} g\left ( \frac{\rho}{h(\lambda)}\right )\right ],\ [1:0:0],\ [0:1:0],\ [0:0:1],
\]
namely we need to check whether any three of them are colinear. The only non-obvious triple to check is the first two points and a coordinate point, which (by applying $\alpha_g \circ g$, if needed) we can assume to be $[1:0:0]$. In that case, the $x_1$- and $x_2$-components of the first two points must satisfy the equation $x_1+ax_2=0$ for some $a\in\FF^*$, which implies $\rho=\Norm_h(\lambda)$ and hence $S$ is $\LL$-trivial, a contradiction.

If $p$ splits over a different field $\EE$, then we have that $\varphi(p)$ is the $t$ orbit of $ (\lambda, \lambda g(\lambda)), \, \lambda \in (\SplittingHex\EE)^*$ with different conditions on $\lambda$ depending on the case of $G$. As in the previous case, we need to examine the position of points
\[
[1:\lambda: \lambda g(\lambda))],\  [1:t(\lambda): t(\lambda) g(t(\lambda)))],\ [1:0:0],\ [0:1:0],\ [0:0:1]
\]
on $\PP^2_{\overline{\kk}}$. As before, we only check that the first three points are not collinear. If they are, then $t(\lambda) = \lambda$, which contradicts to $(\lambda,\lambda g(\lambda))$ having a $t$-orbit of size 2.

\ref{general position index 3} First, we show that a 3-point $p\in S$ cannot lie on a 1-curve $C\subset S$, the preimage of a conic on $\PP_\FF^2$. Suppose that $p\in C$ and consider the blow-down $\pi\colon S\to X$ over $\KK$, where $X$ is a non-trivial Severi-Brauer surface. Note that $C$ passes through the three $(-1)$-curves on $S$ which are blow-down, and $\pi(p)$ is a 3-point on $X$. The conic $\pi(C)$ contains both $\pi(p)$ and the $3$-point blown up by $\pi$, hence six distinct geometric
points. It is unique by B\'ezout's theorem and therefore
invariant under $\Gal(\overline{\kk}/\KK)$.
A Sarkisov $3$-link centred at either of these $3$-points
maps it to a $\KK$-line on $X^{\rm op}$, contradicting
the non-triviality of $X$. Suppose that $p$ splits over $\EE \not \subset \SplittingHex$. Then $\varphi(p)$ is the $w$-orbit of $(\lambda, \lambda g(\lambda))$, $\lambda \in (\SplittingHex\EE)^*$ with different conditions on $\lambda$ depending on $G$. We check that there are no collinear triples among the points
\[
[1:\lambda: \lambda g(\lambda))],\  [1:w(\lambda): w(\lambda) wg(\lambda))],\ [1:w^2(\lambda): w^2(\lambda) w^2g(\lambda))],\ [1:0:0],\ [0:1:0],\ [0:0:1]
\]
As above, we check this for the first two points and $[1:0:0]$. If these points were collinear, then we would have $w(\lambda) = \lambda$ which again gives a contradiction.

Finally, we consider $3$-points which split over $\SplittingHex$. In the case \ref{genPos Z6}, the components of $\varphi(p)$ are 
\[
(\lambda_1, \lambda_2),\ \left (\frac{1}{g(\lambda_2 \xi) }, g\left (\frac{\lambda_1}{\lambda_2 \xi}\right )\right ), \, \left (g^2\left (\frac{\lambda_2}{\lambda_1} \right ), \frac{1}{g^2(\lambda_1\xi)}\right ).
\]
If $x_1+Ax_2=0$ is a line through $[1:0:0]$ and the first two components above, then we mus have $A=-\lambda_1/\lambda_2$ and $\lambda_2=\lambda_1g(\lambda_1)$. Similarly, we find that if $[1:0:0]$ is collinear with the second and the third components, then $A=-1/g(\lambda_1)$ and $\xi g(\lambda_1) g^2(\lambda_2)=1$. If $[1:0:0]$ is collinear with the first and the third components, then $\xi\lambda_2 g^2(\lambda_2)=\lambda_1$.

In the cases \ref{genPos S3 D6}, the components of $\varphi(p)$ are 
\[
\left (\lambda,\frac{1}{f(\lambda)\xi} \right ), \big (gf(\lambda),g(\lambda f(\lambda)) \big), \left (\frac{1}{g^2(\lambda f(\lambda))\xi},\frac{1}{\xi g^2(\lambda)} \right ),
\]
where $\lambda\in\SplittingHex^*$ and $\Norm_h(\lambda) = 1$ when $G\simeq\Dih_6$. As in the previous case, we check when the components of any two of these points simultaneously satisfy the equation $x_1+Ax_2=0$. The result follows.
\end{proof}

\begin{prop}\label{prop: 3-points}
Let $S$ be a sextic $G$-del Pezzo surface of index 3. Then it has $3$-points in general position whose splitting field is a subfield of $\SplittingHex$, namely $\LL=\SplittingHex^h$ when $G\simeq\ZZ/6\ZZ$, or $\FF$ itself when $G\simeq\Sym_3$, or $\LL=\SplittingHex^h$ when $G\simeq\Dih_6$.
\end{prop}

\begin{proof}

We consider different possibilities for $G$, following the notation of Proposition \ref{prop: position of points}.

\begin{itemize}
\item If $G\simeq \ZZ / 6 \ZZ$, then we look for $\lambda_1,\lambda_2 \in \SplittingHex^*$ such that $\Norm_h(\lambda_1)=\Norm_h(\lambda_2)=1$ and the condition of \ref{genPos Z6} is satisfied. Choose $\lambda_1 = b/h(b)$, where  $b \in (\SplittingHex^g)^*$, and $\lambda_2 = \xi^{-1}$ (recall that we may assume $\rho=1$ in \eqref{eq: Z6 coefficient conditions} for index 3 surfaces). We therefore get the inequalities:
\[ 
\xi^{-1} \neq \frac{b^2}{h(b^2)},\ \frac{b}{h(b)} \neq \xi^{-1},\ \frac{b}{h(b)}\neq 1.
\]
To satisfy the third one, it is enough to choose $b \notin L_1=\kk^*$. Since $\Norm_h(\xi^{-1}) = 1$ and $\xi \in \SplittingHex^g$, Theorem \ref{thm: Hilbert 90} implies that there exists $a \in (\SplittingHex^g)^*$ such that $a/h(a) = \xi^{-1}$. Note that for any other $a' \in  \SplittingHex^g$ such that $a'/h(a')=\xi^{-1}$ one has $a/a'\in \SplittingHex^{g,h} = \kk$, so all elements $b\in\SplittingHex^g$ such that $b/h(b)=\xi^{-1}$ lie in $1$-dimensional subspace $L_2\subset \SplittingHex^g$.

If $\xi^{-1}$ is not a square in $\SplittingHex^g$, then there are no $b$ satisfying the equality $\xi^{-1}=b^2/h(b^2)$ and we set $L_3=\varnothing$. Suppose that $\xi^{-1}=c^2$ for some $c\in\SplittingHex^g$. Then for all $b\in\SplittingHex^g$ satisfying $c^2=b^2/h(b^2)$ one has $b/h(b)=\pm c$. As above, the equality $b_1/h(b_1)=c=b_2/h(b_2)$ implies $b_1/b_2\in\SplittingHex^{g,h}=\kk$ and we conclude that the set $L_3$ of all $b$ satisfying the condition $\xi^{-1}=b^2/h(b^2)$ is of dimension at most~1. To sum up, since $\SplittingHex^g$ is of dimension $2$ over $\kk$, the set $\SplittingHex^g\setminus\{L_1\cup L_2\cup L_3\}$ is not empty, hence the existence of $b$ and $\lambda_1$, $\lambda_2$ satisfying the desired restrictions.

\item  If $G\simeq \Sym_3$, then we take $\lambda \in \SplittingHex^* \setminus \SplittingHex^{gf}$ and we only need to check that $\xi^{-1} \neq \lambda g(\lambda) f(\lambda)$. But if one has $\xi^{-1}=\lambda g(\lambda) f(\lambda)$, then the $f$-invariance of $\xi$ implies that $\lambda f(\lambda) fg(\lambda) = f(\lambda f(\lambda) g(\lambda)) = \lambda f(\lambda) g(\lambda) $ which is equivalent to $gf(\lambda) = \lambda$, a contradiction. 

\item If $G\simeq \Dih_6$, then we look for the same $\lambda$ as in the $\Sym_3$-case, with an additional property $\Norm_h(\lambda) = 1$. Take $a \in \SplittingHex$ and put $\lambda =a/h(a)$. The property $\lambda \neq gf(\lambda)$ is equivalent to $\Norm_{gs}(a)  \notin \SplittingHex^h$, so if the latter holds for $a$ then we are done. Suppose that $\Norm_{gs}(a) \in \SplittingHex^h$. Then take $b=a + 1$. If $\Norm_{gs}(b)\notin\SplittingHex^h$, then we choose $\lambda=b/h(b)$. Otherwise, we notice that 
\[
\Norm_{gs}(b)=\Norm_{gs}(a) + 1 + (a + gs(a)) = h(\Norm_{gs}(b)) = \Norm_{gs}(a) + 1 + h(a + gs(a)),
\] 
and therefore
\[
a + \frac{\Norm_{gs}(a)}{a} = h(a) + \frac{\Norm_{gs}(a)}{h(a)}
\]
which gives 
\[
\Norm_h(a) (h(a)-a)=\Norm_{gs}(a) (h(a)-a ).
\]
by multiplying with $N_h(a)$ and rearranging. \\
Thus, either $a \in \SplittingHex^h$ or $\Norm_h(a) = \Norm_{gs}(a)$. The latter is equivalent to $gf(a) = a$, hence it is enough to choose $a \in \SplittingHex \setminus (\SplittingHex^h \cup \SplittingHex^{gf}) \ne \varnothing$. 
\end{itemize}
\end{proof}

\begin{prop}\label{prop: 2-points}
	Let $S$ be a sextic $G$-del Pezzo surface of index 2. Then it has $2$-points in general position whose splitting field is a quadratic subfield of $\SplittingHex$, namely $\KK=\SplittingHex^g$ when $G\simeq\ZZ/6\ZZ$, or $\SplittingHex^{\langle g,f\rangle}$ and $\KK=\SplittingHex^{\langle g,s\rangle}$ when $G\simeq\Dih_6$.
\end{prop}
\begin{proof}
We again consider different possibilities for $G$, and use the notation of Proposition \ref{prop: Z6 closed points}. If $G\simeq \ZZ / 6 \ZZ$, then we can choose $\lambda = 1$. The case $G\simeq\Sym_3$ is not possible. Suppose that $G\simeq\Dih_6$. 

Recall that we have two options for the splitting field $\EE$ of such a $2$-point. If we want $\EE=\SplittingHex^{\langle g,f\rangle}$, then we can take $\lambda = 1$.
If we want $\EE = \KK$, then we need $\lambda \in \SplittingHex$ such that $\Norm_g(\lambda) = 1$, $\Norm_{gs}(\lambda) = \rho$. Let $\lambda = g(\rho^{-1})$. Since, by Proposition \ref{prop: Second D6 param}, one has $\Norm_g(\rho) = 1$ and $g(\rho)\Norm_f(\rho)=1$, these properties are indeed satisfied. 
\end{proof}

We will see in the last Section, that these two propositions guarantee the existence of self-links; this will be used in our construction of homomorphisms.

\subsection{A surface with many splitting fields of 3-points} In the following example, we construct, by using the explicit cocycle parametrization of sextic del Pezzo surfaces from above, an $\Sym_3$-surface that contains an uncountable number of 3-points in general position with pairwise distinct splitting fields. This example will later demonstrate that the surfaces mentioned in Theorem~\ref{thm: pliability} do exist.

\begin{ex}\label{example: main example}
Let $\SplittingHex=\CC(t_1,t_2,t_3,s)$ and $G \subset \Aut(\SplittingHex)$ be the subgroup which acts as $\Sym_3$ on the variables $\{t_1,t_2,t_3\}$ and fixes $s$. We let $g = (123)$ and $f=(23)$ be the generators of $G$. Set $\kk=\FF^G$.

\begin{itemize}[leftmargin=*, labelindent=3pt, itemsep=5pt]
	\item First, we claim that $\xi=s$ defines a non-$\kk$-rational sextic $\Sym_3$-del Pezzo surface $S$ over $\kk$. 
	
	Indeed, $s \in \kk^*$, hence by Propositions \ref{prop: S3 final parametrization} it is enough to prove that $s\notin \Norm_g(\SplittingHex^*)$ to show that $S$ is not rational over $\kk$. Assume that there are $P,Q \in\CC[t_1,t_2,t_3,s]$ such that 
	\[
	s = \Norm_g\left (\frac{P}{Q}\right ) = \frac{P(t_1,t_2,t_3,s) P(t_2,t_3,t_1,s) P(t_3,t_1,t_2,s)}{Q(t_1,t_2,t_3,s)Q(t_2,t_3,t_1,s) Q(t_3,t_1,t_2,s)}.
	\] 
	Multiplying both sides of the equation by the denominator of the fraction on the right, and considering the resulting expressions on both sides as polynomials in $s$ with coefficients in $\CC[t_1,t_2,t_3]$, we see that the highest degree of $s$ on the right is divisible by 3, whereas on the left it is not. This is a contradiction.
	
	\item Next, for every $z\in\CC$, define 
	\[
	\EE_z=\kk\big (\sqrt[3]{s(t_1 + z)(t_2 + z)(t_3 + z)}\big ),\ \  \lambda_z= \frac{\sqrt[3]{s(t_1 + z)(t_2 + z)(t_3 + z)}}{t_3 + z}
	\] 
	We claim that $\EE_z$ is a cubic Galois extension of $\kk$ and $\lambda_z^{-1}$ yields a $3$-point on $S$ in general position.
	
	Since $s(t_1 + z)(t_2 + z)(t_3 + z)$ is $G$-invariant, it belongs to $\kk^*$. Notice that $s(t_1+z)(t_2+z)(t_3+z)$ is not a cube in $\kk$. Indeed, assuming otherwise and comparing degrees in $s$ gives a contradiction. Since $\kk^*$ contains a third root of unity, we immediately see that $\EE_z$ is of degree $3$ over $\kk$ and is Galois. To prove that $\lambda_z^{-1}$ yields a $3$-point we need, according to Proposition \ref{prop: S3 closed points}, to prove that
	\[ 
	\lambda_z \in (\EE_z \SplittingHex)^{gf},\ \ \Norm_g(\lambda_z)=s.
	\]
	The invariance of $\lambda_z$ under $gf=(12)$ is obvious. We further compute
	\[
	\Norm_g(\lambda_z) = \frac{\sqrt[3]{s(t_1 + z)(t_2 + z)(t_3 + z)}^3}{\Norm_g(t_3 + z)} = \frac{s(t_1 + z)(t_2 + z)(t_3 + z)}{(t_1 + z)(t_2 + z)(t_3 + z)} = s,
	\] 
	as required. Notice that the resulting 3-point is in general position on $S$ by Proposition \ref{prop: position of points}, since $\EE_z$ is not a subfield of $\FF$.
	
	\item Finally, we claim that for $v \neq z$ one has $\EE_z \neq \EE_v$ inside of a fixed algebraic closure of $\kk$. 
	It is enough to prove that $\sqrt[3]{s(t_1 + z)(t_2 + z)(t_3 + z)} \not\in \EE_v$. 
	Assume the converse holds, then there are $a,b,c\in\kk$ such that:
	\begin{equation}\label{eq: example of S3-surface coefficients}
	s(t_1 + z)(t_2 + z)(t_3 + z) = \left  (a + b\sqrt[3]{s(t_1 + v)(t_2 + v)(t_3 + v)} + c\sqrt[3]{s^2(t_1 + v)^2(t_2 + v)^2(t_3 + v)^2} \right)^3.
	\end{equation}
	Setting $\mu_v = s(t_1 + v)(t_2 + v)(t_3 + v)$, the right hand side becomes
	\[
	(a^3 + b^3 \mu_v + c^3 \mu_v^2 + 6abc\mu_v) + (3a^2b + 3c^2a\mu_v + 3b^2c\mu_v)\sqrt[3]{\mu_v} + (3 ab^2 + 3a^2c + 3bc^2 \mu_v) \sqrt[3]{\mu_v^2}.
	\]
	Now on the left hand side of \eqref{eq: example of S3-surface coefficients} we have an element of the field $\kk$, while on the right hand side a polynomial from $\kk[\sqrt[3]{\mu_v}]$. This gives us the three equations
	\begin{align}
		& a^3 + b^3 \mu_v + c^3 \mu_v^2 + 6abc\mu_v = s(t_1 + z)(t_2 + z)(t_3 + z) \label{prop: pliab eq1} \\
		& 3a^2b + 3c^2a\mu_v + 3b^2c\mu_v = 0 \label{prop: pliab eq2} \\
		& 3 ab^2 + 3a^2c + 3bc^2 \mu_v = 0 \label{prop: pliab eq3}
	\end{align}
	Let us first prove that $a,b,c\neq 0$. Suppose that $a=0$. Then
	\eqref{prop: pliab eq3} implies that $bc=0$. If $b=0$, then
	\eqref{prop: pliab eq1} gives
	$
	c^3\mu_v^2=\mu_z,
	$
	which is impossible by comparing degrees in $s$ modulo $3$. If $c=0$,
	then
	$
	b^3\mu_v=\mu_z.
	$
	After cancelling $s$, this becomes
	\[
	b^3(t_1+v)(t_2+v)(t_3+v)
	=(t_1+z)(t_2+z)(t_3+z),
	\]
	which is impossible: for $v\neq z$, the valuation along
	$t_1+z=0$ is divisible by $3$ on the left-hand side
	and equals $1$ on the right-hand side.
	Thus $a\neq0$.
	
	Now, if $b=0$, then \eqref{prop: pliab eq2} and $a\neq0$ imply that $c=0$. Equation~\eqref{prop: pliab eq1} then gives $a^3=\mu_z$, which is impossible by comparing degrees in $s$.
	Similarly, if $c=0$, then \eqref{prop: pliab eq2} implies that $b=0$,
	and we obtain the same contradiction. Hence $a,b,c\neq0$.
	
	Multiplying \eqref{prop: pliab eq2} by $b$ and
	\eqref{prop: pliab eq3} by $a$, we obtain
	\[
	3a^2b^2+3abc^2\mu_v+3b^3c\mu_v=3a^2b^2+3a^3c+3abc^2\mu_v=0,
	\]
	which gives $b^3\mu_v=a^3$. 
	This is impossible by comparing degrees in $s$. Therefore, $\EE_z\neq\EE_v$.	
\end{itemize}
\end{ex}

\section{Sarkisov links and birational classification}\label{sec: birational}

Let $S$ be a del Pezzo surface of degree 6 over $\kk$ with $\Pic(S)\simeq\ZZ$, and $\phi\colon S\dashrightarrow S'$ be a birational map to another Mori fibre space $S'$. By \cite[Theorem 2.6]{Isk1996}, the map $\phi$ is a composition of automorphisms and Sarkisov links, and every link starting from $S$ is represented by the following diagram:
\begin{equation}\label{eq: Sarkisov link of type II bis}
	\xymatrix{
		&T\ar@{->}[dl]_{\eta}\ar@{->}[dr]^{\eta'}&\\
		S\ar@{-->}[rr]^{\chi}\ar[dr]&& S'\ar[dl]\\
		& {\rm pt} &}
\end{equation}
where $\eta,\eta'$ are birational morphisms, namely the blow-ups of closed points of degrees $d$ and $d'$, and $S',T$ are del Pezzo surfaces as well. Moreover, one of the following holds:
\begin{enumerate}
	\item $S\simeq S'$, $d=5$, $\chi$ is a birational Bertini involution;
	\item $S\simeq S'$, $d=4$, $\chi$ is a birational Geiser involution;
	\item $d=3$, $K_{S'}^2=6$;
	\item $d=2$, $K_{S'}^2=6$;
	\item $d=1$, $K_{S'}^2=8$, and $S'_{\overline{\kk}}\simeq
	\PP^1_{\overline{\kk}}\times\PP^1_{\overline{\kk}}$.
\end{enumerate}

\begin{prop}\label{prop: structure of maps for dP6}
	Let $S$ be a sextic del Pezzo surface with $\Pic(S)\simeq\ZZ$. Suppose that $S$ is not rational over $\kk$, and let $\phi\colon S\dashrightarrow S'$ be a birational map to another Mori fibre space $S'$. Then $S'$ is a sextic del Pezzo surface with $\Pic(S')\simeq\ZZ$, and one has the following:
	\begin{enumerate}[leftmargin=*, labelindent=20pt, itemsep=5pt]
		\item If $\indexx(S)=6$, then $\phi$ is an isomorphism.
		\item If $\indexx(S)=2$, then $\phi$ is a composition of isomorphisms and Sarkisov links of type II at points of degrees 2 or 4 (Geiser birational involutions).
		\item If $\indexx(S)=3$, then $\phi$ is a composition of isomorphisms and Sarkisov links of type II at points of degree 3.
	\end{enumerate}
\end{prop}
\begin{proof}
	Follows from the classification of links and Propositions \ref{prop: index 2}, \ref{prop: index 3}, and \ref{prop: index 6}.
\end{proof}

Since Geiser birational involutions always lead to an isomorphic surface, our next goal is to understand how the Sarkisov links at points of degrees 2 and 3 change the Severi-Brauer data of del Pezzo surfaces. If $\EE$ is a splitting field of the centre $p$ of our Sarkisov link \eqref{eq: Sarkisov link of type II}, then all $(-1)$-curves on $T$ are clearly defined over the composite field $\EE\FF$. However, the splitting field $\FF'$ of the new hexagon $\Sigma'\subset S'$ may be a proper subfield of $\EE\FF$. 

\begin{rem}
	In particular, as we will see below, a Sarkisov link can change the Galois action on the hexagon of $(-1)$-curves. In geometric setting, i.e. for complex $G$-minimal sextic del Pezzo surfaces acted on by a finite group $G\subset\Aut(S)$, this phenomena was already noticed in \cite[Example 4.5]{YasinskyRigidity}.
\end{rem}

So, let $\chi\colon S\dashrightarrow S'$ be a Sarkisov link \eqref{eq: Sarkisov link of type II} between sextic del Pezzo surfaces at a point $p\in S$ of degree $\deg p\in\{2,3\}$. Let $\EE$ be the splitting field of $p$, let $\Sigma$ and $\Sigma'$ be the hexagons of $S$ and $S'$, and $\FF$ and $\FF'$ be their splitting fields. Then we have an exact sequence of groups
\begin{equation}\label{eq: new Galois action}
	\begin{tikzcd}
		1
		\ar{r}
		& 
		H
		\ar{r}
		& 
		\Gal(\EE\FF/\kk)
		\ar{r}{\Psi}
		& 
		\Aut(\Sigma')
		& 
	\end{tikzcd}
\end{equation}
and $\FF'=(\EE\FF)^H$. Our next goal is to identify the subgroups $H$ in each case. Furthermore, we will see how the fields $\KK$ and $\LL$, associated to a sextic del Pezzo surface, change under Sarkisov links.

\begin{rem}
	The statements of Propositions \ref{prop: 2-points new Splitting Field} and \ref{prop: 3-points new Splitting Field} about new fields $\KK'$ and $\LL'$ can be also extracted from the proofs of \cite[Claim 5.8]{LinShinderZimmermann} and \cite[Proposition 9.8]{AuelBernadara}. Our approach here is more naïve and is based on studying the Galois actions on $(-1)$-curves, and does not use either the motivic invariants of birational maps or mutations in derived categories
\end{rem}

\subsection{Points of degree 2}\label{subsec: Sarkisov at 2-points} In this case, the surface $T$ in \eqref{eq: Sarkisov link of type II} is a del Pezzo surface of degree 4. Recall that $T_{\overline{\kk}}$ is a blow-up of $\PP_{\overline{\kk}}^2$ in 5 points $x_1,\ldots,x_5$ in general position. Denote by $e_1,\ldots,e_5$ the exceptional divisors of this blow-up, $\ell_{ij}$, $i,j\in\{1,\ldots,5\}$, $i<j$, the strict transforms of the lines through $x_i$ and $x_j$, and by $\mathcal{c}$ the strict transform of the conic through $x_1,\ldots,x_5$. These are sixteen $(-1)$-curves on $T$. Their intersection graph is the Clebsch strongly regular quintic graph on 16 vertices, shown on Figure \ref{figure: Clebsch}. We now consider the blow-up $\eta\colon T\to S$ and make the following identifications:

\begin{itemize}[leftmargin=*, labelindent=20pt, itemsep=5pt]
	\item The strict transforms of the sides $E_1,E_2,E_3,F_1,F_2,F_3$ of $\Sigma$ are denoted the same on $T$; in the blow-up of $\PP_{\overline{\kk}}^2$ model they correspond to $e_1,e_2,e_3,\ell_{23},\ell_{13},\ell_{12}$, respectively, where we identify $S_{\overline{\kk}}$ with $\PP_{\overline{\kk}}^2$ blown up in $x_1$, $x_2$ and $x_3$.
	\item The exceptional divisors over the blown up points are denoted $E_4$ and $E_5$; they correspond to $e_4$ and $e_5$, respectively.
	\item The strict transforms of the two $1$-curves in $S$ passing through the $2$-point and $E_1,E_2,E_3$ and $F_1,F_2,F_3$, respectively. They correspond to $\mathcal{c}$ and $\ell_{45}$ and are denoted $C$ and $L_{45}$ on $T$.
	\item The strict transforms of the six $0$-curves on $S$ that pass through one component of the $2$-point and opposite lines in the hexagon; these correspond to $\ell_{*4}$ and $\ell_{*5}$ for $*\in\{1,2,3\}$ and are denoted $L_{*4}$ and $L_{*5}$ on $T$.
\end{itemize}
		
\begin{figure}[!htbp]
	\centering
			\begin{tikzpicture}
				\def\scale{1.7} 
				
				\draw (18:\scale cm) -- (162:\scale cm) -- (306:\scale cm) -- (90:\scale cm) -- (234:\scale cm) -- cycle;
				\draw (18:3*\scale cm) -- (90:3*\scale cm) -- (162:3*\scale cm) -- (234:3*\scale cm) -- (306:3*\scale cm) -- cycle;
				\draw (18:\scale cm) -- (54:2*\scale cm) --(90:\scale cm) -- (126:2*\scale cm) -- (162:\scale cm) -- (198:2*\scale cm) -- (234:\scale cm) -- (270:2*\scale cm) -- (306:\scale cm) -- (342:2*\scale cm) -- cycle;
				\draw (18:3*\scale cm) -- (126:2*\scale cm) -- (234:3*\scale cm)-- (342:2*\scale cm) -- (90:3*\scale cm) -- (198:2*\scale cm) -- (306:3*\scale cm) -- (54:2*\scale cm) -- (162:3*\scale cm) -- (270:2*\scale cm) -- cycle;
				
				\foreach \x in {18,90,162,234,306}
				{
					\draw (\x:\scale cm) -- (\x:3*\scale cm);
				}
				
				\foreach \x in {54, 126, 198, 270, 342}
				{
					\draw (\x:0cm) -- (\x:2*\scale cm);
				}
				
				\filldraw[black,fill=mygreen] (18:3*\scale cm) circle (12pt) node {$F_3$};
				\filldraw[black,fill=mygreen] (90:3*\scale cm) circle (12pt) node {$E_1$};
				\filldraw[black,fill=mygreen] (162:3*\scale cm) circle (12pt) node {$F_2$};
				\filldraw[black,fill=white] (234:3*\scale cm) circle (12pt) node {$L_{24}$};
				\filldraw[black,fill=white] (306:3*\scale cm) circle (12pt) node {$L_{35}$};
				
				\filldraw[black,fill=white] (18:\scale cm) circle (12pt) node {$L_{34}$};
				\filldraw[black,fill=mypink] (90:\scale cm) circle (12pt) node {$C$};
				\filldraw[black,fill=white] (162:\scale cm) circle (12pt) node {$L_{25}$};
				\filldraw[black,fill=myblue] (234:\scale cm) circle (12pt) node {$E_4$};
				\filldraw[black,fill=myblue] (306:\scale cm) circle (12pt) node {$E_5$};
				
				\filldraw[black,fill=mygreen] (54:2*\scale cm) circle (12pt) node {$E_3$};
				\filldraw[black,fill=mygreen] (126:2*\scale cm) circle (12pt) node {$E_2$};
				\filldraw[black,fill=white] (198:2*\scale cm) circle (12pt) node {$L_{14}$};
				\filldraw[black,fill=mypink] (270:2*\scale cm) circle (12pt) node {$L_{45}$};
				\filldraw[black,fill=white] (342:2*\scale cm) circle (12pt) node {$L_{15}$};
				
				\filldraw[black,fill=mygreen] (342:0cm) circle (12pt) node {$F_1$};
				
			\end{tikzpicture}
		\caption{The Clebsch graph of 16 exceptional curves. Two vertices are connected by an edge if and only if corresponding curves intersect. Blue vertices are contracted by $\eta$, pink vertices are contracted by $\eta'$. The green vertices are mapped to the hexagon on $S$ by $\eta$, the white vertices are mapped to the hexagon on $S'$ by $\eta'$.}
		\label{figure: Clebsch}
\end{figure}

\begin{cor}
	Consider the Sarkisov link \eqref{eq: Sarkisov link of type II bis} at a point of degree 2. The morphism $\eta$  contracts $E_4$ and $E_5$. The morphism $\eta'$ contracts $C$ and $L_{45}$. The hexagon $\Sigma'$ consists of (the $\eta'$-images of) the curves $L_{14}$, $L_{15}$, $L_{24}$, $L_{25}$, $L_{34}$, $L_{35}$, whose configuration is shown on Figure \ref{pic: new hexagon 1}. 
\end{cor}

\begin{figure}
	\centering
	\begin{subfigure}{0.45\textwidth}
		\centering
			\begin{tikzpicture}
			\newdimen\R
			\R=3.2cm
			\newdimen\A
			\A=1.9cm
			\newdimen\B
			\B=2.8cm
			\newdimen\C
			\C=2.45cm
			\draw (0:\R) \foreach \x in {60,120,...,360} {  -- (\x:\R) };
			
			\foreach \x/\l/\p in
			{ 60/{}/above,
				120/{}/above,
				180/{}/left,
				240/{}/below,
				300/{}/below,
				360/{}/right
			}
			\node[inner sep=1pt,circle,draw,fill,label={\p:\l}] at (\x:\R) {};
			
			\foreach \start/\end/\label in
			{ 0/60/$L_{35}\ \ $,
				60/120/$L_{14}$,
				120/180/$\ \ L_{25}$,
				180/240/$L_{34}\ \ $,
				300/360/$\ \ L_{24}$
			}
			\draw (\start:\R) -- (\end:\R) node[midway, above, fill=white] {\label};
			
			\draw (240:\R) -- (300:\R) node[midway, below, fill=white] {$L_{15}$};
			
			
			\draw[<->, blue] 
			($(60:\B)!0.5!(120:\B)$) -- ($(240:\B)!0.5!(300:\B)$)
			node[pos=0.1, left, fill=white] {\( f \)};  
			
			\draw[<->, blue] 
			($(120:\B)!0.5!(180:\B)$) -- ($(180:\B)!0.5!(240:\B)$);  
			
			\draw[<->, blue] 
			($(0:\B)!0.5!(60:\B)$) -- ($(300:\B)!0.5!(360:\B)$);  
			
			
			\draw[<->, darkgreen] 
			($(60:\C)!0.5!(120:\C)$) -- ($(240:\C)!0.5!(300:\C)$)
			node[pos=0.6, below, fill=white] {\( h \)};  
			
			\draw[<->, darkgreen] 
			($(120:\C)!0.5!(180:\C)$) --
			($(300:\C)!0.5!(360:\C)$); 
			
			\draw[<->, darkgreen] 
			($(0:\C)!0.5!(60:\C)$) --  ($(180:\C)!0.5!(240:\C)$);  
			
			\draw[->, red] 
			($(120:\A)!0.5!(180:\A)$) -- ($(0:\A)!0.5!(60:\A)$)
			node[midway, above,fill=white] {\( g \)};  
			
			\draw[->, red] 
			($(0:\A)!0.5!(60:\A)$) -- ($(240:\A)!0.5!(300:\A)$);  
			
			\draw[->, red] ($(240:\A)!0.5!(300:\A)$) -- ($(120:\A)!0.5!(180:\A)$); 
			
		\end{tikzpicture}
		\caption{$f$ permutes the components of $p$}
		\label{pic: new hexagon 1}
	\end{subfigure}
	\hfill
	\begin{subfigure}{0.45\textwidth}
		\centering
			\begin{tikzpicture}
			\newdimen\R
			\R=3.2cm
			\newdimen\A
			\A=1.9cm
			\newdimen\B
			\B=2.8cm
			\newdimen\C
			\C=2.45cm
			\draw (0:\R) \foreach \x in {60,120,...,360} {  -- (\x:\R) };
			
			\foreach \x/\l/\p in
			{ 60/{}/above,
				120/{}/above,
				180/{}/left,
				240/{}/below,
				300/{}/below,
				360/{}/right
			}
			\node[inner sep=1pt,circle,draw,fill,label={\p:\l}] at (\x:\R) {};
			
			\foreach \start/\end/\label in
			{ 0/60/$L_{35}\ \ $,
				60/120/$L_{14}$,
				120/180/$\ \ L_{25}$,
				180/240/$L_{34}\ \ $,
				300/360/$\ \ L_{24}$
			}
			\draw (\start:\R) -- (\end:\R) node[midway, above, fill=white] {\label};
			
			\draw (240:\R) -- (300:\R) node[midway, below, fill=white] {$L_{15}$};
			
			
			\draw[<->, blue] 
			($(120:\B)!0.5!(180:\B)$) -- 	($(0:\B)!0.5!(60:\B)$)
			node[pos=0.1, left, fill=white] {\( f \)};  
			
			\draw[<->, blue] 
			($(180:\B)!0.5!(240:\B)$) -- ($(300:\B)!0.5!(360:\B)$);  

			
			\draw[<->, darkgreen] 
			($(60:\C)!0.5!(120:\C)$) -- ($(240:\C)!0.5!(300:\C)$)
			node[pos=0.6, below, fill=white] {\( h \)};  
			
			\draw[<->, darkgreen] 
			($(120:\C)!0.5!(180:\C)$) --
			($(300:\C)!0.5!(360:\C)$); 
			
			\draw[<->, darkgreen] 
			($(0:\C)!0.5!(60:\C)$) --  ($(180:\C)!0.5!(240:\C)$);  
			
			\draw[->, red] 
			($(120:\A)!0.5!(180:\A)$) -- ($(0:\A)!0.5!(60:\A)$)
			node[midway, above,fill=white] {\( g \)};  
			
			\draw[->, red] 
			($(0:\A)!0.5!(60:\A)$) -- ($(240:\A)!0.5!(300:\A)$);  
			
			\draw[->, red] ($(240:\A)!0.5!(300:\A)$) -- ($(120:\A)!0.5!(180:\A)$); 
			
		\end{tikzpicture}
		\caption{$f$ preserves the components of $p$}
		\label{pic: new hexagon 2}
	\end{subfigure}
	\caption{Possible actions in the case $\EE\subset\FF$, $\Gal(\FF/\kk)\simeq\Dih_6$.}
	\label{pic:twohexagons}
\end{figure}

\begin{prop}\label{prop: 2-points new Splitting Field}
	Let $S$ be a sextic $G$-del Pezzo surface over $\kk$ and $\SplittingHex$ be the splitting field of its hexagon. Let $p$ be a point of degree 2 on $S$ with splitting field $\EE$. Let $\chi\colon S \dashrightarrow S'$ be a Sarkisov link based at $p$. Then $S'$ is a sextic del Pezzo surface over $\kk$ and one has $\SplittingHex'=(\SplittingHex \EE)^H$ for the splitting field of its hexagon, where $H \subset\Gal(\SplittingHex \EE / \kk)$ is a subgroup and one of the following cases holds.
	\begin{enumerate}[leftmargin=*, labelindent=20pt, itemsep=5pt]
	\item If $\EE\subset\SplittingHex$ and case~(2) below does not occur,	then $H=\{\id\}$ and $\SplittingHex'=\SplittingHex$. In particular,
	$\Gal(\SplittingHex'/\kk)=\Gal(\SplittingHex/\kk)$.
	\item If $G\simeq\Dih_6$ and $\EE=\FF^{\langle g,h\rangle}$ (which implies that $S$ is $\kk$-rational), then $H=\langle h\rangle$ and $\SplittingHex'=\LL$.
	In particular, $\Gal(\SplittingHex'/\kk)\simeq\Sym_3$.
	\item If $\EE\cap\SplittingHex=\kk$ and $\Gal(\FF/\kk)\not\simeq\Sym_3$, then $H=\langle h\rangle$. In particular, $\Gal(\SplittingHex'/\kk)\simeq (\Gal(\SplittingHex/\kk)/H)\times\Gal(\EE/\kk)$.
	\item If $\EE\cap\SplittingHex=\kk$ and $\Gal( \FF/\kk) \simeq \Sym_3$ (which implies that $S$ is $\kk$-rational), then $H=\{\id\}$. In particular, $\Gal(\SplittingHex'/\kk)\simeq \Gal(\SplittingHex/\kk)\times\Gal(\EE/\kk)$.
	\end{enumerate}	
	Furthermore, one has $\KK'=\EE$ and $\LL'=\LL$.
\end{prop}

\begin{rem}\label{rem: indexx 2 D6 new embedding}
In the case where $\EE = \SplittingHex^{\langle g, f \rangle}$ we have $\SplittingHex' = \SplittingHex$, but the embedding changes (compare to Remark \ref{rem: EmbedUniq}) and therefore also the surface changes.
\end{rem}

\begin{proof}
We distinguish between the cases $\EE\subset\SplittingHex$
and $\EE\cap\SplittingHex=\kk$.

\begin{description}[leftmargin=*, labelindent=20pt, itemsep=5pt]
	\item[\it Case $\EE\subset\SplittingHex$] Then $\EE\SplittingHex=\SplittingHex$ and the generators of $\Gal(\EE\SplittingHex/\kk)=\Gal(\SplittingHex/\kk)$ belong to the set $\{f,g,h\}$. We now recover their actions on the hexagon $\Sigma'$ of $S'$.
	\begin{itemize}[leftmargin=*, labelindent=15pt, itemsep=5pt]
		\item First, $g$ acts trivially on the components of $p$, hence fixes $E_4$ and $E_5$ on $T$. The action on $\Sigma'$ can be recovered either from intersection indices, or by using the identifications made in the beginning of this subsection. Namely, since $g$ acts as $E_1\mapsto E_2\mapsto E_3\mapsto E_1$, we find that $g(L_{14})$ corresponds to the curve passing through $p_4$, $E_2$ and $F_2$, i.e. $L_{24}$. We see that $g$ acts by the clock-wise rotation by $2\pi/3$ on $\Sigma'$.
		
		\item  If $\Gal(\FF/\kk) = \langle g,h\rangle\simeq\ZZ/6$ then the action of $h$ on the components of $p$, hence on $E_4$ and $E_5$, is non-trivial (since $p$ is not $\kk$-rational). As above, we find that $h$ exchanges the opposite sides of $\Sigma'$ and conclude that all non-trivial elements of the Galois group act non-trivially on the new hexagon $\Sigma'$, hence $H=\{\id\}$ and $\SplittingHex$ is the splitting field of $\Sigma'$.

		If $\Gal(\FF/\kk)=\langle g,f\rangle\simeq\Sym_3$, then we have, as for $\langle g,h \rangle$, that $f$ permutes the components of $p$. Further, $f$ permutes $\{E_2,F_2\}, \{E_3,F_3\}$ and fixes $\{E_1,F_1\}$. Therefore it acts as $L_{15}\leftrightarrow L_{14}$, $L_{34}\leftrightarrow L_{25}$, $L_{24}\leftrightarrow L_{35}$. Thus again we find that $H = \{\id\}$ and $\FF'=\FF$. See Figure \ref{pic: new hexagon 1} for the action on $\Sigma'$.
		
		Now, if $\Gal(\SplittingHex/\kk)=\langle g,f,h\rangle\simeq\Dih_6$, then we notice that the action of $f$ on $\Sigma'$ is nontrivial, regardless of whether $f$ permutes the components of $p$ or not. Namely, either it acts as $L_{15}\leftrightarrow L_{14}$, $L_{34}\leftrightarrow L_{25}$, $L_{24}\leftrightarrow L_{35}$ (when the components of $p$ are permuted by $f$, i.e. $\EE=\FF^{\langle g,u\rangle}, \, u \in \{s,h\}$), or as $L_{25}\leftrightarrow L_{35}$, $L_{34}\leftrightarrow L_{24}$ (when the components of $p$ are not permuted by $f$, i.e. $\EE=\FF^{\langle g,f\rangle}$). On the other hand, if $h$ does not permute the components of $p$, then it acts trivially on $\Sigma'$ and hence $S'$ is a $\Sym_3$-del Pezzo surface, which is only possible in the rational case. If $h$ permutes the components of $p$, then it acts by central
		symmetry on $\Sigma'$; see Figures~\ref{pic: new hexagon 1}
		and~\ref{pic: new hexagon 2}.
		In this case, the new action $\Psi$ in
		\eqref{eq: new Galois action} is faithful.
	\end{itemize}	 
	
	\medskip
	
	If $\Gal(\SplittingHex/\kk)=\langle g,h\rangle\simeq\ZZ/6$, then $\EE=\FF^g$. Since the action of $g$ and $h$ on $\Sigma'$ is the same as on $\Sigma$, we conclude that $\KK'={\SplittingHex'}^g=\SplittingHex^g$. If $\Gal(\FF/\kk)\simeq\Sym_3$, then
	$\EE=\FF^g=\KK$ and $\LL'=\FF'=\FF=\LL$. If $\Gal(\SplittingHex/\kk)=\langle g,h,f\rangle\simeq\Dih_6$, then $\EE=\SplittingHex^{\langle g,u\rangle}, \, u \in \{ f,s,h \}$. In the case $u = f$, $g$ and $f$ preserve the components of $p$ and hence we are in the situation of Figure \ref{pic: new hexagon 2}. The contraction to Severi-Brauer surfaces is thus defined over $\FF^{\langle g,f\rangle}$, i.e. $\KK'=\EE$. In the case $u = s$, $g$ preserves the components of $p$, while $f=s\circ h$ does not, hence we are in the situation of Figure \ref{pic: new hexagon 1}. The contraction to Severi-Brauer surfaces is thus defined over $\FF^{\langle g,s\rangle}$, i.e. once again $\KK'=\EE$. For $G\simeq\ZZ/6\ZZ$, and in the cases $u=f,s$ just considered,	one has $\LL'=\FF^h=\LL$. Lastly, in the case $u=h$, we have $\FF'=\FF^h$ and
	$\Gal(\FF'/\kk)\simeq\Sym_3$. Thus
	$\LL'=\FF'=\FF^h=\LL$. Further $\KK' = \FF'^{g} = \FF^{\langle g,h \rangle} = \EE$. The action of $g,f$ on $\Sigma'$ is as in Figure \ref{pic: new hexagon 1}.
	
	\vspace{0.3cm}
	
	\item[\it Case $\EE\cap\SplittingHex=\kk$] Then $\Gal(\EE\SplittingHex/\kk)\simeq\Gal(\EE/\kk)\times\Gal(\SplittingHex/\kk)$. Let $t$ be the generator of $\Gal(\EE/\kk)$. Note that $g$ and, whenever they belong to $G$, $h$ and $f$	act trivially on the components of $p$, while the automorphism $t$ permutes the components of $p$ and preserves the curves $E_1,E_2,E_3,F_1,F_2,F_3$. Therefore, $g$ and, when present, $f$ act on $\Sigma'$
	as shown in Figure~\ref{pic: new hexagon 2};
	$h$, when present, acts trivially on $\Sigma'$,
	and $t$ switches its opposite sides. Thus $H=\langle h\rangle$ or $H = \{ \id \}$ if $h$ does not exist, and we have either $\KK'=(\SplittingHex\EE)^{\langle g,h\rangle}=\EE$ (if $\Gal(\SplittingHex/\kk)\simeq\ZZ/6\ZZ$), $\KK'=(\SplittingHex\EE)^{\langle g,h,f\rangle}=\EE$ (if $\Gal(\SplittingHex/\kk)\simeq\Dih_6$), or $\KK' = (\SplittingHex\EE)^{\langle g,f\rangle}= \EE$ (if $\Gal(\FF/ \kk)  \simeq \Sym_3)$. Finally, in the first two cases $\LL'=(\EE\SplittingHex)^{\langle h,t\rangle}=\SplittingHex^h=\LL$ and in the last case $\LL' =( \FF')^{t} = \FF$.
\end{description}
\end{proof}

\begin{cor}\label{cor: 2 linkSBData}
Let $S$ be a $G$-del Pezzo surface of degree 6 and index $2$ with Severi-Brauer data $\{\SplittingHex,\ \KK,\ \PP^2_{\KK},\ \LL,\ Y\}$, and $\chi\colon S \dashrightarrow S'$ be a Sarkisov link of type II based at a $2$-point with splitting field~$\EE$. Then the Severi-Brauer data of $S'$ is given by $\{\SplittingHex',\ \EE,\ \PP^2_{\EE},\ \LL,\ Y\}$. Moreover the surface $S'$ is uniquely determined by $S$ and $\EE$.
\end{cor}

\subsection{Points of degree 3}\label{subsec: Sarkisov at 3-points} In this case, the surface $T$ in \eqref{eq: Sarkisov link of type II} is a del Pezzo surface of degree 3. Then $T_{\overline{\kk}}$ is a blow-up of $\PP_{\overline{\kk}}^2$ in 6 points $x_1,\ldots,x_6$ in general position. Denote by $e_1,\ldots,e_6$ the exceptional divisors of this blow-up, by $\ell_{ij}$, $i,j\in\{1,\ldots,6\}$, $i<j$, the 15 strict transforms of the lines through $x_i$ and $x_j$, by $\mathcal{c}_i$ the strict transforms of the conics through all points except $x_i$. These are twenty-seven $(-1)$-curves on $T$. Their intersection graph can be visualized as the complement of the Schl\"afli strongly regular graph shown on Figure \ref{graph: Schläfli}. We now consider the blow-up $\eta\colon T\to S$ of the point $p=\{p_4,p_5,p_6\}$ and make identifications similar to the previous case:

\begin{itemize}[leftmargin=*, labelindent=20pt, itemsep=5pt]
	\item The strict transforms of the sides $E_1,E_2,E_3,F_1,F_2,F_3$ of $\Sigma$ are denoted the same on $T$; they correspond to $e_1,e_2,e_3,\ell_{23},\ell_{13},\ell_{12}$, respectively.
	\item The exceptional divisors over $p_4$, $p_5$ and $p_6$ are denoted $E_4$, $E_5$ and $E_6$; they correspond to $e_4$, $e_5$ and $e_6$ in the blow-up of the plane model.
	\item The strict transforms of:
	\begin{itemize}[leftmargin=*, labelindent=5pt, itemsep=5pt]
		\item[---] three $1$-curves on $S$ passing through all $E_1,E_2,E_3$ and a pair\footnote{Here and below, it means that the curve simply intersects both members of the pair.} $\{p_4,p_5\}$, $\{p_4,p_6\}$ or $\{p_5,p_6\}$ correspond to $\mathcal{c}_6$, $\mathcal{c}_5$ and $\mathcal{c}_4$, respectively. They will be denoted $C_6$, $C_5$ and $C_4$ on $T$.
		\item[---] three $1$-curves on $S$ passing through all $F_1,F_2,F_3$ and a pair $\{p_4,p_5\}$, $\{p_4,p_6\}$ or $\{p_5,p_6\}$ correspond to the lines $\ell_{45}$, $\ell_{46}$ and $\ell_{56}$. They will be denoted $L_{45}$, $L_{46}$ and $L_{56}$.   
		\item[---] three $2$-curves on $S$ passing through $p$ and a pair $\{E_1,E_2\}$, $\{E_1,E_3\}$ or $\{E_2,E_3\}$ correspond to the conics $\mathcal{c}_3$, $\mathcal{c}_2$ and $\mathcal{c}_1$. They will be denoted $C_3$, $C_2$ and $C_1$.    
	\end{itemize}
	\item The strict transforms of the nine $0$-curves on $S$ that pass through one point among $p_4$, $p_5$ and $p_6$, and opposite lines in the hexagon, i.e. $\{E_1,F_1\}$, $\{E_2,F_2\}$ or $\{E_3,F_3\}$; these correspond to $\ell_{i4}$, $\ell_{i5}$ and $\ell_{i6}$ for $i\in\{1,2,3\}$ and will be denoted $L_{i4}$, $L_{i5}$, $L_{i6}$ on $T$.
\end{itemize}

\begin{cor} \label{cor: 3LinkHexagon}
	Consider the Sarkisov link \eqref{eq: Sarkisov link of type II bis} at a point of degree 3. The morphism $\eta$  contracts $E_4$, $E_5$ and $E_6$. The morphism $\eta'$ contracts $C_1$, $C_2$ and $C_3$. The hexagon $\Sigma'$ consists of (the $\eta'$-images of) the curves $L_{45}$, $L_{46}$, $L_{56}$, $C_4$, $C_5$, $C_6$, whose configuration is shown on Figure \ref{pic:two hexagons points of degree 3}.
\end{cor}

\begin{rem}
	The configuration of $(-1)$-curves on the cubic surface $T$, appearing in our link, is shown on Figure~\ref{graph: Schläfli}. To visualize this graph, we adapt Coxeter's approach \cite{Coxeter-polytope2-21}, where the initial double-six \cite[p. 457]{Coxeter-polytope2-21} is chosen as
	\[
	\begin{tabular}{cccccc}
		$E_1$ & $E_2$ & $E_3$ & $E_4$ & $E_5$  & $E_6$ \\
		$C_1$ & $C_2$ & $C_3$ & $C_4$ & $C_5$ & $C_6$ \\
	\end{tabular}
	\]
	Two vertices are adjacent in the Schläfli graph if and only if the corresponding pair of $(-1)$-curves on $T$ are skew. The adjacencies should be understood as follows. Consider a vertex $x$ on the outer circle and all the vertices along the corresponding axis in the graph. If there are 4 such vertices, then $x$ is connected to the two vertices on the inner circle that lie on this axis. If there are 3 such vertices (i.e., the axis passes through the center), then $x$ is connected to two (out of the three coinciding) vertices at the center. Now consider a vertex $y$ on the inner circle and all the vertices along the corresponding axis in the graph. If there are 4 such vertices, then $y$ is connected to the two vertices on the outer circle that lie on this axis. If there are 3 such vertices (i.e., the axis passes through the center), then $y$ is connected to exactly two vertices in the centre.
\end{rem}

\begin{figure}[!htbp]\centering
	\begin{tikzpicture}[scale=7]
		
		\foreach \i in {1,...,12}
		{
			\draw ({cos(30*\i+15)},{sin(30*\i+15)}) -- ({cos(30*(\i+1)+15)},{sin(30*(\i+1)+15)});
			\draw ({cos(30*\i+15)},{sin(30*\i+15)}) -- ({cos(30*(\i+2)+15)},{sin(30*(\i+2)+15)});
			\draw ({cos(30*\i+15)},{sin(30*\i+15)}) -- ({cos(30*(\i+3)+15)},{sin(30*(\i+3)+15)});
			\draw ({cos(30*\i+15)},{sin(30*\i+15)}) -- ({cos(30*(\i+4)+15)},{sin(30*(\i+4)+15)});
			\draw ({cos(30*\i+15)},{sin(30*\i+15)}) -- ({cos(30*(\i+5)+15)},{sin(30*(\i+5)+15)});
			\draw ({cos(30*\i+15)},{sin(30*\i+15)}) -- ({cos(30*(\i+6)+15)},{sin(30*(\i+6)+15)});
			\draw ({cos(30*\i)*sqrt(2-sqrt(3))},{sin(30*\i)*sqrt(2-sqrt(3))}) -- ({cos(30*(\i+2))*sqrt(2-sqrt(3))},{sin(30*(\i+2))*sqrt(2-sqrt(3))});
			\draw ({cos(30*\i)*sqrt(2-sqrt(3))},{sin(30*\i)*sqrt(2-sqrt(3))}) -- ({cos(30*(\i+3))*sqrt(2-sqrt(3))},{sin(30*(\i+3))*sqrt(2-sqrt(3))});
			\draw ({cos(30*\i)*sqrt(2-sqrt(3))},{sin(30*\i)*sqrt(2-sqrt(3))}) -- ({cos(30*(\i+4))*sqrt(2-sqrt(3))},{sin(30*(\i+4))*sqrt(2-sqrt(3))});
			\draw ({cos(30*\i)*sqrt(2-sqrt(3))},{sin(30*\i)*sqrt(2-sqrt(3))}) -- ({cos(30*(\i+5))*sqrt(2-sqrt(3))},{sin(30*(\i+5))*sqrt(2-sqrt(3))});
			\draw ({cos(30*\i)*sqrt(2-sqrt(3))},{sin(30*\i)*sqrt(2-sqrt(3))}) -- ({cos(30*(\i+6))*sqrt(2-sqrt(3))},{sin(30*(\i+6))*sqrt(2-sqrt(3))});
		}
		\draw[fill=myyellow] ({cos(30*0+15)},{sin(30*0+15)}) circle (2pt) node {$C_5$};
		\draw[fill=myyellow] ({cos(30*1+15)},{sin(30*1+15)}) circle (2pt) node {$C_6$};
		\draw[fill=mygreen] ({cos(30*2+15)},{sin(30*2+15)}) circle (2pt) node {$L_{23}$};
		\draw[fill=white] ({cos(30*3+15)},{sin(30*3+15)}) circle (2pt) node {$L_{34}$};
		\draw[fill=myblue] ({cos(30*4+15)},{sin(30*4+15)}) circle (2pt) node {$E_6$};
		\draw[fill=mygreen] ({cos(30*5+15)},{sin(30*5+15)}) circle (2pt) node {$E_1$};
		\draw[fill=mygreen] ({cos(30*6+15)},{sin(30*6+15)}) circle (2pt) node {$E_2$};
		\draw[fill=mygreen] ({cos(30*7+15)},{sin(30*7+15)}) circle (2pt) node {$E_3$};
		\draw[fill=myyellow] ({cos(30*8+15)},{sin(30*8+15)}) circle (2pt) node {$L_{56}$};
		\draw[fill=white] ({cos(30*9+15)},{sin(30*9+15)}) circle (2pt) node {$L_{16}$};
		\draw[fill=mypink] ({cos(30*10+15)},{sin(30*10+15)}) circle (2pt) node {$C_3$};
		\draw[fill=myyellow] ({cos(30*11+15)},{sin(30*11+15)}) circle (2pt) node {$C_4$};
		
		\draw[fill=mygreen] ({cos(30*0)*sqrt(2-sqrt(3))},{sin(30*0)*sqrt(2-sqrt(3))}) circle (2pt) node {$L_{12}$};
		\draw[fill=mypink] ({cos(30*1)*sqrt(2-sqrt(3))},{sin(30*1)*sqrt(2-sqrt(3))}) circle (2pt) node {$C_1$};
		\draw[fill=mygreen] ({cos(30*2)*sqrt(2-sqrt(3))},{sin(30*2)*sqrt(2-sqrt(3))}) circle (2pt) node {$L_{13}$};
		\draw[fill=white] ({cos(30*3)*sqrt(2-sqrt(3))},{sin(30*3)*sqrt(2-sqrt(3))}) circle (2pt) node {$L_{24}$};
		\draw[fill=white] ({cos(30*4)*sqrt(2-sqrt(3))},{sin(30*4)*sqrt(2-sqrt(3))}) circle (2pt) node {$L_{35}$};
		\draw[fill=myblue] ({cos(30*5)*sqrt(2-sqrt(3))},{sin(30*5)*sqrt(2-sqrt(3))}) circle (2pt) node {$E_5$};
		\draw[fill=myyellow] ({cos(30*6)*sqrt(2-sqrt(3))},{sin(30*6)*sqrt(2-sqrt(3))}) circle (2pt) node {$L_{45}$};
		\draw[fill=myblue] ({cos(30*7)*sqrt(2-sqrt(3))},{sin(30*7)*sqrt(2-sqrt(3))}) circle (2pt) node {$E_4$};
		\draw[fill=myyellow] ({cos(30*8)*sqrt(2-sqrt(3))},{sin(30*8)*sqrt(2-sqrt(3))}) circle (2pt) node {$L_{46}$};
		\draw[fill=white] ({cos(30*9)*sqrt(2-sqrt(3))},{sin(30*9)*sqrt(2-sqrt(3))}) circle (2pt) node {$L_{15}$};
		\draw[fill=white] ({cos(30*10)*sqrt(2-sqrt(3))},{sin(30*10)*sqrt(2-sqrt(3))}) circle (2pt) node {$L_{26}$};
		\draw[fill=mypink] ({cos(30*11)*sqrt(2-sqrt(3))},{sin(30*11)*sqrt(2-sqrt(3))}) circle (2pt) node {$C_2$};
		
		\draw[fill=white,thick] (0,0) circle (3pt) node {$L_{14,25,36}$};
		
	\end{tikzpicture}
	\caption{The Schl\"afli graph, which is the \emph{complement} of the intersection graph of $(-1)$-curves on a cubic surfaces: two vertices are adjacent in the Schläfli graph if and only if the corresponding pair of $(-1)$-curves are skew. Blue vertices are contracted by $\eta$, pink vertices are contracted by $\eta'$. The green vertices are mapped to the hexagon on $S$ by $\eta$, the yellow vertices are mapped to the hexagon on $S'$ by~ $\eta'$.}\label{graph: Schläfli}
\end{figure}

\begin{figure}
	\centering
	\begin{subfigure}{0.45\textwidth}
		\centering
		\begin{tikzpicture}
			\newdimen\R
			\R=3.2cm
			\newdimen\A
			\A=1.9cm
			\newdimen\B
			\B=2.8cm
			\newdimen\C
			\C=2.45cm
			\draw (0:\R) \foreach \x in {60,120,...,360} {  -- (\x:\R) };
			
			\foreach \x/\l/\p in
			{ 60/{}/above,
				120/{}/above,
				180/{}/left,
				240/{}/below,
				300/{}/below,
				360/{}/right
			}
			\node[inner sep=1pt,circle,draw,fill,label={\p:\l}] at (\x:\R) {};
			
			\foreach \start/\end/\label in
			{ 0/60/$L_{46}\ \ $,
			60/120/$C_4$,
			120/180/$\ \ L_{45}$,
			180/240/$C_5\ \ $,
			300/360/$\ \ C_6$
			}
			\draw (\start:\R) -- (\end:\R) node[midway, above, fill=white] {\label};
		
			\draw (240:\R) -- (300:\R) node[midway, below, fill=white] {$L_{56}$};
			
			
			\draw[<->, blue] 
			($(60:\B)!0.5!(120:\B)$) -- ($(240:\B)!0.5!(300:\B)$)
			node[pos=0.1, left, fill=white] {\( f \)};  
			
			\draw[<->, blue] 
			($(120:\B)!0.5!(180:\B)$) -- ($(180:\B)!0.5!(240:\B)$);  
			
			\draw[<->, blue] 
			($(0:\B)!0.5!(60:\B)$) -- ($(300:\B)!0.5!(360:\B)$);  
			
			
			\draw[<->, darkgreen] 
			($(60:\C)!0.5!(120:\C)$) -- ($(240:\C)!0.5!(300:\C)$)
			node[pos=0.6, below, fill=white] {\( h \)};  
			
			\draw[<->, darkgreen] 
			($(120:\C)!0.5!(180:\C)$) --
			($(300:\C)!0.5!(360:\C)$); 
			
			\draw[<->, darkgreen] 
			($(0:\C)!0.5!(60:\C)$) --  ($(180:\C)!0.5!(240:\C)$);  
			
			\draw[<-, red] 
			($(120:\A)!0.5!(180:\A)$) -- ($(0:\A)!0.5!(60:\A)$)
			node[midway, above,fill=white] {\( g \)};  
			
			\draw[<-, red] 
			($(0:\A)!0.5!(60:\A)$) -- ($(240:\A)!0.5!(300:\A)$);  
			
			\draw[<-, red] ($(240:\A)!0.5!(300:\A)$) -- ($(120:\A)!0.5!(180:\A)$); 
			
		\end{tikzpicture}
		\caption{Case $\EE\subset\FF$}
		\label{pic: d=3 new hexagon 1}
	\end{subfigure}
	\hfill
	\begin{subfigure}{0.45\textwidth}
		\centering
		\begin{tikzpicture}
			\newdimen\R
			\R=3.2cm
			\newdimen\A
			\A=1.9cm
			\newdimen\B
			\B=2.8cm
			\newdimen\C
			\C=2.45cm
			\draw (0:\R) \foreach \x in {60,120,...,360} {  -- (\x:\R) };
			
			\foreach \x/\l/\p in
			{ 60/{}/above,
				120/{}/above,
				180/{}/left,
				240/{}/below,
				300/{}/below,
				360/{}/right
			}
			\node[inner sep=1pt,circle,draw,fill,label={\p:\l}] at (\x:\R) {};
			
			\foreach \start/\end/\label in
			{ 0/60/$L_{46}\ \ $,
				60/120/$C_4$,
				120/180/$\ \ L_{45}$,
				180/240/$C_5\ \ $,
				300/360/$\ \ C_6$
			}
			\draw (\start:\R) -- (\end:\R) node[midway, above, fill=white] {\label};
			
			\draw (240:\R) -- (300:\R) node[midway, below, fill=white] {$L_{56}$};
			
			
			\draw[<->, blue] 
			($(120:\B)!0.5!(180:\B)$) -- 	($(0:\B)!0.5!(60:\B)$)
			node[pos=0.2, left, fill=white] {\( t \)};  
			
			\draw[<->, blue] 
			($(180:\B)!0.5!(240:\B)$) -- ($(300:\B)!0.5!(360:\B)$);  
			
			
			\draw[<->, darkgreen] 
			($(60:\C)!0.5!(120:\C)$) -- ($(240:\C)!0.5!(300:\C)$)
			node[pos=0.6, below, fill=white] {\( h,f \)};  
			
			\draw[<->, darkgreen] 
			($(120:\C)!0.5!(180:\C)$) --
			($(300:\C)!0.5!(360:\C)$); 
			
			\draw[<->, darkgreen] 
			($(0:\C)!0.5!(60:\C)$) --  ($(180:\C)!0.5!(240:\C)$);  
			
			\draw[<-, red] 
			($(120:\A)!0.5!(180:\A)$) -- ($(0:\A)!0.5!(60:\A)$)
			node[midway, above,fill=white] {\( w \)};  
			
			\draw[<-, red] 
			($(0:\A)!0.5!(60:\A)$) -- ($(240:\A)!0.5!(300:\A)$);  
			
			\draw[<-, red] ($(240:\A)!0.5!(300:\A)$) -- ($(120:\A)!0.5!(180:\A)$); 

		\end{tikzpicture}
		\caption{Case $\EE\cap\FF=\kk$}
		\label{pic: d=3 new hexagon 2}
	\end{subfigure}
	\caption{Some $\Gal(\FF\EE/\kk)$-actions on $\Sigma'$ for $d=3$}
	\label{pic:two hexagons points of degree 3}
\end{figure}

\begin{prop}\label{prop: 3-points new Splitting Field}
	Let $S$ be a  sextic $G$-del Pezzo surface over $\kk$ and $\SplittingHex$ be the splitting field of its hexagon. Let $p$ be a point of degree 3 on $S$ with splitting field $\EE$. Let $\chi\colon S \dashrightarrow S'$ be a Sarkisov link based at $p$. Then $S'$ is a sextic del Pezzo surface over $\kk$ and one has $\SplittingHex'=(\SplittingHex \EE)^H$ for the splitting field of its hexagon, where $H \subset\Gal(\SplittingHex \EE / \kk)$ is a subgroup and one of the following cases holds.
	\begin{enumerate}[leftmargin=*, labelindent=20pt, itemsep=5pt]
			\item $\EE\subset\FF$. Then $H = \{\id\}$, i.e. $\FF'=\FF$. 
			\item $[\EE\cap\FF:\kk]=2$. Then $H = \langle g \rangle$ if $G\simeq\ZZ/6\ZZ$ or $G\simeq\Sym_3$, and $H=\langle g\rangle$ or $H=\langle g,s\rangle$ if $G\simeq\Dih_6$ $($the exact answer depends on the extension $\EE\cap\FF/\kk$, see the proof for the details$)$.
			\item $\EE\cap\FF=\kk$. Then $H=\langle g\rangle$ if $G\simeq\ZZ/6\ZZ$ or $G\simeq\Sym_3$, and $H=\langle g,s\rangle$ if $G\simeq\Dih_6$.
		\end{enumerate}
		Furthermore, one has $\KK'=\KK$ and $\LL'=\EE$.
		If $\Gal(\FF'/\kk)\simeq\Sym_3$, then $\FF'=\LL'=\EE$.
\end{prop}

\begin{proof}
Since $\EE/\kk$ is of degree 3 or 6, with $\Gal(\EE/\kk)\simeq\ZZ/3\ZZ$ or $\Gal(\EE/\kk)\simeq\Sym_3$ respectively, one has
	 $\EE\cap\FF=\EE$, or
	 $\EE\cap\FF=\kk$, or
	 $\EE\cap\FF$ is quadratic over~$\kk$, or
	 $\EE\cap\FF\ne\EE$ is cubic over $\kk$. However this last case is not possible, as $\EE\cap\FF$ is then not Galois, which contradicts Remark~\ref{rem: Galois group of the composite}.  So, we deal with the three remaining cases.
	 
\vspace{0.3cm}

\begin{description}
	\item[\it Case $\EE\subset\SplittingHex$] Then $\FF\EE=\FF$ and the generators of $\Gal(\FF\EE/\kk)=\Gal(\SplittingHex/\kk)$ belong to the set $\{f,g,h\}$. If $\Gal(\FF/\kk)\simeq\ZZ/6\ZZ$ then $\EE=\FF^h$ and $\Gal(\EE/\kk)=\langle g\rangle$; if $\Gal(\FF/\kk)\simeq\Sym_3$ then $\EE=\FF$ and $\Gal(\EE/\kk)=\langle g,f\rangle$, as otherwise $\EE$ is not Galois over $\kk$; if $\Gal(\FF/\kk)\simeq\Dih_6$ then $\EE=\FF^{h}$ and $\Gal(\EE/\kk)=\langle g,f\rangle$. In what follows, we assume, without loss of generality, that $g$ acts by $p_4\mapsto p_5\mapsto p_6\mapsto p_4$, while $f$ fixes $p_4$ and swaps $p_5$ and $p_6$. It is then straightforward, by using the identifications made in beginning of Section \ref{subsec: Sarkisov at 3-points}, to recover the action of Galois elements on $\Sigma'$. For example, since $g$ preserves the set $\{E_1,E_2,E_3\}$ and permutes the points $p_4$, $p_5$, $p_6$ as we said above, one has $g(C_4)=C_5$.
	
	Suppose that $\Gal(\FF/\kk)\simeq\ZZ/6\ZZ$ or $\Gal(\FF/\kk)\simeq\Dih_6$, so $\EE=\FF^h$. Then $h$ fixes $p_4$, $p_5$, $p_6$. Since $h$ switches $E_i$ with $F_i$ for $i\in\{1,2,3\}$, we find that $h(C_4)=L_{56}$; similarly, $h(C_5)=L_{46}$ and $h(C_6)=L_{45}$. Similarly, by recalling the action of $f$ on $\Sigma$, we deduce that $f$ switches $C_5$ with $L_{45}$, $C_6$ with $L_{46}$ and $C_4$ with $L_{56}$. To sum up, the action of $g$, $h$ and $f$ on $\Sigma'$ is shown on Figure \ref{pic: d=3 new hexagon 1}. In particular, $H=\{\id\}$, and $\KK'=\KK$, $\LL'=\EE$ when $\Gal(\FF/\kk)\simeq\ZZ/6\ZZ$ or $\Gal(\FF/\kk)\simeq\Dih_6$, and $\KK'=\KK$ when $\Gal(\FF/\kk)\simeq\Sym_3$.		
	
	\vspace{0.3cm}
	
	\item[\it Case $\EE\cap\FF$ is quadratic over $\kk$] 
		
	Note that $\Gal(\EE/\kk)\simeq\Sym_3$ in this case. We use Remark \ref{rem: Galois group of the composite} with $\EE_1=\FF$, $\EE_2=\EE$ to calculate the Galois group $\Gal(\FF\EE/\kk)$. As in that Remark, we denote by $t$ and $w$ some elements of order $2$ and $3$, respectively, in $\Gal(\EE/\kk)$. 
	
	If $G\simeq\ZZ/6\ZZ$ or $G\simeq\Sym_3$, then $\EE\cap\FF=\SplittingHex^g=\EE^w$. If $G\simeq\Dih_6$, then $\EE\cap\FF= \SplittingHex^{\langle g,u\rangle}=\EE^w$ for $u \in \{ s,f,h \}$. Set $t'= h$ for $G\simeq\ZZ/6\ZZ$, set $t'=f$ for $G\simeq\Sym_3$, and choose the element $t'\in\Dih_6$ as in the Remark \ref{rem: Galois group of the composite} when $G\simeq\Dih_6$. In all three cases, observe that 
	\[
	t'|_{\EE \cap \SplittingHex} = t|_{\EE \cap \SplittingHex}
	\]	
	is the unique non-trivial element of $\Gal(\EE\cap\FF/\kk)$.	Thus $\Gal(\FF\EE/\kk)$ is generated by [the preimages under the isomorphism given in the Remark of] $(\id,w)$, $(t',t)$, $(g,\id)$, and in the case of $\Dih_6$ we also add $(u,\id)$. By a slight abuse of notation, we will denote those elements by $w,t,g,u$, respectively.

	Notice that $g$ and $u$ act trivially on the components of the point $p$ (since, by Remark \ref{rem: Galois group of the composite}, their restrictions on~$\EE$ give trivial automorphisms), while $w$ acts as a $3$-cycle on its geometric components and $t$ fixes one component and switches the other two components. 
		
	Since $g$ fixes all components of $p$ and does not exchange the two triplets $\{E_1,E_2,E_3\}$ and  $\{F_1,F_2,F_3\}$ in $\Sigma$, its action on $\Sigma'$ is trivial. The action of $w$ is always by rotation of order 3. The action of $t$ on $\Sigma'$ is by $t(C_4)=L_{56}$, $t(C_5)=L_{45}$, $t(C_6)=L_{46}$.  Thus, for $G\simeq\ZZ/6\ZZ$ and  $G\simeq\Sym_3$ we have $H = \langle g \rangle$ and therefore $\KK' = (\FF\EE)^{\langle g,w\rangle} = \FF^g = \KK$. In both cases, one has $\Gal(\SplittingHex'/\kk)\simeq\Sym_3$. Note that $\FF'=(\FF\EE)^{g}=\EE$.

	Now suppose $G\simeq\Dih_6$. If $u = (h,\id)$ or $u=(f, \id)$ then $u$ acts as central symmetry on $\Sigma'$ and thus $H = \langle g \rangle$. 
	We observe that $\Gal(\FF'/\kk)\simeq\Dih_6$ and therefore $\KK' = (\FF\EE)^{\langle g,tu,w\rangle} = \FF^{\langle g,s\rangle} = \KK$ and $\LL' = (\FF\EE)^{\langle g,u\rangle} = \EE$. 

	In the case $u = (s,\id)$ we see that $s$ fixes the components of $p$ and does not exchange the two triplets $\{E_1,E_2,E_3\}$ and $\{F_1,F_2,F_3\}$, hence it acts trivially on $\Sigma'$ and $H = \langle g,s \rangle$. Now we have $\Gal(\FF'/\kk)\simeq\Sym_3$ and therefore we have that $\KK'=(\FF\EE)^{\langle g,s,w\rangle } = \FF^{\langle g,s\rangle }=\KK$. Note that $\FF'=(\FF\EE)^{\langle g,s\rangle}=\EE$. To sum up, possible transitions between Galois groups are presented in Table ~\ref{table: d=3 Galois groups quadratic case}. 
	
	\renewcommand{\arraystretch}{1.5}	
	\begin{table}[h!]
		\centering
		\begin{tabular}{|c|c|c|c|}
			\hline
			\diagbox{$\Gal(\EE/\kk)$}{$\Gal(\FF/\kk)$} & $\ZZ/6\ZZ=\langle g,h\rangle$ & $\Sym_3=\langle g,f\rangle$ & $\Dih_6=\langle g,h,f\rangle$\\
			\hline
			$\Sym_3=\langle w,t\rangle$ & $ \Sym_3$ & $ \Sym_3$ & $ \Dih_6$ or $\Sym_3$ \\
			\hline
		\end{tabular}
		\caption{Each cell indicates a group isomorphic to $\Gal(\FF'/\kk)$.}
		\label{table: d=3 Galois groups quadratic case}
	\end{table}
	
	\vspace{0.3cm}
	
	\item[\it Case $\EE\cap\FF=\kk$] Then $\Gal(\EE\FF/\kk)\simeq\Gal(\EE/\kk)\times\Gal(\FF/\kk)$. As before, $w$ and $t$ are the generators of order $3$ and $2$ in $\Gal(\EE/\kk)$, or take just $w$ if $\Gal(\EE/\kk)\simeq\ZZ/3\ZZ$. We will assume that $w(p_4)=p_5$, $w(p_5)=p_6$, while $t$ fixes $p_4$ and swaps $p_5$ with $p_6$. Note that $g$, $h$ and $f$ act trivially on the components of $p$. So, $g(C_4)=C_4$ and we find that $g$ acts trivially on $\Sigma'$. Further, $h$ and $f$ both act by swapping the opposite sides of $\Sigma'$, so in particular $s=h\circ f$ acts trivially on $\Sigma'$. One has $w(C_4)=C_5$, $w(C_5)=C_6$, while $t(C_4)=C_4$, $t(L_{45})=L_{46}$, $t(C_5)=C_6$. To sum up, the actions are shown on Figure \ref{pic: d=3 new hexagon 2}, and the corresponding Galois groups are summarized in Table \ref{table: d=3 Galois groups}. 
	
	\renewcommand{\arraystretch}{1.5}	
	\begin{table}[h!]
		\centering
		\begin{tabular}{|c|c|c|c|}
			\hline
			\diagbox{$\Gal(\EE/\kk)$}{$\Gal(\FF/\kk)$} & $\ZZ/6\ZZ=\langle g,h\rangle$ & $\Sym_3=\langle g,f\rangle$ & $\Dih_6=\langle g,h,f\rangle$\\
			\hline
			$\ZZ/3\ZZ=\langle w\rangle$ & $ \ZZ/6\ZZ=\langle w,h\rangle$ & $ \ZZ/6\ZZ=\langle w,f\rangle$ & $ \ZZ/6\ZZ=\langle w,f\rangle$ \\
			\hline
			$\Sym_3=\langle w,t\rangle$ & $ \Dih_6=\langle w,t,h\rangle$ & $ \Dih_6=\langle w,t,f\rangle$ & $ \Dih_6=\langle w,t,f\rangle$ \\
			\hline
		\end{tabular}
		\caption{Each cell contains a new action of the Galois group on $\Sigma'$.}
		\label{table: d=3 Galois groups}
	\end{table}
\end{description}
In the last case, we have $H=\langle g\rangle$ if $\Gal(\FF/\kk)\simeq\ZZ/6\ZZ$ or $\Gal(\FF/\kk)\simeq\Sym_3$, and $H=\langle g,s\rangle$ if $\Gal(\FF/\kk)\simeq\Dih_6$. In each case, we see directly that $\KK'=\KK$ and $\LL'=\EE$.
\end{proof}

\begin{cor}\label{cor: 3 linkSBData}
Let $S$ be a $G$-del Pezzo surface of degree 6 and index $3$ with Severi-Brauer data $\{\SplittingHex,\ \KK,\ X,\ \LL,\ \PP^1_\LL \times \PP^1_\LL\}$ and $\chi\colon S \dashrightarrow S'$ be a Sarkisov link of type II based at a $3$-point with splitting field $\EE$. Then the Severi-Brauer data of $S'$ is given by $\{\SplittingHex',\ \KK,\ X,\ \EE,\ \PP^1_\EE \times \PP^1_\EE\}$. Moreover, the surface $S'$ is uniquely determined by $S$ and $\EE$.
\end{cor}

\subsection{Birational rigidity of sextic del Pezzo surfaces}

Now we will derive several corollaries from the description of Sarkisov links between sextic del Pezzo surfaces. At the end of the paragraph, we will prove Theorem \ref{thm: pliability}.

\vspace{0.3cm}

Given a sextic $G$-del Pezzo surface $S$ and an integer $d\in\{2,3\}$, put 
\[
\Fields_d(S)=\big \{\EE/\kk\colon \EE\ \text{is a splitting field of a $d$-point on $S$ in general position} \big \}.
\]
It turns out that these sets are birational invariants of a minimal sextic del Pezzo surface~$S$. 

\begin{prop}\label{prop: splitting fields are invariants}
If $S$ and $S'$ are two birational non-$\kk$-rational sextic del Pezzo surfaces of Picard rank 1, then $\Fields_d(S)=\Fields_d(S')$ for $d\in\{2,3\}$.
\end{prop}

\begin{proof}
It is enough to show that $\Fields_d(S) \subseteq \Fields_d(S')$. Furthermore, it is enough to show that this inclusion holds for $S$ and $S'$ related by a single Sarkisov link $\chi\colon S\dashrightarrow S'$ at a point of degree $d$ (recall that Geiser birational involutions, which may appear in the index 2 case, lead to an isomorphic surface). So, let $\EE \in \Fields_d(S)$ and let $p\in S$ be a $d$-point in general position with splitting field $\EE$. We consider $\indexx(S)=2$ and $\indexx(S)=3$ cases separately.

Assume that $d = 2$. If $p \not\in \Exc(\chi)$ then\footnote{Here and below $\Exc(\phi)$ denotes the \emph{exceptional set} of a birational map $\phi\colon X\dashrightarrow Y$ between surfaces $X$ and $Y$. Recall that there exists a unique maximal pair of open subsets $U\subset X$, $V\subset Y$ such that $\phi$ restricts as an isomorphism $U\iso V$. Then $\Exc(\phi)=X\setminus U$. By Zariski's theorem, an irreducible curve $C\subset X$ is exceptional if and only if $\phi(C)$ is a point. In particular, the exceptional set of the link \eqref{eq: Sarkisov link of type II} consists of $\eta$-images of curves contracted by $\eta'$.} $\chi$ is a local $\kk$-isomorphism at $p$ and $\chi(p) \in S'$ has $\EE$ as its splitting field. By Proposition~\ref{prop: position of points}, this $2$-point is in general position on $S'$, so
$\EE\in\Fields_2(S')$. Suppose that $p$ is contained in $\Exc(\chi) \setminus \Ind(\chi)$. In the notation of Section \ref{subsec: Sarkisov at 2-points}, one has $\Exc(\chi)=\eta(C)\cup\eta(L_{45})$ and $\Ind(\chi)=\eta(C)\cap\eta(L_{45})$. Then $\chi(p) \subset \Ind(\chi^{-1})$. Moreover, since $\chi(p)$ is $\Gal(\overline{\kk}/\kk)$-invariant, $\Ind(\chi^{-1})$ has two geometric components and $S'(\kk)=\varnothing$, we have  $\chi(p) = \Ind(\chi^{-1})$. Since $\chi$ is defined over $\kk$, the 2-point $\Ind(\chi^{-1})$ splits over $\EE$ and there is no smaller field over which it splits. Notice that $\Ind(\chi^{-1})$ is in general position on $S'$ by definition of a Sarkisov link, thus $\EE\in\Fields_2(S')$. It remains to consider the case $p=\Ind(\chi)$. By Proposition~ \ref{prop: 2-points new Splitting Field}, the field $\EE$ is exactly the one over which $\KK'$-contractions to Severi-Brauer surfaces are defined. By Proposition \ref{prop: 2-points}, there are $2$-points in general position that split over $\KK'=\EE$ on $S'$. 

Assume that $d=3$. If $p \not\in \Exc(\chi)$, then again $\chi$ is a local $\kk$-isomorphism at $p$ and $\chi(p) \in S'$ has $\EE$ as its splitting field. If $\EE\nsubseteq\FF'$, this point is in general position by Proposition~\ref{prop: position of points}. If $\EE\subseteq\FF'$, then $\EE=\LL'$, and
Proposition~\ref{prop: 3-points} provides a $3$-point
in general position on $S'$ with splitting field $\EE$.
Thus $\EE\in\Fields_3(S')$ in either case. Suppose that $p\in\Exc(\chi) \setminus \Ind(\chi)$. Notice that $\Gal(\EE/\kk)\simeq\ZZ/3\ZZ$ or $\Gal(\EE/\kk)\simeq\Sym_3$. As above we observe that $\chi(p) = \Ind(\chi^{-1})$ since  $S'$ has no points of degree~1 or 2, and therefore $\Ind(\chi^{-1})$ splits over $\EE$. Since in the $\Gal(\EE/\kk)\simeq\Sym_3$ case there is no Galois extention of degree $3$ in between $\kk \subset \EE$ we conclude that $\EE$ is the splitting field of the 3-point $\Ind(\chi^{-1})$, which is in general position on $S'$ by the definition of a Sarkisov link. It remains to consider the case $p=\Ind(\chi)$. But then we finish as in the case $d=2$. Indeed, either $\EE=\LL'$ is the field over which the contractions to involution surfaces are defined, or $\EE=\FF'$ (when $\Gal(\FF'/\kk)\simeq\Sym_3$). In both cases, Proposition~\ref{prop: 3-points} gives the existence of 3-points in general position on $S'$ with splitting field $\EE$.
\end{proof}

Next, we investigate birational rigidity of sextic $G$-del Pezzo surfaces (see Definition \ref{def: BR}).

\begin{thm}\label{thm: birational rigidity}
	Let $S$ be a $G$-del Pezzo surface of degree 6 over a perfect field $\kk$. If $\indexx(S)=1$, then $S$ is $\kk$-rational and hence is not birationally rigid. If $\indexx(S)=6$, then $S$ is birationally super-rigid. In the remaining two cases, the following holds.
	\begin{enumerate}[leftmargin=*, labelindent=20pt, itemsep=5pt]
		\item If $\indexx(S)=2$, then $S$ is birationally rigid if and only if $G\simeq\ZZ/6\ZZ$ and every 2-point on~$S$ splits over $\KK$.
		\item If $\indexx(S)=3$, then $S$ is birationally rigid if and only if either $G\in\{\ZZ/6\ZZ,\Dih_6\}$ and every $3$-point on $S$ splits over $\LL$, or $G\simeq\Sym_3$ and every $3$-point on $S$
		splits over $\FF$.
	\end{enumerate} 
\end{thm}
\begin{proof}
	The first two statements were proven in Section \ref{subsec: closed points general}. Note that $S$ is birationally rigid if and only if every Sarkisov link on $S$ leads to an isomorphic surface. Let $\chi\colon S\dashrightarrow S'$ be a Sarkisov link. We keep the usual notation for the Severi-Brauer data of these surfaces. By the results of Section~\ref{sec: biregular classification}, we need to find when the data of $S'$ is equivalent to the one of $S$.
	
	If $\indexx(S)=2$, then Geiser involutions lead to isomorphic
	targets, so it suffices to consider a $2$-link~$\chi$.
	Let $\EE$ be the splitting field of its centre.
	Proposition~\ref{prop: 2-points new Splitting Field}
	gives $\LL'=\LL$ and $\KK'=\EE$. By Corollary~\ref{cor: 2 linkSBData} and	Theorem~\ref{thm: Biregular}, we have $S'\simeq S$ if and only if $\EE=\KK$. Therefore $S$ is birationally rigid if and only if $G\simeq\ZZ/6\ZZ$ and every $2$-point on $S$ splits over $\KK$, because for $G\simeq\Dih_6$	we can always find $2$-points split over
	$\FF^{\langle g,f\rangle}\ne\KK$ by Proposition~\ref{prop: 2-points}. In the $\indexx(S)=3$ case, we conclude by Proposition~\ref{prop: 3-points new Splitting Field}.
\end{proof}

Finally, we can prove Theorem \ref{thm: pliability}.

\begin{proof}[Proof of Theorem \ref{thm: pliability}]
	Let $S$ be a solid del Pezzo surface; recall that we assume $\Pic(S)\simeq\ZZ$. If $K_S^2=9$, then $S(\kk)=\varnothing$, as otherwise the blow-up of a $\kk$-rational point maps $S$ to a conic bundle $\mathbb{F}_1\to\PP_\kk^1$. So, $S$ is a non-trivial Severi-Brauer surface and one has $\Pl(S)=2$ by Proposition \ref{prop: birationality of SB}. Suppose that $K_S^2=8$. If $S(\kk)\ne\varnothing$, then $S$ is $\kk$-rational and hence is not solid. If $S(\kk)=\varnothing$, then \cite[Theorem 1.6]{Trepalin2023} implies that $S$ is birationally rigid. If $K_S^2=7$ or $K_S^2=5$, then it is well known (see e.g. \cite{SwinnertonDyer}) that $S(\kk)\ne\varnothing$ and such surfaces are $\kk$-rational by Theorem \ref{Iskovskikh Criterion}, hence they are not solid. Let $K_S^2=4$. If $S(\kk)\ne\varnothing$, then the blow-up of a $\kk$-rational point leads to a cubic surface with a structure of a conic bundle, given by the residue conics passing through the exceptional divisor; thus $S$ is not solid. Otherwise, the classification of Sarkisov links \cite[Theorem 2.6]{Isk1996} shows that $S$ is birationally rigid. The same classification shows that del Pezzo surfaces of degree $\leqslant 3$ are birationally rigid. Now the result follows from Example \ref{example: main example}.
\end{proof}

\section{Quotients of groups of birational self-maps}\label{sec: quotients}

\subsection{The associated graph}

We will consider the following standard equivalence relation between Sarkisov links:

\begin{mydef}[{\cite[Definition 3.3.1]{BlancSchneiderYasinsky}}]\label{def:equiLinkDim2}
	Let $\chi\colon S_1\dashrightarrow S_2$, $\chi'\colon S_1'\dashrightarrow S_2'$ be two Sarkisov links between del Pezzo surfaces of Picard rank $1$. We say that $\chi, \chi'$ are \emph{equivalent} if there is a commutative diagram
	\begin{equation*}
		\xymatrix{
			S_1\ar[d]_{\alpha}\ar@{-->}[rr]^{\chi} && S_2\ar[d]^{\beta}\\
			S_1'\ar@{-->}[rr]^{\chi'} && S_2'
		}
	\end{equation*}
	for some isomorphisms $\alpha\colon S_1\iso S_1',\beta\colon S_2\iso S_2'$.
\end{mydef}

\begin{rem}[{\cite[Lemma 3.3.2]{BlancSchneiderYasinsky}}]\label{lem:equivalencLinksSurfaceBasePt}
	Let $\chi\colon S_1\dashrightarrow S_2$, $\chi'\colon S_1'\dashrightarrow S_2'$ be two Sarkisov links between del Pezzo surfaces of Picard rank $1$. The following conditions are equivalent:
	\begin{enumerate}[leftmargin=*, labelindent=20pt, itemsep=5pt]
		\item
		$\chi$ and $\chi'$ are equivalent;
		\item
		there exists an isomorphism $\alpha\colon S_1\iso S_1'$ that sends the base-point of $\chi$ onto the base-point of $\chi'$;
		\item
		there exists an isomorphism $\beta\colon S_2\iso S_2'$ that sends the base-point of $\chi^{-1}$ onto the base-point of $\chi'^{-1}$.
	\end{enumerate}
\end{rem}

\begin{mydef}
	Let $S$ be a sextic del Pezzo surface with $\rk\Pic(S)=1$ which is not $\kk$-rational. We define the directed graph $\Graph_S$ associated to $S$ as follows:
	\begin{itemize}[leftmargin=*, labelindent=20pt, itemsep=5pt]
		\item The vertices of $\Graph_S$ are all sextic del Pezzo surfaces of Picard rank one which are birational to $S$, considered up to $\kk$-isomorphism.
		\item The edges are Sarkisov links between two surfaces, considered up to the equivalence.
	\end{itemize}
\end{mydef}

\begin{prop}\label{prop: one link}
	Every two distinct vertices of $\Graph_S$ are connected
	by at least one edge.
\end{prop}

\begin{proof}
	Pick any two vertices in $\Graph_S$, and suppose they correspond to the surfaces $Z$ and $Z'$. Since $Z$ and $Z'$ are both birational to $S$, the Sarkisov program (Theorem \ref{thm: sarkisov}) implies that they are connected by a sequence of links 
	\begin{equation*}
		\begin{tikzcd}
		Z=S_0
		\ar[dashed]{r}{\chi_0}
		& 
		S_1
		\ar[dashed]{r}{\chi_1}
		& 
		S_2
		\ar[dashed]{r}{\chi_2}
		& 
		\ldots
		\ar[dashed]{r}{\chi_{r-1}}
		& 
		S_r
		\ar[dashed]{r}{\chi_{r}}
		& 
		S_{r+1}=Z',
		\end{tikzcd}
	\end{equation*}
	where each $\chi_i$ starts at a point $p_i\in S_{i}$ with splitting field $\EE_i$. We can assume that all $p_i$ are $d$-points with $d\in\{2,3\}$, since birational Geiser involutions lead to an isomorphic surface (recall that there are no points of degree 1 or 5, as we assume $S$ irrational over $\kk$). We may assume that $\indexx(S)\in\{2,3\}$, since otherwise $S$ is birationally super-rigid and $\Graph_S$ is a singleton. By
	Proposition~\ref{prop: splitting fields are invariants},
	there is a $d$-point $p\in Z$ in general position with
	splitting field $\EE_r$. Let
	$\chi\colon Z\dashrightarrow Z''$ be the corresponding link.
	
	If $d=2$, Corollary~\ref{cor: 2 linkSBData} shows that
	the links in the original sequence preserve the datum
	$(\LL,Y)$, while $Z'$ and $Z''$ both have quadratic
	field $\EE_r$ and trivial Severi--Brauer surface data.
	If $d=3$, Corollary~\ref{cor: 3 linkSBData} shows that
	the links preserve the datum $(\KK,X)$, while $Z'$ and
	$Z''$ both have $\LL'=\LL''=\EE_r$ and trivial
	involution-surface data.
	Thus $Z'$ and $Z''$ have equivalent Severi--Brauer data,
	so $Z'\simeq Z''$ by Theorem~\ref{thm: Biregular}.
\end{proof}

At this point, we are able to prove Theorem \ref{thm: Birational classification}.

\begin{proof}[Proof of Theorem \ref{thm: Birational classification}]
	By the results of Section \ref{subsec: closed points general}, the described cases cover all possibilities and correspond to indices 1, 2, 3, and 6, respectively. 
	In the case of index $1$, the statement is clear, since all $\kk$-rational surfaces are birational to each other. The birationality criterion in the index~$6$ case follows from Proposition \ref{prop: index 6} and Theorem \ref{thm: Biregular}. Now we consider surfaces of index 2 and 3. 		
	Suppose that $S$ and $S'$ are birational over $\kk$.
	If $S\simeq S'$, the required equalities follow immediately.
	Otherwise, Proposition~\ref{prop: one link} gives
	a Sarkisov link $S\dashrightarrow S'$. If $\indexx(S)=\indexx(S')=2$, then $S$ and $S'$ are $\KK$-trivial and $\KK'$-trivial, respectively, hence the Amitsur groups $\Am(S_\KK)$ and $\Am(S'_{\KK'})$ vanish. By Proposition \ref{prop: 2-points new Splitting Field}, one has $\LL'=\LL$. Since $S_\LL$ and $S_{\LL'}'$ are $\LL$-birational and Amitsur groups are birational invariants, we are done. The case $\indexx(S)=\indexx(S')=3$ follows similarly.
	
	\medskip
	
	Conversely, suppose that $\LL=\LL'$, $\Am(S_\LL)=\Am(S'_\LL)\ne0$, and $\Am(S_\KK)=0$ and $\Am(S'_{\KK'})=0$. Thus both surfaces have index~$2$. If $\KK'=\KK$, then $\FF'=\KK'\LL'=\KK\LL=\FF$,	and Theorem~\ref{thm: Biregular} gives $S\simeq S'$. Assume therefore that $\KK'\ne\KK$.
	
	We first show that $\rk\Pic(S_{\KK'})=1$.
	If $\FF\cap\KK'=\kk$, then the isomorphism
	\[
	\Gal(\FF\KK'/\KK')\xrightarrow{\sim}\Gal(\FF/\kk),
	\qquad \sigma\longmapsto\sigma|_{\FF}.
	\]
	identifies the Galois action on the hexagon after base change with the original action, so the Picard rank remains~$1$.
	Otherwise $\KK'\subset\FF$. Since $S$ has index~$2$, its Galois group is isomorphic to $\ZZ/6\ZZ$ or~$\Dih_6$.
	The cyclic case would force $\KK'=\KK$, so
	$G\simeq\Dih_6$.
	As $\KK'\ne\KK=\FF^{\langle g,s\rangle}$,
	the subgroup $\Gal(\FF/\KK')$ is either
	$\langle g,h\rangle$ or $\langle g,f\rangle$.
	Both act transitively on the hexagon, proving the claim.
	
	Let $X$ and $Y$ be Severi--Brauer data for $S$ over
	$\KK$ and $\LL$, respectively.
	The surface $X$ is split, since $\Am(S_\KK)=0$.
	Moreover, the equality $\Am(S_\LL)=\Am(S'_\LL)$ implies
	that the conic Brauer classes defining $Y$ vanish over
	$\LL\KK'=\LL'\KK'=\FF'$, since $S'_{\FF'}$ is split.
	Thus $Y_{\FF'}\simeq\PP^1_{\FF'}\times\PP^1_{\FF'}$.
	The Severi--Brauer data of $S_{\KK'}$ are therefore trivial, so $S_{\KK'}$ is $\KK'$-rational. Since $S(\kk)=\varnothing$, a $\KK'$-rational point gives
	a point of degree~$2$ on $S$ with splitting field $\KK'$.
	It is in general position by
	Proposition~\ref{prop: position of points}.
	The resulting Sarkisov link leads to $S'$
	by Corollary~\ref{cor: 2 linkSBData}
	and Theorem~\ref{thm: Biregular}.
	
	\medskip
	
	Finally, suppose that $\KK=\KK'$, $\Am(S_\KK)=\Am(S'_\KK)\ne0$, and
	$\Am(S_\LL)=0$ and $\Am(S'_{\LL'})=0$. Thus both surfaces have index~$3$. If $\LL'\subseteq\FF$, then $\LL'=\LL$:	for each of the possible groups
	$G\simeq\ZZ/6\ZZ$, $\Sym_3$, or $\Dih_6$,
	the field $\LL$ is the unique Galois subfield of $\FF$
	whose Galois group over $\kk$ is isomorphic to
	$\ZZ/3\ZZ$ or $\Sym_3$.
	Hence $\FF'=\KK\LL'=\FF$, and
	Theorem~\ref{thm: Biregular} gives $S\simeq S'$.
	
	Suppose therefore that $\LL'\nsubseteq\FF$.
	By Proposition~\ref{prop: 3-points}, there is a point
	$q'\in S'$ of degree~$3$ with splitting field $\LL'$.
	Let $\EE=\kk(q')$ be the residue field of $q'$.
	Then $[\EE:\kk]=3$, the normal closure of $\EE$ is $\LL'$, and
	the canonical morphism $\Spec\EE\to S'$
	gives an $\EE$-rational point of $S'$. Let $X$ and $X'$ be the Severi--Brauer surfaces associated with $S$ and $S'$, respectively, over $\KK=\KK'$. Since $S'(\EE)\ne\varnothing$, the surface $X'$ splits
	over $\KK\EE$.
	The equality $\Am(S_\KK)=\Am(S'_\KK)$ implies that
	$X$ also splits over~$\KK\EE$, so $S_{\KK\EE}$ is rational. A zero-cycle of degree~$3$ on $S$ remains of degree~$3$ after base change to $\EE$. Moreover, since $S_{\KK\EE}$ is rational, a $\KK\EE$-rational point gives, by pushforward, a zero-cycle of degree $[\KK\EE:\EE]=2$ on $S_\EE$. Thus $\indexx(S_\EE)$ divides both $3$ and $2$, and hence equals~$1$.	By Remark~\ref{rem: index 1 equivalent to rationality}, $S(\EE)\ne\varnothing$.
	
	Since $S(\kk)=\varnothing$, this gives a point of degree~$3$ on $S$ with residue field $\EE$ and splitting field $\LL'$. It is in general position by
	Proposition~\ref{prop: position of points}.
	The resulting Sarkisov link leads to $S'$ by Corollary~\ref{cor: 3 linkSBData} and Theorem~\ref{thm: Biregular}.
\end{proof}

\begin{Notation}
	Consider the graph $\Graph_S$, and let $(V,E)$ be the sets of its vertices and edges, respectively. Let $v_0$ be the vertex represented by $S$, put $S_{v_0}=S$. Choose for every vertex $v\in V\setminus\{v_0\}$ (i.e. an isomorphism class of a surface) in $\mathcal{G}_S$ a representative $S_v$, and put $\mathcal{R}_V=\{S_v\}_{v\in V}$. Further, choose for every $S_v\in \mathcal{R}_V \setminus \{S\}$ one Sarkisov link $\chi_v: S\dashrightarrow S_v$ and collect those in the set~$\mathcal R_E
	=\{\chi_v\colon S\dashrightarrow S_v\}_{v\in V\setminus\{v_0\}}$.	
	Put $\chi_{v_0}=\id_S$ (note that $\id_S\notin\mathcal R_E$).
\end{Notation}

In what follows, by a \emph{tour} in a graph we mean a closed walk: a sequence of adjacent edges that starts and ends at the same vertex, possibly repeating vertices and edges. A (generalized) \emph{cycle} for us will be a closed walk in which all vertices (and hence edges) are distinct, except for the first and last vertex, and we also allow loops, i.e. edges that connect a vertex to itself.

\begin{prop}
    Let $S$ be a sextic $G$-del Pezzo surface over a perfect field $\kk$. Then the group $\Bir_{\kk}(S)$ is generated by elements from the following list:
    \begin{enumerate}[leftmargin=*, labelindent=20pt, itemsep=5pt]
        \item $A_{\chi} = \chi_v^{-1}\circ\chi\circ\chi_u$ for any Sarkisov link $\chi\colon S_u\dashrightarrow S_v$ with $u,v\in V$ distinct vertices.        
        \item $B_{\chi}= \chi_v^{-1}\circ\chi$ for any link $\chi\colon S\dashrightarrow S_v$.
        \item $C_{\varphi} = \varphi$ for any self-link $\varphi\colon S\dashrightarrow S$, or an automorphism of $S$.
        \item $D_{\varphi} = \chi^{-1}_v\circ\varphi\circ\chi_v$ for any self-link $\varphi\colon S_v\dashrightarrow S_v$, or an automorphism of $S_v$.
    \end{enumerate}
    With the convention $\chi_{v_0}=\id_S$, these definitions give
    \[
    A_\chi=B_\chi
    \quad\text{if $\chi\colon S\dashrightarrow S_v$ and $v\ne v_0$,}
    \]
    \[
    A_\chi=B_{\chi^{-1}}^{-1}
    \quad\text{if $\chi\colon S_u\dashrightarrow S$ and $u\ne v_0$,}
    \]
    and $D_\varphi=C_\varphi$ when $\varphi$ is a self-link
    or an automorphism of $S$.
    
We will also call those elements the generating tours corresponding to $\chi$ (respectively, $\varphi$).

\end{prop}
\begin{proof}
    	Every birational self-map of $S$ corresponds, by the Sarkisov program, to a tour starting at $S$ in $\Graph_S$. For this proof, we also allow automorphisms as steps of a tour. Let us prove
    	that every such tour is a product of the stated generators and their inverses. 
    	We argue by induction on the length of a tour. The empty tour represents the identity, and the case of length~$1$ corresponds to a generator of type~$C$.
    	Let $S_{v_0}=S\dashrightarrow S_{v_1}\dashrightarrow\cdots
    	\dashrightarrow S_{v_{n+1}}=S$
    	be a tour of length $n+1$, with steps
    	$\chi_i\colon S_{v_i}\dashrightarrow S_{v_{i+1}}$,
    	and suppose that the claim holds for all tours of length at most~$n$.
    	
    	If $\chi_0$ is a self-link or an automorphism of $S$, then
    	\[
    	\chi_n\circ\cdots\circ\chi_0
    	=(\chi_n\circ\cdots\circ\chi_1)\circ C_{\chi_0},
    	\]
    	and the remaining tour has length~$n$, so the induction
    	hypothesis applies. We may therefore assume that $S_{v_1}\ne S$.
    	
    	Suppose that $\chi_0\notin\mathcal R_E$. Then
    	\[   
    		\chi_n\circ\cdots\circ\chi_0
    		=\chi_n\circ\cdots\circ\chi_1\circ\chi_{v_1}
    		\circ\chi_{v_1}^{-1}\circ\chi_0
    		=(\chi_n\circ\cdots\circ\chi_1\circ\chi_{v_1})
    		\circ B_{\chi_0}.
    	\]
    	Thus it suffices to consider the tour
    	$\chi_n\circ\cdots\circ\chi_1\circ\chi_{v_1}$.
    	We may therefore assume that $\chi_0=\chi_{v_1}$.  	
    	If $\chi_1$ is a self-link or an automorphism, then $v_1=v_2$, and we get
    	\[ 
    		\chi_n\circ\cdots\circ\chi_2\circ\chi_1\circ\chi_{v_1}
    		=\chi_n\circ\cdots\circ\chi_2\circ\chi_{v_1}
    		\circ\chi_{v_1}^{-1}\circ\chi_1\circ\chi_{v_1}
    		=(\chi_n\circ\cdots\circ\chi_2\circ\chi_{v_1})
    		\circ D_{\chi_1}.  
    	\]
    	The remaining tour has length~$n$, so the induction hypothesis applies. If $\chi_1$ leads back to $S$, then
    	\[
    	\chi_n\circ\cdots\circ\chi_0
    	=(\chi_n\circ\cdots\circ\chi_2)
    	\circ B_{\chi_1^{-1}}^{-1}.
    	\]
    	The remaining tour has length $n-1$ and is empty if $n=1$, so again the induction hypothesis applies.  	
    	Finally, suppose that $S_{v_2}$ is distinct from both $S$ and $S_{v_1}$. Then
    	\[
    		\chi_n\circ\cdots\circ\chi_2\circ\chi_1\circ\chi_0
    		=\chi_n\circ\cdots\circ\chi_2\circ\chi_{v_2}
    		\circ\chi_{v_2}^{-1}\circ\chi_1\circ\chi_{v_1}
    		=(\chi_n\circ\cdots\circ\chi_2\circ\chi_{v_2})
    		\circ A_{\chi_1}.
    	\]
    	The remaining tour has length~$n$, and the result follows	by induction.
\end{proof}

\begin{figure}[H]
	\centering
	\captionsetup[subfigure]{labelformat=empty} 
	
	\begin{subfigure}[t]{0.3\textwidth}
		\centering
		\begin{tikzpicture}[scale=2]
			\node[circle, draw, fill=white] (S) at (0,0) {$S$};
			
			\def\n{6} 
			\def\r{1} 
			
			\foreach \i in {1,...,\n} {
				\node[circle, draw, fill=white] (S\i) at ({\r*cos(360/7*\i)}, {\r*sin(360/7*\i)}) {$S_{v_\i}$};
				\draw[->, red] (S) to (S\i);
			}
			\node (D) at ({\r*cos(360/7*7)}, {\r*sin(360/7*7)}) {$\dots$};
			\draw[->, red] (S) to (D);
		\end{tikzpicture}
		\caption{The set $\mathcal{R}_E$}
	\end{subfigure}
	\hfill
	\begin{subfigure}[t]{0.16\textwidth}
		\centering
		\begin{tikzpicture}[scale=2, every node/.style={circle, draw, fill=white}]
			\node (1) at (0,0) {$S$};
			\node (2) at (0,1) {$S_u$};
			\node (3) at (1,0) {$S_v$};
			\draw[->, red] (1) to node[left, draw=none, fill=none, xshift=-3pt, yshift=0pt] {$\chi_u$} (2);
			\draw[->] (2) to node[right, draw=none, fill=none, xshift=-3pt, yshift=5pt] {$\chi$} (3);
			\draw[->, red] (3) to node[left, draw=none, fill=none, xshift=18pt, yshift=-9pt] {$\chi_v^{-1}$} (1);
		\end{tikzpicture}
		\caption{Type $A_\chi$}
	\end{subfigure}
	\hfill\hfill
	\begin{subfigure}[t]{0.17\textwidth}
		\centering
			\begin{tikzpicture}[scale=2, every node/.style={circle, draw, fill=white}]
			\node (1) at (0,0) {$S$};
			\node (2) at (0,1) {$S_v$};
			\draw[->, bend left = 30] (1) to node[left, draw=none, fill=none, xshift=3pt, yshift=1pt] {$\chi$} (2);
			\draw[->, red, bend left = 30] (2) to node[left, draw=none, fill=none, xshift=29pt, yshift=1pt] {$\chi_v^{-1}$} (1);
		\end{tikzpicture}
		\caption{Type $B_\chi$}
	\end{subfigure}
	\hfill
	\begin{subfigure}[t]{0.09\textwidth}
		\centering
		\begin{tikzpicture}[scale=2, every node/.style={circle, draw, fill=white}]
		\node (1) at (0,0) {$S$};
		\draw[->, loop above] (1) to node[left, draw=none, fill=none, xshift=12pt, yshift=10pt] {$\varphi$} (1);
		\end{tikzpicture}
		\caption{Type $C_\varphi$}
	\end{subfigure}
	\hfill
	\begin{subfigure}[t]{0.16\textwidth}
		\centering
		\begin{tikzpicture}[scale=2, every node/.style={circle, draw, fill=white}]
			\node (1) at (0,0) {$S$};
			\node (2) at (0,1) {$S_v$};
			\draw[->, red, bend left = 30] (1) to node[left, draw=none, fill=none, xshift=4pt, yshift=1pt] {$\chi_v$} (2);
			\draw[->, loop above] (2) to node[left, draw=none, fill=none, xshift=12pt, yshift=10pt] {$\varphi$} (2);
			\draw[->, red, bend left = 30] (2) to node[left, draw=none, fill=none, xshift=30pt, yshift=2pt] {$\chi_v^{-1}$} (1);
		\end{tikzpicture}
		\caption{Type $D_\varphi$}
	\end{subfigure}
  	\caption{Generators of $\Bir_{\kk}(S)$ from the perspective of $\mathcal{G}_S$.}
	\label{fig:generators}
\end{figure}

\subsection{Sextic del Pezzo surfaces of index 3}

Let $S$ be a sextic $G$-del Pezzo surface of index~3. If $\chi\colon S\dashrightarrow S'$ is a Sarkisov 3-link \eqref{eq: Sarkisov link of type II}, then $T$ is a cubic surface with $\rk\Pic(T)=2$. So, if $T'\to T\to S$ is a rank 3 fibration corresponding to some relation involving $\chi$, then $T'$ is a del Pezzo surface of degree 2 or 1, hence $T'\to T$ is a blow-up of a point of degree 1 or 2, which is impossible since $\indexx(T)=3$. We conclude that there are no elementary relations in this case.  

Since the composition of a Sarkisov link with an automorphism of Mori fibre space is a Sarkisov link, and the inverse of a Sarkisov link is a Sarkisov link, we obtain \emph{trivial relations} given by $\chi_2\circ \chi_1=\id$ and $\chi_2^{-1}\circ\alpha\circ \chi_1\circ \beta=\id$, where $\chi_1,\chi_2$ are Sarkisov links and $\alpha,\beta$ are automorphisms of Mori fibre spaces; we also include the relations between isomorphism of Mori fibre spaces in the set of trivial relations.

\begin{prop}\label{prop: index 3 relations}
    Let $S$ be a sextic $G$-del Pezzo of index $3$. Then every relation between the generators $A_{\chi}$, $B_{\chi}$, $C_\varphi$, $D_\varphi$ is a product of conjugates of the following relations.
    \begin{enumerate}[leftmargin=*, labelindent=20pt, itemsep=5pt]
        \item The relations between automorphisms given via a product of $C_\varphi$ or $D_\varphi$.
        \item $B_{\chi} = \id$ for each $\chi \in \mathcal{R}_E$. 
        \item The relations
        \begin{enumerate}[label=3\alph*)]
        	\item $A_{\chi_2}\circ A_{\chi_1}$
        	\item $B_{\chi_2^{-1}}^{-1}\circ B_{\chi_1}$
        	\item $C_\varphi\circ C_{\varphi^{-1}}$
        	\item $D_\varphi\circ D_{\varphi^{-1}}$
       	\end{enumerate} 
        for every two links $\chi_1$, $\chi_2$ satisfying $\chi_2\circ\chi_1=\id$, and for all self-links and automorphisms~$\varphi$. 
        \item For every two links $\chi_1$ and $\chi_2$, equivalent via automorphisms $\alpha$ and $\beta$, i.e. satisfying $\alpha\circ \chi_1\circ \beta=\chi_2$, the following relations.
        \begin{enumerate}[label=4\alph*)]
        	\item $A_{\chi_2^{-1}}\circ D_\alpha\circ A_{\chi_1}\circ D_\beta$
        	\item $B_{\chi_2}^{-1}\circ D_\alpha\circ B_{\chi_1}\circ C_\beta$ 
        	\item $C_{\chi_2^{-1}}\circ C_\alpha\circ C_{\chi_1}\circ C_\beta$
        	\item $D_{\chi_2^{-1}}\circ D_\alpha\circ D_{\chi_1}\circ D_\beta$.
        \end{enumerate}        
    \end{enumerate}
\end{prop}
\begin{proof}
	Let $U_n'\circ \ldots \circ U_1' = \id$ be a relation, where $U_i'$ are of the form $A_{\chi}$, $B_{\chi}$, $C_{\varphi}$, $D_{\varphi}$, or their inverse. By Theorem \ref{thm: sarkisov}, we can write this as a relation of Sarkisov links and automorphisms and hence as
	\[
	U_n'\circ \ldots \circ U_1'=\psi_k^{-1}\circ R_k\circ\psi_k \circ\ldots\circ\psi_1^{-1}\circ R_1\circ\psi_1=\id
	\]
	where $R_i$ are trivial relations and $\psi_i$ are some birational maps (possibly including identity maps) in $\BirMori(S)$, see \cite[Proposition 3.15]{LamyZimmermann}. Each $\psi_i^{-1}\circ R_i\circ \psi_i$ is a birational self-map of $S$ and therefore it is enough to prove that $\psi^{-1}\circ R\circ\psi$ is a product of conjugates of the relations from the statement for $R$ a trivial relation and $\psi$ a birational map starting from $S$. Note that $\psi$ can be written as $\psi=\chi\circ\psi'$, where $\psi'\in\Bir(S)$ and $\chi\colon S\dashrightarrow S_v$ is a (possibly empty) sequence of Sarkisov links. We then consider possible subcases.
	
	\begin{enumerate}[leftmargin=*, labelindent=20pt, itemsep=5pt]
		\item \emph{Case} $\chi=\varnothing$. Then $R=U_{i_1}\circ \ldots\circ U_{i_s}$ is a birational self-map of $S$.
		\begin{itemize}[leftmargin=*, labelindent=-5pt, itemsep=5pt]
			\vspace{0.3cm}
			\item Suppose that $R$ is a composition of automorphisms. Then it is a relation between generators of type $C$, i.e. a relation of type (1).
			\item Suppose that $R=\chi_2\circ\chi_1$. Then $R = B_{\chi_1}$ is of type (2) if $\chi_1 \in \mathcal{R}_E$, $R=C_{\chi_2}\circ C_{\chi_1}$ is of type (3c) if $\chi_1$ is a self-link or automorphism, and $R=B_{\chi_2^{-1}}^{-1}\circ B_{\chi_1}$ is of type (3b) if $\chi_1$ is not a self-link. 
			
			\item If $R={\chi_2}^{-1}\circ\alpha\circ\chi_1\circ\beta$, then $R=C_{\chi_2^{-1}}\circ C_\alpha\circ C_{\chi_1}\circ C_\beta$ is a relation of type (4c) when $\chi_1$ and $\chi_2$ are self-links. If they are not, then they have the form $\chi_1,\chi_2\colon S\dashrightarrow S_v$ for some $v\in V$. Thus
			\[
			{\chi_2}^{-1}\circ\alpha\circ\chi_1\circ\beta=
			\underbrace{
				\chi_2^{-1}\circ\chi_v}_{B_{\chi_2}^{-1}}\circ\underbrace{\chi_v^{-1}\circ\alpha\circ \chi_v}_{D_\alpha}\circ\underbrace{\chi_v^{-1}\circ\chi_1}_{B_{\chi_1}}\circ\underbrace{\beta}_{C_\beta}=B_{\chi_2}^{-1}\circ D_\alpha\circ B_{\chi_1}\circ C_\beta
			\]
			which is a relation of type (4b).	
		\end{itemize}
		
		\item \emph{Case} $\psi=\chi\circ\psi'$, where $\psi'\in\Bir(S)$ and $\chi\colon S\dashrightarrow S_v$ is some Sarkisov link. Then
		\[
		\psi^{-1}\circ R\circ\psi={\psi'}^{-1}\circ\chi^{-1}\circ R\circ\chi\circ\psi'
		\]
		and hence we need to show that $\chi^{-1}\circ R\circ\chi$ is a product of conjugates of relations from the statement. We again consider possible subcases.
		\begin{itemize}[leftmargin=*, labelindent=-5pt, itemsep=5pt]
		\vspace{0.3cm}
		\item Suppose that $R$ is a composition of automorphisms, say $R=\alpha_1\circ\ldots\circ\alpha_m$. Then 
		\[
		\chi^{-1}\circ R\circ\chi=B_{\chi}^{-1}\circ D_{\alpha_1}\circ\ldots\circ D_{\alpha_m}\circ B_{\chi}
		\]	
		it is a relation of type (1).
	 	\item Consider relations $R=\chi_2\circ\chi_1$, and note that $\chi_1$ starts at $S_v$. Suppose first that $\chi_1$ goes to $S$. Then 
		\[
		\chi^{-1}\circ\chi_2\circ\chi_1\circ\chi=B_{\chi}^{-1}\circ B_{\chi_2}\circ (B_{\chi_1^{-1}})^{-1}\circ B_\chi
		\]
		is a conjugate of type (3b) relation by $B_\chi$ (since $\chi_2\circ\chi_1=\id$ is equivalent to $\chi_1\circ\chi_2=\id$).
		Now suppose that $\chi_1\colon S_v\dashrightarrow S_w$. If $v=w$, then 
		\[
		\chi^{-1}\circ\chi_2\circ\chi_1\circ\chi= B_{\chi}^{-1}\circ D_{\chi_2}\circ D_{\chi_1}\circ B_\chi
		\]
		is a conjugate of type (3d) relation by $B_{\chi}$.
		If  $v\ne w$, then 
		\[
		\chi^{-1}\circ\chi_2\circ\chi_1\circ\chi= B_{\chi}^{-1}\circ A_{\chi_2}\circ A_{\chi_1}\circ B_\chi.
		\] 
		is a conjugate of type (3a) relation by $B_\chi$.
		
		\item Further, we consider relations $R=\chi_2^{-1}\circ\alpha\circ\chi_1\circ\beta$. 	
		Suppose first that $\chi_1\colon S_v\dashrightarrow S$, then $\chi_2\colon S_v\dashrightarrow S$. We have
		\[
		\chi^{-1}\circ\chi_2^{-1}\circ\alpha\circ\chi_1\circ\beta\circ\chi=B_\chi^{-1}\circ B_{\chi_2^{-1}}\circ C_\alpha\circ B_{\chi_1^{-1}}^{-1}\circ D_\beta\circ B_\chi,
		\]
		which belongs to the normal closure of the relations
		of type~(4b), because
		$\chi_2=\alpha\circ\chi_1\circ\beta$ is equivalent to
		$\chi_2^{-1}=\beta^{-1}\circ\chi_1^{-1}\circ\alpha^{-1}$.
		
		Then suppose that $\chi_1\colon S_v\dashrightarrow S_v$, $\chi_2\colon S_v\dashrightarrow S_v$. We obtain
		\[
		\chi^{-1}\circ\chi_2^{-1}\circ\alpha\circ\chi_1\circ\beta\circ\chi=B_\chi^{-1}\circ D_{\chi_2^{-1}}\circ D_\alpha\circ D_{\chi_1}\circ D_\beta\circ B_\chi,
		\]
		which is a conjugate of type (4d) relation by $B_\chi$.
		
		Finally suppose that $\chi_1\colon S_v\dashrightarrow\colon S_w$, $v\ne w$, then $\chi_2\colon S_v\dashrightarrow S_w$. We have
		\[
		\chi^{-1}\circ\chi_2^{-1}\circ\alpha\circ\chi_1\circ\beta\circ\chi=B_\chi^{-1}\circ A_{\chi_2^{-1}}\circ D_\alpha\circ A_{\chi_1}\circ D_\beta\circ B_\chi,
		\]
		which is a conjugate type (4a) relation by $B_\chi$.
		\end{itemize}
		
		\item Finally, in the general case when $\chi\colon S\dashrightarrow S_v$ is an arbitrary birational map, we can write
		\[
		\psi^{-1}\circ R\circ\psi=(\chi\circ\psi')^{-1}\circ R\circ (\chi\circ\psi') = (\chi_v^{-1}\circ\chi\circ\psi')^{-1}\circ \chi_v^{-1}\circ R\circ \chi_v\circ (\chi_v^{-1}\circ \chi\circ \psi'),
		\]
		and notice that $\chi_v^{-1}\circ\chi\circ\psi'\in\Bir(S)$, hence it is enough to prove the statement for $\chi_v^{-1}\circ R\circ \chi_v$. But this case was considered before. 	
	\end{enumerate}
\end{proof}

Recall that every edge $e\in E$ of $\Graph_S$ corresponds to the equivalence class of a Sarkisov link~$\chi$. Let us denote such edge $e=[\chi]$. Let $\bar{e}=[\chi^{-1}]$ be the edge corresponding to the equivalence class of $\chi^{-1}$. If $e = \bar{e}$, i.e. $\chi$ is equivalent to $\chi^{-1}$, then we call $\chi$ an \emph{almost involution}. Since the vertices of $\Graph_S$ represent different isomorphism classes of surfaces, amost involutions are self-loops (self-links) at the vertices of $\Graph_S$ (but a priori not every self-link is an almost involution). Let us label all the edges of $\Graph_S$, which are not almost involutions, with signs $+$ or $-$, so that for each $e\in E$ with $e\ne\bar{e}$ the edges $e$ and $\bar{e}$ have different signs. Given a Sarkisov link $\chi$ corresponding to an edge $e$ that is not an almost involution, we say that $\chi$ is \emph{positive} if $e$ has label $+$, and \emph{negative} if $e$ has label $-$. Finally, declare the elements of $\mathcal{R}_E$ positive. 

\begin{thm}\label{thm: index 3 structure of Bir}
	Let $S$ be a del Pezzo surface of degree 6 with $\rk\Pic(S)=1$ and $\indexx(S)=3$. Let~ $\mathcal{E}_1$~be the set of positive edges in $\mathcal{G}_S$ such that corresponding links are not equivalent to their inverses, and let $\mathcal{E}_2$ be the set of edges such that corresponding links are equivalent to their inverses (note that this is a subset of the set of self-loops at the vertices of $\Graph_S$). Identify $\mathcal R_E$
	with the set of edge classes $\{[\chi]\mid\chi\in\mathcal R_E\}$. Then the map
	\[
	\Psi\colon U_\chi\mapsto
	\begin{cases}
		1_{[\chi]},
		& \text{if $\chi$ is positive or an almost involution,
			and $[\chi]\notin\mathcal R_E$,}\\
		-1_{[\chi^{-1}]},
		& \text{if $\chi$ is negative and
			$[\chi^{-1}]\notin\mathcal R_E$,}\\
		\id,
		& \text{otherwise.}
	\end{cases}
	\]
defines a surjective non-trivial group homomorphism
\[ 
\Bir_{\kk}(S) \twoheadrightarrow \left (\bigast_{\mathcal{E}_1\setminus\mathcal{R}_E} \ZZ \right ) \ast \left (\bigast_{\mathcal{E}_2} \ZZ / 2 \ZZ \right ). 
\]
\end{thm}
\begin{proof}
	For brevity, let us call a $\mathcal{E}_1$-part and a $\mathcal{E}_2$-part the respective subgroups (free products) in the target of $\Psi$, defined in the statement. We first comment why the relations of Proposition \ref{prop: index 3 relations} are sent to identity. 
	
	The relations of type (1) and (2) are sent to $\id$ according to the third line in the definition of~ $\Psi$. The relations of type (3a) and (3b) involve only two links $\chi_1$ and $\chi_2$ which are inverse to each other (and are not equivalent to their inverses), hence $[\chi_1]$ and $[\chi_2]$ have different signs and the respective relations are sent to $1_{[\chi]}-1_{[\chi]}=0$ in the $\mathcal{E}_1$-part, where $\chi\in\{\chi_1,\chi_2\}$. If $\chi_1$ and $\chi_2$ are almost involutions in (3c) or (3d), then this relation is sent onto $1_{[\varphi]}+1_{[\varphi]}\equiv 0$ in $\mathcal{E}_2$-part, and  onto $1_{[\varphi]}-1_{[\varphi]}=0$ in $\mathcal{E}_1$-part if those are self-links which are not equivalent to their inverses.
	
	The relations of type (4) involve only equivalent links $\chi_1\sim\chi_2$. Hence, (4a) and (4b) are sent to the identity of the $\mathcal{E}_1$-part. If $\chi_1$ and $\chi_2$ are almost involutions in (4c) or (4d), then this relation is sent onto $1_{[\chi_1]}+1_{[\chi_2]}\equiv 0$ in $\mathcal{E}_2$-part, and  onto $1_{[\chi_1]}-1_{[\chi_1]}=0$ in $\mathcal{E}_1$-part if those are self-links which are not equivalent to their inverses.
		
	Next, we check that the target of~ $\Psi$ is not trivial. If this is not the case, then $\mathcal{E}_2=\varnothing$ and $\mathcal{E}_1=\mathcal{R}_E$. Due to Proposition \ref{prop: 3-points}, there is a self-link $\chi$ on $S$. It cannot be equivalent to its inverse, since $\mathcal{E}_2=\varnothing$. Hence, there are two edges $[\chi]$ and $[\chi^{-1}]$ of different signs in $E$. Choose a positive link. It is not in $\mathcal{R}_E$, as this set contains no self-links. Thus $\mathcal{E}_1\setminus\mathcal{R}_E\ne\varnothing$, a contradiction. The surjectivity of $\Psi$ follows from its definition.
\end{proof}

\begin{rem}\label{rem: number of models}
	The indexing set $\mathcal{E}_1 \setminus \mathcal{R}_E$ in Theorem \ref{thm: index 3 structure of Bir} is nonempty when $S$ admits at least two distinct birational models besides itself, i.e. the graph $\Graph_S$ has at least 3 vertices. Indeed, in that case there exist surfaces $S_u$ and $S_v$ with $u,v\in V$, $u\ne v$, distinct from $S$, and a positive edge between $u$ and $v$, which is not in $\mathcal{R}_E$.
\end{rem}

\subsection{Sextic del Pezzo surfaces of index 2}

We start with recalling what are elementary relations between Sarkisov 2-links in this case. 

\begin{prop}[{\cite[Theorem 4.5]{Isk1996}}]\label{prop: relations of 2-links}
Let $S$ be a sextic $G$-del Pezzo surface of index~$2$. Then any non-trivial elementary relation between Sarkisov links on $S$ is of the form 
\[
\chi_6\circ\chi_5\circ\chi_4\circ\chi_3\circ\chi_2\circ\chi_1 = \id
\]
where $\chi_1,\chi_4\colon S \dashrightarrow S_1$, $\chi_2,\chi_5\colon S_1 \dashrightarrow S_2$, $\chi_3,\chi_6\colon S_2 \dashrightarrow S$ are Sarkisov links based at $2$-points such that for all $i \in \mathbb{Z}/6\mathbb{Z}$ we have $\chi_i(\Ind(\chi_{i-1}^{-1})) = \Ind(\chi_{i+1})$. Moreover, one has $\chi_1\sim\chi_4$, $\chi_2\sim \chi_5$ and $\chi_3\sim\chi_6$.
\end{prop}
\begin{proof}
	This was already proven in \cite[Theorem 4.5]{Isk1996}, so let us briefly comment on possible strategy (based on the study of rank 3 fibrations) and on the equivalence of links. Let $\chi_n\circ\ldots\circ\chi_1=\id$ be an elementary relation, where $\chi_i\colon S_i\dashrightarrow S_{i+1}$ are Sarkisov links starting with a blow-up $Z_i\to S_i$ of a 2-point. We argue as in \cite[Lemma 3.3.4]{BlancSchneiderYasinsky} and analyse possible contractions from the surface $T$, obtained from $S$ by two consecutive blow-ups of 2-points $\{p_4,p_5\}$ and $\{p_6,p_7\}$. Here the rank 3 fibration dominating $\chi_i$ is a del Pezzo surface $T$ of degree 2. The set of fifty-six $(-1)$-curves on $T_{\overline{\kk}}$ consists of (the strict transforms of) exceptional divisors $E_i$ over the points $p_i\in\PP_{\overline{\kk}}^2$, $i\in\{1,\ldots 7\}$ (where, as usual, $S_1$ is identified with the blow-up of $p_1$, $p_2$, $p_3$ over $\overline{\kk}$), 21 lines $L_{ij}$ through $p_i$ and $p_j$, 21 conics $C_{ij}$ which pass through all $p_k$'s \emph{except} $p_i$ and $p_j$, and 7 cubics $F_i$ passing through $p_1,\ldots, p_7$ with
	multiplicity two at $p_i$, see \cite[Theorem 26.2]{ManinCubicForms}. The contractions inducing an elementary relation are shown on Figure \ref{fig: relation of 2-links}. Notice that $T$ is equipped with the canonical \emph{Geiser involution}, which is the desk involution of the double cover $|-K_T|\colon T\to\PP^2$. This involution belongs to the centre of $\Aut(T_{\overline{\kk}})$ and swaps $E_k$ with $F_k$, and $C_{ij}$ with $L_{ij}$, hence inducing the equivalence of links.	
\end{proof}

\begin{figure}
	\begin{center}
		
		\begin{tikzpicture}[scale=1.0]
			\def\R{5}     
			\def\r{2.5}   
			
			\draw[white] (0,0) circle (\R cm);
			
			\foreach \i in {1,...,6} {
				\pgfmathsetmacro\angle{60*(\i-1)}
				\coordinate (S\i) at ({\R*cos(\angle)}, {\R*sin(\angle)});
				
				\filldraw[white, fill=white] (S\i) circle (10pt);
				\node[centered] at ($(S\i)$) {$S_{\i}$};
			}
			
			
			\foreach \i in {1,...,6} {
				\pgfmathsetmacro\angle{60*(\i-1) + 30}
				\coordinate (Z\i) at ({\r*cos(\angle)}, {\r*sin(\angle)});
				
				\filldraw[white, fill=white] (Z\i) circle (10pt);
				\node[centered] at ($(Z\i)$) {$Z_{\i}$};
			}
			
			\filldraw[white, fill=white] (0,0) circle (10pt);
			\node[centered] at (0,0) {$T$};
			
			\draw[->, shorten <=10pt, shorten >=10pt] (Z1) -- (S1) node[midway, fill=white, inner sep=1pt] {$E_4,E_5$};
			\draw[->, shorten <=10pt, shorten >=10pt] (Z1) -- (S2) node[midway, fill=white, inner sep=1pt] {$L_{45},C_{67}$};
			\draw[->, shorten <=10pt, shorten >=10pt] (0,0) -- (Z1) node[midway, fill=white, inner sep=1pt] {$E_6,E_7$};; 
			
			\draw[->, shorten <=10pt, shorten >=10pt] (Z2) -- (S2) node[midway, fill=white, inner sep=1pt] {$E_6,E_7$};
			\draw[->, shorten <=10pt, shorten >=10pt] (Z2) -- (S3) node[midway, fill=white, inner sep=1pt] {$F_4,F_5$};
			\draw[->, shorten <=10pt, shorten >=10pt] (0,0) -- (Z2) node[midway, fill=white, inner sep=1pt] {$L_{45},C_{67}$};; 
			
			\draw[->, shorten <=10pt, shorten >=10pt] (Z3) -- (S3) node[midway, fill=white, inner sep=1pt] {$L_{45},C_{67}$};
			\draw[->, shorten <=10pt, shorten >=10pt] (Z3) -- (S4) node[midway, fill=white, inner sep=1pt] {$F_6,F_7$};
			\draw[->, shorten <=10pt, shorten >=10pt] (0,0) -- (Z3) node[midway, fill=white, inner sep=1pt] {$F_4,F_5$};; 
			
			\draw[->, shorten <=10pt, shorten >=10pt] (Z4) -- (S4) node[midway, fill=white, inner sep=1pt] {$F_4,F_5$};
			\draw[->, shorten <=10pt, shorten >=10pt] (Z4) -- (S5) node[midway, fill=white, inner sep=1pt] {$L_{67},C_{45}$};
			\draw[->, shorten <=10pt, shorten >=10pt] (0,0) -- (Z4) node[midway, fill=white, inner sep=1pt] {$F_6,F_7$};; 
			
			\draw[->, shorten <=10pt, shorten >=10pt] (Z5) -- (S5) node[midway, fill=white, inner sep=1pt] {$F_6,F_7$};
			\draw[->, shorten <=10pt, shorten >=10pt] (Z5) -- (S6) node[midway, fill=white, inner sep=1pt] {$E_4,E_5$};
			\draw[->, shorten <=10pt, shorten >=10pt] (0,0) -- (Z5) node[midway, fill=white, inner sep=1pt] {$L_{67},C_{45}$}; 
			
			\draw[->, shorten <=10pt, shorten >=10pt] (Z6) -- (S6) node[midway, fill=white, inner sep=1pt] {$L_{67},C_{45}$};
			\draw[->, shorten <=10pt, shorten >=10pt] (Z6) -- (S1) node[midway, fill=white, inner sep=1pt] {$E_6,E_7$};
			\draw[->, shorten <=10pt, shorten >=10pt] (0,0) -- (Z6) node[midway, fill=white, inner sep=1pt] {$E_4,E_5$}; 
			
			\draw[->, dashed] ([shift={(4:5cm)}]0,0) arc (4:56:5cm) node[midway, above right] {$\chi_1$};
			
			\draw[->, dashed] ([shift={(64:5cm)}]0,0) arc (64:116:5cm) node[midway, above right] {$\chi_2$};
			
			\draw[->, dashed] ([shift={(124:5cm)}]0,0) arc (124:176:5cm) 
			node[midway, above left] {$\chi_3$};
			
			\draw[->, dashed] ([shift={(184:5cm)}]0,0) arc (184:236:5cm) 
			node[midway, below left] {$\chi_4$};
			
			\draw[->, dashed] ([shift={(244:5cm)}]0,0) arc (244:296:5cm) 
			node[midway, below right] {$\chi_5$};
			
			\draw[->, dashed] ([shift={(304:\R cm)}]0,0) arc (304:356:\R cm) 
			node[midway, below right] {$\chi_6$};	
		\end{tikzpicture}
	\end{center}
	\caption{The relation $\chi_6\circ\chi_5\circ\chi_4\circ\chi_3\circ\chi_2\circ\chi_1 = \id$ between 2-links}
	\label{fig: relation of 2-links}
\end{figure}

We finally construct some non-trivial quotients of $\Bir(S)$, where $S$ is a sextic $G$-del Pezzo of index 2.

\begin{rem}
	In the course of the proof of Theorem \ref{thm: index 2 homomorphism}, we will find the system of relations between the generators $U$ in the index 2 case. Namely, these are relations of Proposition \ref{prop: index 3 relations}, the relations \eqref{eq: index 2 relations 1}-\eqref{eq: index 2 relations 3}, and the relations \eqref{eq: index 2 relations 4}-\eqref{eq: index 2 relations 6} sitting between $B_\chi^{-1}$ and $B_\chi$.
\end{rem}

\begin{thm}\label{thm: index 2 homomorphism}
	Let $S$ be a del Pezzo surface of degree~$6$ with
	$\rk\Pic(S)=1$ and $\indexx(S)=2$.
	Let~$\mathcal E$ be the set of positive edges and almost
	involutions in $\Graph_S$, and let
	$\mathcal I\subseteq\mathcal E$ be the set of equivalence
	classes of Sarkisov $4$-links in $\Graph_S$.
	We identify $\mathcal R_E$ with the corresponding set of
	edge classes.
	
	Then the map
	\[
	\Psi\colon U_\chi\longmapsto
	\begin{cases}
		1_{[\chi]},
		& \text{if $\chi$ is positive or an almost involution,
			and $[\chi]\notin\mathcal R_E$,}\\
		1_{[\chi^{-1}]},
		& \text{if $\chi$ is negative and
			$[\chi^{-1}]\notin\mathcal R_E$,}\\
		\id,
		& \text{otherwise}
	\end{cases}
	\]
	defines a surjective non-trivial group homomorphism
	\[
	\Bir_\kk(S)\twoheadrightarrow
	\left(
	\bigoplus_{\mathcal E\setminus(\mathcal R_E\cup\mathcal I)}
	\ZZ/2\ZZ
	\right)
	*
	\left(
	\bigast_{\mathcal I}\ZZ/2\ZZ
	\right).
	\]
	Here $1_e$ denotes the standard generator indexed by $e$
	in the direct sum if $e\notin\mathcal I$, and in the
	corresponding free-product factor if $e\in\mathcal I$.
\end{thm}

\begin{proof}
Geiser birational involutions do not appear in any non-trivial elementary relations. Now we find the relations between the generators $U$, which come from the relations of Proposition \ref{prop: relations of 2-links}. Let $R$ be the relation of length 6, given in Proposition \ref{prop: relations of 2-links}, which we view as a sequence of links
\begin{equation}\label{eq: six-relation sequence}
	\begin{tikzcd}
		Z
		\arrow[dashed]{r}{\chi_1}
		& 
		S_{v_1}
		\arrow[dashed]{r}{\chi_2}
		& 
		S_{v_2}
		\arrow[dashed]{r}{\chi_3}
		& 
		S_{v_3}
		\arrow[dashed]{r}{\chi_4}
		& 
		S_{v_4}
		\arrow[dashed]{r}{\chi_5}
		& 
		S_{v_5}
		\arrow[dashed]{r}{\chi_6}
		& 
		Z
	\end{tikzcd}
\end{equation}
We use the notation of the proof of Proposition \ref{prop: index 3 relations}, and again distinguish several cases.
\begin{enumerate}[leftmargin=*, labelindent=5pt, itemsep=5pt]
	\item Suppose that $\chi=\varnothing$, so $Z=S$ in the sequence \eqref{eq: six-relation sequence}. 
	\begin{itemize}[leftmargin=*, labelindent=-5pt, itemsep=5pt]
		\vspace{0.3cm}
		\item Suppose that all surfaces in \eqref{eq: six-relation sequence} are isomorphic, i.e. $v=v_1=\ldots=v_5$. Then
		\begin{equation}\label{eq: index 2 relations 1}
		R=C_{\chi_6}\circ\ldots\circ C_{\chi_1}.
		\end{equation} 
		\item Suppose that any two surfaces in \eqref{eq: six-relation sequence} related by a link are different. Then 
		\begin{equation}\label{eq: index 2 relations 2}
		R=B_{\chi_6^{-1}}^{-1}\circ A_{\chi_5}\circ B_{\chi_4}\circ B_{\chi_3^{-1}}^{-1}\circ A_{\chi_2}\circ B_{\chi_1}
		\end{equation} 
		\item Denote $Z=S=S_v$. If $\chi_1$ and $\chi_4$, or $\chi_2$ and $\chi_5$, or $\chi_3$ and $\chi_6$ are self-links, then we get that $R$ is equal, respectively, to		
		\begin{equation}\label{eq: index 2 relations 3} 
		\begin{aligned}
			& B_{{\chi_6}^{-1}}^{-1}\circ B_{\chi_5}\circ C_{\chi_4} \circ B_{\chi_3^{-1}}^{-1}\circ B_{\chi_2} \circ C_{\chi_1} \quad && \text{when}\ v=v_1=v_3=v_4,\ v_1\ne v_2,\ v_2\ne v_3, \\
			& B_{{\chi_6}^{-1}}^{-1}\circ D_{\chi_5}\circ B_{\chi_4}\circ B_{\chi_3^{-1}}^{-1}\circ D_{\chi_2}\circ B_{\chi_1} \quad && \text{when}\ v_1=v_2=v_4=v_5,\ v\ne v_1,\ v_2\ne v_3. \\
			& C_{\chi_6}\circ B_{{\chi_5}^{-1}}^{-1}\circ B_{\chi_4} \circ C_{\chi_3}\circ B_{{\chi_2}^{-1}}^{-1} \circ B_{\chi_1} \quad && \text{when}\ v=v_2=v_3=v_5,\ v\ne v_1,\ v_1\ne v_2.
		\end{aligned}
		\end{equation} 		
	\end{itemize}	
	
	\item Suppose that $\psi=\chi\circ\psi'$, where $\psi'\in\Bir(S)$ and $\chi\colon S\dashrightarrow S_v=Z$ is a birational map. As in the proof of Proposition \ref{prop: index 3 relations}, it is enough to prove the statement for $\chi$ a single Sarkisov link (the general case then can be reduced to this one, as in the part (3) of the aforementioned proof). 
	\begin{itemize}[leftmargin=*, labelindent=-5pt, itemsep=5pt]
		\item Suppose that all surfaces in \eqref{eq: six-relation sequence} are isomorphic. Then
		\begin{equation}\label{eq: index 2 relations 4} 
		\chi^{-1}\circ R\circ\chi=B_\chi^{-1}\circ D_{\chi_6}\circ\ldots\circ D_{\chi_1}\circ B_\chi.
		\end{equation}  
		\item Suppose that any two surfaces in \eqref{eq: six-relation sequence} related by a link are different. Then
		\begin{equation}\label{eq: index 2 relations 5} 
		\chi^{-1}\circ R\circ\chi=B_\chi^{-1}\circ A_{\chi_6}\circ\ldots\circ A_{\chi_1}\circ B_\chi.
		\end{equation}  
		\item Similarly, we get that $\chi^{-1}\circ R\circ\chi$ equals		
		\begin{equation}\label{eq: index 2 relations 6}
		\begin{aligned}
			B_\chi^{-1}\circ A_{\chi_6}\circ A_{\chi_5} \circ D_{\chi_4}\circ A_{\chi_3} \circ A_{\chi_2}\circ D_{\chi_1}\circ B_\chi \quad & \text{when}\ v=v_1=v_3=v_4,\ v_1\ne v_2,\ v_2\ne v_3, \\
			B_\chi^{-1}\circ A_{\chi_6}\circ D_{\chi_5} \circ A_{\chi_4}\circ A_{\chi_3} \circ D_{\chi_2}\circ A_{\chi_1}\circ B_\chi \quad & \text{when}\ v_1=v_2=v_4=v_5,\ v\ne v_1,\ v_2\ne v_3. \\
			B_\chi^{-1}\circ D_{\chi_6}\circ A_{\chi_5} \circ A_{\chi_4}\circ D_{\chi_3} \circ A_{\chi_2}\circ A_{\chi_1}\circ B_\chi \quad & \text{when}\ v=v_2=v_3=v_5,\ v\ne v_1,\ v_1\ne v_2.
		\end{aligned}
		\end{equation} 			
	\end{itemize}	
	Notice that if two links in the sequence $Z\dashrightarrow\ldots\dashrightarrow S_{v_3}$ are self-links, then all links in \eqref{eq: six-relation sequence} are. This gives all relations.
\end{enumerate}		
	
To check that $\Psi$ is well defined, we verify that all
relations are sent to the identity. For the relations
of Proposition~\ref{prop: index 3 relations}, this follows
as in Theorem~\ref{thm: index 3 structure of Bir}:
equivalent links have the same image, automorphisms have
trivial image, and the images of links have order dividing~$2$. Every non-trivial elementary relation consists of $2$-links,
with
$[\chi_1]=[\chi_4]$, $[\chi_2]=[\chi_5]$ and
$[\chi_3]=[\chi_6]$ by
Proposition~\ref{prop: relations of 2-links}.
Their images sum to zero, and the same holds for conjugates
of these relations. Surjectivity follows because every standard generator of the target is the image of a generating tour. Finally, Proposition~\ref{prop: 2-points} and Corollary~\ref{cor: 2 linkSBData} give a $2$-self-link
on $S$. Its class, or the class of its inverse,
belongs to $\mathcal E\setminus
(\mathcal R_E\cup\mathcal I)$, so the target is non-trivial.
\end{proof}

\def\bibindent{2.5em}

\bibliographystyle{alphadin}
\bibliography{biblio}

\end{document}